\documentclass{article}

\usepackage{xkeyval}
\usepackage{enumitem}
\usepackage{hyperref}
\usepackage{authblk}   
\usepackage{orcidlink} 

\newcommand{\fnm}[1]{#1}            
\newcommand{\sur}[1]{#1}            
\newcommand{\spfx}[1]{#1 }          
\newcommand{\orgdiv}[1]{#1, }       
\newcommand{\orgname}[1]{#1}        
\newcommand{\orgaddress}[1]{, #1}   
\newcommand{\postcode}[1]{#1 }
\newcommand{\city}[1]{#1, }
\newcommand{\country}[1]{#1}
\newcommand{\email}[1]{\thanks{Corresponding author: \texttt{#1}}}
\usepackage{mathrsfs}
\usepackage{mathtools}
\usepackage{amsmath}
\usepackage{amsthm}
\usepackage{amssymb}
\usepackage{bm}
\usepackage{dsfont}
\usepackage{cleveref}
\usepackage[margin=1.0in]{geometry} 
\usepackage{stackengine}
\usepackage{appendix}
\usepackage[most]{tcolorbox}
\usepackage{changepage}
\usepackage{mdframed}
\usepackage{graphicx} 
\usepackage{subcaption} %
\usepackage{stmaryrd} 

\usepackage{float}
\usepackage{natbib}

\usepackage{tikz-cd}
\usepackage{tikz-3dplot}
\usetikzlibrary{positioning,calc}

\usepackage{todonotes}

\usepackage{booktabs}   
\usepackage{multirow}   
\usepackage{pifont}     
\usepackage{array}      
\usepackage{colortbl}   
\usepackage{tabularx}   
\newcolumntype{Y}{>{\centering\arraybackslash}X}

\newcommand{\KL}[2]{\text{KL}_{#1}\left( #2 \right)} 
\DeclareMathOperator*{\argmin}{\arg\,\min} 
\newcommand{\defeq}{\coloneqq} 

\RenewCommandCopy{\leq}{\leqslant} 
\RenewCommandCopy{\geq}{\geqslant} 
\renewcommand{\d}[1]{\operatorname{d}\!{#1}} 
\newcommand{\closure}[1]{\overline{#1}} 
\newcommand{\card}[1]{\operatorname{card}\left(#1\right)} 
\newcommand{\Min}{\operatorname{Min}} 
\newcommand{\obj}{\operatorname{obj}} 
\newcommand{\morph}{\operatorname{mor}} 
\newcommand{\ev}{\operatorname{ev}} 
\newcommand{\Hom}{\mathrm{Hom}} 
\newcommand{\bpreceq}{\mathrel{\bm{\preceq}}} 
\newcommand{\id}{\operatorname{id}} 
\newcommand{\pf}{{}_{\#}} 

\newtheorem{theorem}{Theorem}[section]
\newtheorem{lemma}[theorem]{Lemma}
\newtheorem{recall}[theorem]{Recall}

\newtheorem{proposition}[theorem]{Proposition}
\newtheorem{corollary}{Corollary}[theorem]

\theoremstyle{definition}
\newtheorem{remark}[theorem]{Remark}

\newtheorem{axiomO}{O}

\newtheorem{axiomI}{I}
\newtheorem{axiomI*}{I*}

\numberwithin{equation}{section}

\newtheoremstyle{bolddefinition}%
  {10pt}
  {15pt}
  {\normalfont}
  {}
  {\bfseries}
  {.}
  {0.5em}
  {%
    \thmname{\bfseries #1}\ \thmnumber{\bfseries #2}\thmnote{\bfseries\ (#3)}%
  }

\theoremstyle{bolddefinition}
\newtheorem{definition}[theorem]{Definition}

\newcommand{\inlineDef}[2][]{%
  \refstepcounter{theorem}
  \textbf{Definition~\thetheorem}%
  \if\relax\detokenize{#1}\relax
  \else\textbf{ (#1)}%
  \fi
  \textbf{. }%
  #2%
}

\title{Reformulation Invariance and the Axiomatic Foundations of Inference}
\date{\today}
\author[1]{\fnm{Rapha\"el} \sur{Tr\'esor}~\orcidlink{0009-0009-0411-9002}\email{r.v.tresor@tue.nl}}
\author[1]{\fnm{Thijs} \spfx{van de} \sur{Laar}}
\author[1,2]{\fnm{Bert} \spfx{de} \sur{Vries}}

\affil[1]{\orgdiv{Department of Electrical Engineering}, \orgname{Eindhoven University of Technology}, \orgaddress{\postcode{5612 AE}, \city{Eindhoven}, \country{The Netherlands}}}
\affil[2]{\orgname{GN Hearing}, \orgaddress{\postcode{5612 AB}, \city{Eindhoven}, \country{The Netherlands}}}

\begin{document}

\maketitle
\begin{abstract}
   Maximum entropy, Bayesian updating, and exponential-family estimation are all instances of a common inference principle: selecting the measure or distribution that minimizes a divergence subject to the available constraints.
    Which divergence to use is usually decided by analytic convenience, by empirical performance, or by a set of axioms chosen to single it out, leaving open a basic question: why one divergence and not another? We answer it from a single requirement: an inference method should return the same answer whenever the same problem is presented in an equivalent form, for instance, after simply renaming its parts.
    This requirement alone forces inference to be the minimisation of a classical divergence, and each further reformulation it must respect tightens the admissible family one notch, narrowing the broad \(f\)-divergences to the \(\alpha\)-divergences and finally to the single Kullback-Leibler (KL) divergence.

    Mathematically, inference is recast from minimising a numerical functional to selecting a least element under a preorder on positive measures, a divergence being merely one numerical scale that reproduces that preorder.
    The reformulations are the morphisms of a category of inference problems, and the invariance requirement says the inference operator is a covariant functor into the category of statistical models of \v{C}encov, mirroring his characterisation of the Fisher metric.
    The representation is proved on finite spaces and lifted to general measurable spaces by an elementary closure, covering discrete and continuous spaces alike.

    Earlier axiomatisations, such as those of Shore-Johnson and Csisz\'ar, postulate their consistency axioms directly and only on finite alphabets; here the axioms follow from reformulation invariance alone.
\end{abstract}

\section{Introduction}\label{sec:introduction}

Inference by divergence minimisation, the principle underlying maximum entropy methods, Bayesian updating, and exponential-family estimation, selects an inferred measure as the minimiser of a divergence on a constraint set.
The divergence itself, Kullback-Leibler (KL)-, \(f\)-, or \(\alpha\)-, is usually justified by analytic convenience, by historical use, or by axioms postulated alongside the choice itself.
This work takes the opposite stance.
We require an inference operator to return the same answer under every reformulation of the inference problem that carries the same information, and we show that this single requirement forces the operator to be the minimisation of a classical divergence.
The classical divergences thus appear not as analytic objects to be axiomatised, but as the unique representations of inference operators that process information independently of how the underlying problem is formulated.

\subsection{Invariance under reformulation of the inference problem}\label{subsec:intro-invariance}

We take as a guiding principle that an inference theory should return the same answer under every equivalent formulation of the same inference problem. 
The equivalent formulations of a problem are the measurable maps relating the spaces on which the problem is posed.
We construct a category of inference problems whose morphisms are exactly these reformulations, together with the category of statistical models of \citet{cencov1982}, and we require the inference operator to be a covariant functor between the two.
This single requirement, presented informally on a guiding example in Section~\ref{sec:main-result}, isolates inference by classical divergence minimisation along a strict hierarchy:
\(f\)-divergence minimisation under the bare invariance,
\(\alpha\)-divergence minimisation when mass and distribution information are processed independently,
and KL minimisation when moment information is further decoupled from the inferred distribution.
The hierarchy mirrors \citet{cencov1982}'s characterisation of metrics on statistical models: Markov-kernel invariance alone admits a family of invariant metrics, and the Fisher metric is singled out only once invariance under linear reparametrisations is further imposed.
The same pattern recurs here, the bare invariance under reformulations of the inference problem leaves the \(f\)-divergence family, and KL divergence minimisation is singled out only once invariance under strictly positive measurable reweighting is added.

\subsection{Representation by ordering of measures}\label{subsec:intro-ordering}

The functorial invariance of Section~\ref{subsec:intro-invariance} is abstract: it constrains an operator on the category of inference problems without saying anything about its internal mechanism.
To make it tractable, we step away from the numerical view of inference as the minimisation of a functional and recast the inference operator as the selection of a minimum under a preorder on the space of finite positive measures.
The classical divergences are then demoted to numerical representations of this preorder rather than primitive objects: the preorder is what the operator depends on, and a divergence is one possible scale that reproduces its ranking.
The ordering view places the work in the lineage of logic-based foundations of inference (Cox's theorem for probability \cite{coxProbabilityFrequencyReasonable1946,vanhornConstructingLogicPlausible2003}, Knuth's preorder view of Maximum Entropy (MaxEnt) \cite{knuthFoundationsInference2012}) rather than alongside the functional-analytic axiomatisations of \citet{shoreAxiomaticDerivationPrinciple1980} and \citet{paris_note_1990}, which stack independent requirements until uniqueness is forced.

\subsection{A hierarchy of classical divergences}\label{subsec:intro-hierarchy}

The same set of axioms, narrowed by one logical assumption at each step, isolates in turn the \cite{csiszarInformationstheoretischeUngleichungUnd1963} $f$-divergences, the \cite{renyiMeasuresEntropyInformation1961} $\alpha$-divergences, and the \cite{kullbackInformationSufficiency1951} divergence, together with the inference operators they induce.
The three families, usually treated as separate constructions with their own arithmetic axiomatisations \cite{faddeev1956concept,hobson_new_1969,renyiMeasuresEntropyInformation1961}, are exposed here as a strict hierarchy of rationality assumptions, each family being the previous one plus one independence requirement.
The hierarchy unifies maximum-entropy inference, Bayesian updating, and exponential-family estimation under a single derivation.
In particular, \(f\)-divergence minimisation is recovered as the unique additive-divergence inference compatible with Bayesian conditioning.

\subsection{From finite to general measurable spaces}\label{subsec:intro-general-spaces}

Earlier rigorous axiomatisations of inference \cite{csiszar_why_1991,paris_note_1990} are limited to finite alphabets, and no work proves the continuous-space case.
We bridge the two settings constructively.
Every representation theorem stated in the main text is given directly on a general measurable space \((E,\mathcal{E})\); the proof proceeds in two steps, first on finite spaces (Appendix~\ref{sec:axiomatic-finite-spaces}), then lifted to general measurable spaces through an ordering closure that aggregates the finite-space preorders across increasingly fine partitions (Definition~\ref{def:closure-preordering}, Appendix~\ref{sec:convergence-proof}).
The construction is elementary in the sense that it avoids functional analysis on Banach spaces and Radon-Nikodym derivatives on abstract spaces.
The same axiomatic consequently applies on finite, countable, continuous, and abstract measurable spaces alike.

\subsection{Contributions}\label{subsec:intro-contributions}

These results establish a single thesis: the classical divergences are not arbitrary analytic objects but the necessary representation of any inference that processes information independently of how its problem is formulated.
The main contributions of this paper are the following.
\begin{enumerate}
    \item We prove that an inference operator returns the same answer under every information-preserving reformulation of the problem if and only if it is the minimisation of a classical divergence, so the divergence is forced by the invariance under reformulation rather than chosen in advance. This invariance is the requirement that the operator be a covariant functor into the category of statistical models of \v{C}encov, equivalent to a small set of consistency axioms.
    \item We recast inference away from the minimisation of a numerical functional and as the selection of a minimum under a preorder on measures, the classical divergences being demoted to numerical scales that reproduce this preorder. Cast on the preorder, the axioms acquire logical content rather than numerical.
    \item We expose the \(f\)-, \(\alpha\)-, and KL divergence families not as separate constructions but as a single hierarchy in which each family is the previous one narrowed by one further independence requirement, KL being isolated last.
    \item We establish every representation theorem on general measurable spaces using an elementary closure method, so that our foundation covers finite, countable, continuous, and abstract spaces.
\end{enumerate}

\subsection{Outline}\label{subsec:intro-outline}

Section~\ref{sec:related-work} positions the paper against existing axiomatisations of inference and the categorical characterisations of information-theoretic quantities.
Section~\ref{sec:main-result} states the main result informally on a guiding example, without category-theoretic background.
Section~\ref{sec:axiomatic-characterization} formalises the recast of Section~\ref{subsec:intro-ordering}: it states the consistency axioms and the divergence and inference representation theorems directly on a general measurable space, with pointers to the appendix proofs.
Section~\ref{sec:category-derivation} takes invariance under reformulation as primitive, formalises it as a functoriality condition on the inference operator, and derives the consistency axioms from it.
Sections~\ref{sec:discussion} and~\ref{sec:conclusion} discuss implications and open directions.

The appendices hold the proof machinery for the representation theorems of Section~\ref{sec:axiomatic-characterization}, ordered by mathematical progression.
Appendix~\ref{sec:axiomatic-finite-spaces} proves the divergence and inference representations on finite spaces.
Appendix~\ref{sec:convergence-proof} constructs the ordering closure and lifts these representations to general measurable spaces.

\section{Related Work}\label{sec:related-work}

A first-principles justification of divergence-based inference has been pursued from several directions, each settling part of the question without resolving it as a whole.

The modern formulation of the Maximum Entropy Principle was introduced by Jaynes as a heuristic rule for selecting probability distributions under constraints~\citep{jaynesInformationTheoryStatistical1957,jaynesRationaleMaximumentropyMethods1982}.
Subsequent works gave it a conceptual rather than purely heuristic standing~\citep{skilling_axioms_1988,catichaEntropicInferenceFoundations2012,catichaUpdating2021}.
Closest in spirit to the present paper is the lattice-quantification programme of \citet{knuthLatticeDuality2005,knuthFoundationsInference2012}, which derives the probability calculus and KL divergence from order-theoretic symmetries and locates a preorder at the root of inference.
This lineage is correct in spirit but remains a conceptual programme rather than a proved theorem.
We share its lattice intuition and supply the missing derivation: Bayesian conditioning, for one, emerges as the content of a single consistency axiom (Corollary~\ref{cor:bayes-identification}) rather than being postulated alongside it.

The first rigorous axiomatisations of an inference operator are those of \citet{shoreAxiomaticDerivationPrinciple1980}, \citet{paris_note_1990} and \citet{csiszar_why_1991}, the last of which derives $f$-divergence, Bregman and KL divergence minimisation each from a separate set of postulates; see also the survey of \citet{csiszarAxiomaticCharacterizations2008}.
These are correct on their domain but incomplete in two recurring ways.
First, the inference operator is axiomatised as a numerical method, so the hierarchy $f\to\alpha\to\text{KL}$ does not emerge but is reconstructed family by family.
Second, all three operate on finite discrete sample spaces; the $\sigma$-algebra is absent as a primitive, which blocks extension to general measurable spaces.
The present axiomatisation acts directly on measurable events, recovers the three families along a single hierarchy of consistency axioms, and obtains the general measurable-space case constructively.

The uniqueness claim of the Shore-Johnson programme is, moreover, mathematically problematic on its own terms.
\citet{uffinkCanMaximumEntropy1995} shows that the Shore-Johnson axioms do not isolate the KL divergence even in the finite case, the R\'enyi family being equally compatible, and do not extend to continuous spaces.
The debate continued inside complex-systems physics: \citet{presseGhoshLeeDill2013} defends Shore-Johnson against Tsallis-type entropies, \citet{tsallis2015ConceptualInadequacy} argues that the Shore-Johnson axioms are conceptually inadequate for wide classes of systems, and \citet{presseReply2015} replies in turn.
More recently, \citet{jizbaKorbel2020,jizbaKorbel2019NonShannonian} unify the Shannon-Khinchin and Shore-Johnson axiomatics and recover the one-parameter Uffink-Jizba-Korbel family of admissible entropies.
The debate persists because the Shore-Johnson framework lacks both the $\sigma$-algebra primitive and an axiom controlling reweightings of the underlying measure.
Adding strictly positive measurable reweighting to our hierarchy isolates KL from the R\'enyi-Jizba-Korbel family: that family is admitted by our $f$-divergence layer and excluded by our KL layer, closing the uniqueness gap of \citet{uffinkCanMaximumEntropy1995} on its own terms.
Table~\ref{tab:axiom-correspondence} sets out the row-by-row correspondence with both axiomatisations, showing that no Shore-Johnson axiom controls this reweighting.

A parallel tradition is correct but axiomatises the \emph{numerical value} of a divergence rather than the inference operator built from it.
Classical divergences, including the KL divergence~\citep{kullbackInformationSufficiency1951}, the $f$-divergences~\citep{csiszarInformationstheoretischeUngleichungUnd1963} and the R\'enyi divergences~\citep{renyiMeasuresEntropyInformation1961}, were given arithmetic axiomatisations at or shortly after their formulation~\citep{faddeev1956concept,hobson_new_1969}.
Within this tradition, \citet{amari2009AlphaUnique} singles out the $\alpha$-divergences as the unique class lying at the intersection of $f$-divergences and Bregman divergences under information monotonicity, a result that anchors our $\alpha$-layer from the divergence side.
Properties of the R\'enyi and KL divergences on general spaces are surveyed by \citet{vanErvenHarremoes2014}.
In all of these works it is the divergence, not the inference operator, that is the axiomatised primitive.

The closest related works to this paper in technical form are the functorial characterisations of statistical and information-theoretic quantities.
The category of statistical models of \citet{cencov1982} singles out the Fisher metric and the Amari-\v{C}encov tensor as the unique Markov-kernel invariants on finite probability simplexes, with the Fisher metric isolated only after adding invariance under linear reparametrisation; the analogous step appears at the KL layer of our hierarchy.
\citet{ayJostLeSchwachhofer2017InfoGeo} extend the \v{C}encov uniqueness theorem from finite simplexes to parametrised measure models over general sample spaces through Banach-manifold structure; we obtain the analogous extension through an ordering closure on positive finite measures, which is the route adapted to lifting an inference operator rather than a metric.
\citet{baezFritzLeinster2011} characterise Shannon entropy as the unique functorial, convex-linear, lower-semicontinuous information-loss assignment on the category $\mathrm{FinStat}$, and \citet{baezFritz2014} extend the same line to a functorial characterisation of the KL divergence on the same finite category.
\citet{gagnePanangaden2018} lift the relative-entropy characterisation from $\mathrm{FinStat}$ to standard Borel spaces through Polish-space limit constructions.
The two liftings trade generality: theirs covers arbitrary pairs of measures on a standard Borel space, ours only measures dominated by a fixed reference, but on any $\sigma$-finite measured space the ordering closure reduces every comparison to the countably generated $\sigma$-algebra of a density ratio.
The characterised objects also differ, a Bayesian-loss functor there and an inference preorder here, and the two representations agree on standard Borel models with equivalent measures.
\citet{perrone2024} proceeds in the opposite direction: a divergence enriches a Markov category and is used to define entropy and mutual information, recovering the Shannon, R\'enyi and Gini-Simpson quantities; this route is mathematically problematic in the general case, since \citet{perrone2024} reports that the nondiscrete setting degenerates, with entropies maximal for atomless distributions.
We instead derive the divergence rather than assume it, and our ordering closure lifts the axioms constructively from the finite to the general measurable case.

Two strands run through this landscape.
The categorical characterisations that achieve uniqueness constrain a \emph{derived} object, a divergence value, an entropy functor, or the Fisher metric, never the inference operator itself; the axiomatisations that do reach the operator constrain it by isolated numerical postulates rather than a single organising principle.
Taking the inference operator as the object to be characterised is what reveals that principle: invariance under reformulation of the inference problem is the property that singles out canonical inference.
Imposing this invariance alone yields divergence minimisation as the unique representation of the operator, the $f\to\alpha\to\text{KL}$ hierarchy by successive independence requirements, and validity on arbitrary measurable spaces.
This is the central problem left open by every approach above, and closing it settles the uniqueness question of \citet{uffinkCanMaximumEntropy1995} on its own terms.

\section{Main Result: Informal Overview}\label{sec:main-result}

This section presents our main result informally, through a guiding example.
Consider the task of \emph{inferring the number of people and their position in a building} \(B\) composed of floors \(\{F_j\}_{j\in J}\), each containing rooms \(\{R_k F_j\}_{k\in K_j}\), from a set of information \(I\).
An inferred distribution \(P^{\star}\) assigns to each room a number of people \(P^{\star}(R_k F_j)\), to each floor the sum \(\sum_{k\in K_j}P^{\star}(R_k F_j) = P^{\star}(F_j)\), and to the building the total \(\sum_{j\in J}P^{\star}(F_j) = P^{\star}(B)\).

\begin{figure}[tb]
    \centering
    \begin{subfigure}[t]{\linewidth}
        \centering
        \begin{tikzpicture}
            \draw[very thick] (0,0) rectangle (6,4.5);
            \draw[thick] (0,1.5) -- (6,1.5);
            \draw[thick] (0,3)   -- (6,3);
            \draw (1.5,0) -- (1.5,1.5); \draw (3,0) -- (3,1.5); \draw (4.5,0) -- (4.5,1.5); 
            \draw (2,1.5) -- (2,3);     \draw (4,1.5) -- (4,3);                            
            \draw (3,3) -- (3,4.5);                                                        
            \node[align=center,font=\footnotesize] at (0.75,0.75) {\(R_1F_1\)\\[1pt]\(2\)};
            \node[align=center,font=\footnotesize] at (2.25,0.75) {\(R_2F_1\)\\[1pt]\(1\)};
            \node[align=center,font=\footnotesize] at (3.75,0.75) {\(R_3F_1\)\\[1pt]\(1\)};
            \node[align=center,font=\footnotesize] at (5.25,0.75) {\(R_4F_1\)\\[1pt]\(1\)};
            \node[align=center,font=\footnotesize] at (1,2.25) {\(R_1F_2\)\\[1pt]\(3\)};
            \node[align=center,font=\footnotesize] at (3,2.25) {\(R_2F_2\)\\[1pt]\(2\)};
            \node[align=center,font=\footnotesize] at (5,2.25) {\(R_3F_2\)\\[1pt]\(2\)};
            \node[align=center,font=\footnotesize] at (1.5,3.75) {\(R_1F_3\)\\[1pt]\(1\)};
            \node[align=center,font=\footnotesize] at (4.5,3.75) {\(R_2F_3\)\\[1pt]\(4\)};
            \node[left,font=\small] at (0,0.75) {\(F_1\)};
            \node[left,font=\small] at (0,2.25) {\(F_2\)};
            \node[left,font=\small] at (0,3.75) {\(F_3\)};
            \node[right,font=\footnotesize] at (6,0.75) {\(P^{\star}(F_1)=5\)};
            \node[right,font=\footnotesize] at (6,2.25) {\(P^{\star}(F_2)=7\)};
            \node[right,font=\footnotesize] at (6,3.75) {\(P^{\star}(F_3)=5\)};
            \node[above,font=\small] at (3,4.5) {Building \(B\)};
            \node[below=1pt,font=\small] at (3,0) {\(P^{\star}(B)=\sum_{j\in J}P^{\star}(F_j)=\sum_{j\in J}\sum_{k\in K_j}P^{\star}(R_kF_j)=17\)};
        \end{tikzpicture}
        \caption{The building as a partition: floors \(\{F_j\}\) partition \(B\), rooms \(\{R_kF_j\}\) partition each floor. The inferred measure \(P^{\star}\) assigns a count to every room; room counts add up to the floor total, floor totals to the building total.}
        \label{subfig:building-plan}
    \end{subfigure}

    \vspace{0.9em}

    \begin{subfigure}[t]{\linewidth}
        \centering
        \resizebox{\linewidth}{!}{%
        \begin{tikzpicture}[
            bnode/.style={draw,rounded corners,minimum width=0.9cm,font=\small},
            leaf/.style={draw,font=\scriptsize,inner sep=2pt},
        ]
            \node[bnode] (B)  at (6.05,4.4) {\(B\)};
            \node[bnode] (F1) at (1.95,2.2) {\(F_1\)};
            \node[bnode] (F2) at (6.70,2.2) {\(F_2\)};
            \node[bnode] (F3) at (10.15,2.2){\(F_3\)};
            \node[leaf] (r11) at (0,0)    {\(R_1F_1\)};
            \node[leaf] (r12) at (1.3,0)  {\(R_2F_1\)};
            \node[leaf] (r13) at (2.6,0)  {\(R_3F_1\)};
            \node[leaf] (r14) at (3.9,0)  {\(R_4F_1\)};
            \node[leaf] (r21) at (5.4,0)  {\(R_1F_2\)};
            \node[leaf] (r22) at (6.7,0)  {\(R_2F_2\)};
            \node[leaf] (r23) at (8.0,0)  {\(R_3F_2\)};
            \node[leaf] (r31) at (9.5,0)  {\(R_1F_3\)};
            \node[leaf] (r32) at (10.8,0) {\(R_2F_3\)};
            \draw (B) -- (F1); \draw (B) -- (F2); \draw (B) -- (F3);
            \draw (F1) -- (r11); \draw (F1) -- (r12); \draw (F1) -- (r13); \draw (F1) -- (r14);
            \draw (F2) -- (r21); \draw (F2) -- (r22); \draw (F2) -- (r23);
            \draw (F3) -- (r31); \draw (F3) -- (r32);
            \node[right,font=\small] at (12.0,4.4) {\(\mathcal{A}_{B}\) \scriptsize(building)};
            \node[right,font=\small] at (12.0,2.2) {\(\mathcal{A}_{F}\) \scriptsize(floors)};
            \node[right,font=\small] at (12.0,0)   {\(\mathcal{A}_{R}\) \scriptsize(rooms)};
            \draw[->,thick] (-1.4,0) -- (-1.4,4.4)
                node[midway,left,font=\scriptsize,align=center] {coarsening\\\(t\)};
        \end{tikzpicture}}
        \caption{The same data as a refinement of \(\sigma\)-algebras \(\mathcal{A}_{B}\subseteq\mathcal{A}_{F}\subseteq\mathcal{A}_{R}\), with \(\mathcal{A}_{B}=\{\emptyset,B\}\). Coarse-graining \(t\) sends each room to its floor and each floor to the building, collapsing the finer resolution onto the coarser one.}
        \label{subfig:building-sigma}
    \end{subfigure}
    \caption{Guiding example: a building \(B\) of three floors \(\{F_1,F_2,F_3\}\) holding \(4\), \(3\), and \(2\) rooms. Figure~\subref{subfig:building-plan} shows the inferred counts \(P^{\star}\), additive from rooms to floors to the building; figure~\subref{subfig:building-sigma} shows the three resolutions as a chain of \(\sigma\)-algebras, the room partition refining the floor partition refining \(\{\emptyset,B\}\). A coarse-resolution constraint such as \(P^{\star}(F_1)=5\) or \(7\le P^{\star}(F_2\cup F_3)\le 12\) lives on \(\mathcal{A}_{F}\) and says nothing about \(\mathcal{A}_{R}\). The two constraints are also disjoint: \(P^{\star}(F_1)=5\) constrains only \(\mathcal{A}|_{F_1}\) and says nothing about \(\mathcal{A}|_{F_2\cup F_3}\), while \(7\le P^{\star}(F_2\cup F_3)\le 12\) constrains only \(\mathcal{A}|_{F_2\cup F_3}\) and says nothing about \(\mathcal{A}|_{F_1}\), so the two pieces of information are uncorrelated and split the problem into two independent inference tasks on the disjoint subspaces \(F_1\) and \(F_2\cup F_3\).}
    \label{fig:building-example}
\end{figure}
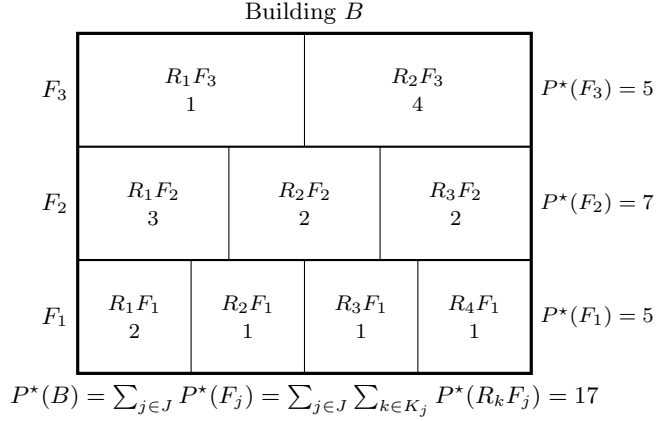
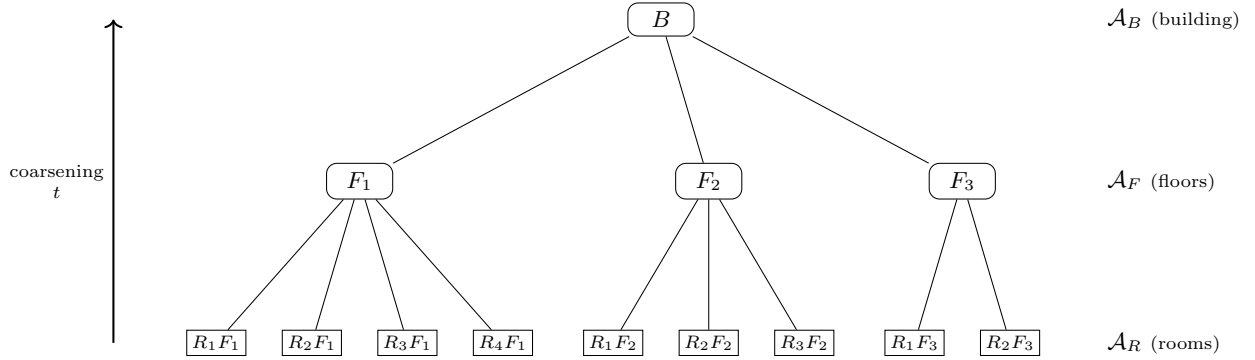

\begin{theorem}[Main Result, Informal]\label{theo:main-informal}

Let \(\mathcal{T}\) be an inference operator on measurable spaces.
$\mathcal{T}$ processes information identically under splitting into disjoint subproblems and coarsening to a lower resolution if and only if it minimises an $f$-divergence against a reference measure $Q$. These same invariances fix the prior $Q$ for incomplete information, so the existence of a reference measure is a consequence of reformulation invariance rather than an assumption.
Requiring in addition that \(\mathcal{T}\) processes coarse-grain mass and fine-grain distribution information independently restricts the family to the \(\alpha\)-divergences, and the further requirement that it processes moment and distribution information independently restricts it to a single operator, the minimisation of the Kullback-Leibler divergence.
Imposing invariance under relabelling of components forces \(Q\) to be uniform, so each divergence reads as the corresponding entropy functional and inference becomes maximum-entropy inference (Table~\ref{tab:main-result-summary}).
\end{theorem}

Table~\ref{tab:main-result-summary} summarises this result: each reformulation, the consistency axiom it embodies, and the divergence (or, under relabelling, the entropy functional) it forces.

\subsection{Transport and reformulation of an inference problem}\label{subsec:main-result-transport}

A good inference theory must return the same answer whenever the same problem is presented in two equivalent forms.
This invariance under reformulation is the single principle we impose: the paragraphs below are not a list of independent axioms but illustrations of this one principle on the inference problem above, each made precise and rigorous in Section~\ref{sec:category-derivation}.
They collect three canonical reformulations that must leave the inferred measure unchanged, and the choice they force when information is incomplete.

\paragraph{Relabelling of the problem.}
The original problem decomposes into several sub-tasks: counting the people in the building, on one floor, or in one room.
Each of these is an instance of the same canonical task, counting the total of a single closed space, with only the name attached to that space differing.
Similarly, inferring the distribution of people across the floors of the building and inferring the distribution of people across the rooms of a single floor are the same problem twice: a total space subdivided into disjoint subspaces, with only the labels (\emph{building / floor} versus \emph{floor / room}) being swapped.
A correct inference theory must treat problems that differ only by the names attached to their components identically.

\paragraph{Disjoint problems.}
Suppose now that two pieces of information are given: one about \(F_1\) (for instance, \emph{"there are 5 people on floor 1"}, i.e.\ \(P(F_1)=5\)), and one about \(F_2\cup F_3\) (e.g.\ \(7\le P(F_2\cup F_3)\le 12\)), with nothing relating the two.
Replacing the first piece by \(P(F_1)=7\) should leave the inferred measure on \(F_2\cup F_3\) untouched, since the two pieces of information are uncorrelated.
The original problem splits into two independent inference problems, one on \(F_1\) and the other on \(F_2\cup F_3\).
A correct inference theory must not couple uncorrelated information across disjoint subspaces.

\paragraph{Coarse-resolution information.}
Suppose now that the information available concerns only the total number of people on each floor and says nothing about how that number is distributed across rooms.
The inferred floor-level measure should then be the same as if the building had been described from the start without rooms.
A correct inference theory must process information at a \emph{coarse resolution} as if the \emph{fine resolution} were not part of the description.

\paragraph{The prior.}

The original task still asks for room counts, yet the information may fix only the floor totals.
How to split a floor total across its rooms should not depend on the total itself: the distribution within floor 1 cannot differ between the information \(P(F_1)=5\) and the information \(P(F_1)=7\), or it would invent room-level information from a floor-level statement that carries none.
Reformulation invariance pins this down: coarsening sends the room-level problem to its floor-only version, and the floor answer is carried back to the room-level by one completion rule, the same for every floor total.
This completion rule is exactly the \emph{prior} \(Q\) of Table~\ref{tab:main-result-summary}: not an extra input but a rule enforced by reformulation invariance.

\subsection{Independence between mass and distribution}\label{subsec:main-result-independence}

\paragraph{Mass-distribution independence.}

The previous paragraph already required the room ratios within a floor to be inferred independently of that floor's total when no room-level information is available; the same must hold when a room-ratio constraint and a separate floor-total constraint are both present but uncorrelated.
Given a constraint describing only the distribution at the room level (for instance, \(P(R_1F_2) = 2\,P(R_2F_2\cup R_3F_2)\)) together with a separate constraint on the floor total (\(7\le P(F_2)\le 12\)), the inferred ratio between rooms of floor 2 should not change if the floor total later sharpens to \(P(F_2)=14\).
An inference theory that processes both fine-grain distribution and coarse-grain mass information should process the latter without affecting the former.

\paragraph{Moment-distribution independence.}

Mass-distribution independence should not hold only for the total number of people; it should survive every reweighting of the measure.
Multiply each room's count by a fixed positive weight \(f\) (for instance the rent charged per occupant of that room) and total the result: this weighted total \(\langle P, f\rangle = \sum_{\text{rooms}} f\,P\) is the building's total rent, the plain total being the case where every weight is one.
A piece of information such as \emph{"the total rent for floor 2 is 120"}, i.e.\ \(\sum_{k\in K_2} f(R_k F_2)P(R_k F_2)=120\), reports only the overall rent for a floor, not how that rent divides between its rooms, while \(f(R_1 F_2)P(R_1 F_2)= 3f(R_2 F_2)P(R_2 F_2)\) informs only how the rent is split.
Changing the total rent should leave the inferred room ratios untouched, since the two pieces of information are uncorrelated.
A correct inference theory must process every such moment information without letting its value affect the inferred distribution.

\subsection{Roadmap to the formal results}\label{sec:roadmap-formal-results}

The single principle of Section~\ref{subsec:main-result-transport} is made precise in three steps, applied successively to the inference operator and to its quotient operators on distribution and moment information.
Section~\ref{subsec:information-category-and-functor} shows that the reformulations (relabelling, disjoint splitting, and coarsening) are characterised exactly by the measurable maps \(t : (\Omega,\mathcal{A}) \to (\Omega',\mathcal{A}')\) and identifies what information each one faithfully preserves, which selects the \(f\)-divergence.
Section~\ref{subsec:quotient-operators} shows that mass-distribution independence induces a quotient inference operator on the space of distribution information, which selects the \(\alpha\)-divergences.
Finally, moment-distribution independence is captured by a family of quotient inference operators required to agree on the intersection of their domains, which selects the Kullback--Leibler divergence as the unique compatible representative.

\begin{table}[htbp]
  \centering

  \caption{From reformulation invariance to canonical inference.
  Each row pairs a reformulation of the inference problem (left) with the consistency axiom on the inference operator \(\mathcal{T}\) it embodies (centre) and the inference rule it forces (right).
  Without relabelling, that rule is minimisation of a divergence against a reference measure \(Q\) that is itself forced by the invariances, not assumed; imposing relabelling forces \(Q\) uniform, so each divergence becomes the corresponding entropy functional, recovering maximum-entropy inference.}
  \label{tab:main-result-summary}
  \small
  \begin{tabularx}{\linewidth}{@{}>{\raggedright\arraybackslash}p{0.155\linewidth} >{\raggedright\arraybackslash}X >{\centering\arraybackslash}p{0.15\linewidth} >{\centering\arraybackslash}p{0.19\linewidth}@{}}
    \toprule
    Reformulation invariance & Simplified axiom on the inference operator \(\mathcal{T}\) & Without relabelling & With relabelling \\
    \midrule
    Disjoint problems \mbox{(I\ref{axiom:isolated-system})}
      & An uncorrelated fact about one region leaves the inference on the rest unchanged.
      & \multirow{3}{=}{\centering\arraybackslash \(f\)-divergence \(D_f(P\Vert Q)\)}
      & \multirow{3}{=}{\centering\arraybackslash \(f\)-entropy (generalised MaxEnt)} \\
    Coarse resolution \mbox{(I\ref{axiom:lower-scale-conservation})}
      & Coarse-resolution information is processed as if the fine resolution were absent. & & \\
    Prior information \mbox{(I\ref{axiom:prior})}
      & With no fine information, the split across fine cells is independent of the coarse total. & & \\

    \midrule
    \(+\) Mass \(\perp\) distribution \mbox{(I\ref{axiom:upper-scale-conservation})}
      & The inferred distribution is independent of the total mass.
      & \(\alpha\)-divergence & R\'enyi / Tsallis entropy \\
    \midrule
    \(+\) Moment \(\perp\) distribution \mbox{(I\ref{axiom:linear-constraint-independance})}
      & The inferred distribution is invariant under any strictly positive measurable reweighting.
      & Kullback--Leibler divergence & Shannon entropy (MaxEnt) \\
    \bottomrule
  \end{tabularx}
\end{table}

\section{Axiomatic Characterization}\label{sec:axiomatic-characterization}

\subsection{Divergence: axioms and representation}\label{subsec:divergence-axioms-presentation}

\medskip
\noindent\textbf{Notation.}
A measurable space $(\Omega,\mathcal{A})$ consists of a set $\Omega$ and a $\sigma$-algebra $\mathcal{A}$.
$\mathcal{M}(\Omega,\mathcal{A})$ denotes the set of finite non-negative measures of positive total mass on
$(\Omega,\mathcal{A})$.
We will work on the measured space $(\Omega,\mathcal{A},Q)$ where the prior/reference measure is denoted $Q$ and is taken strictly positive throughout, so that $P\ll Q$ and the divergence is well defined; the variable measure of $\mathcal{M}(\Omega,\mathcal{A})$ is typically denoted $P$.
For a measure $P$ and $A \neq \emptyset\in \mathcal{A}$, we write
$P(\cdot \mid A)$ for the conditional measure.
The atoms of a finite $\sigma$‑algebra $\mathcal{A}$ are denoted $\mathrm{At}(\mathcal{A})$.
For any real-valued measurable function \(f\) on $(\Omega,\mathcal{A})$, $\langle P, f \rangle$ denotes the sum $\sum_{a\in\mathrm{At}(\mathcal{A})} f(a) P(a)$.
Finally, we denote \(\mathcal{F}_f(c) \defeq \{P\in \mathcal{M}(\Omega,\mathcal{A}) \,:\, \langle P, f \rangle = c \}\) a fibre of the map \(P \mapsto \langle P, f \rangle\)  on \(\mathcal{M}(\Omega,\mathcal{A})\) for \(c>0\).
The fibres \(\{\mathcal{F}_f(c)\}_{c>0}\) are pairwise homeomorphic via rescaling, so we fix the \emph{canonical fibre} \(\mathcal{F}_f \defeq \mathcal{F}_f(1)\) and the projection
\[
    \pi_f : \mathcal{M}(\Omega,\mathcal{A}) \to \mathcal{F}_f \,,
    \qquad
    \pi_f(P) \defeq \frac{P}{\langle P, f \rangle}\,.
\]
The relation defined by \(\pi_f(P) = \pi_f(P')\) holds iff \(P' = cP\) for some \(c>0\), so it does not depend on \(f\); we write it \(\sim\).
Its equivalence classes are the rays \([P] \defeq \{P' \,:\, \exists\, c>0,\ P' = cP\}\), called the \emph{distribution profile} of \(P\).
The projection \(\pi_f\) realises each profile concretely as the unique representative in the canonical fibre \(\mathcal{F}_f\) (the one with \(\langle\cdot,f\rangle = 1\)), and descends to a homeomorphism between the (\(f\)-independent) quotient and the fibre,
\(
    \mathcal{M}(\Omega,\mathcal{A})/\!\sim \;\cong\; \mathcal{F}_f \,,
\)
so each fibre is one concrete realisation of the distribution profiles.
For the constant function \(f \equiv \boldsymbol{1}\), \(\mathcal{F}_{\boldsymbol{1}}\) is the simplex of distributions.

A real-valued measurable function \(f>0\) defines the measure \(f\cdot P \in \mathcal{M}(\Omega,\mathcal{A})\) by \((f\cdot P)(a) = f(a)\,P(a)\) for all \(a\in\mathrm{At}(\mathcal{A})\).

\medskip
\noindent\textbf{Disjoint-union operator.}
For measurable spaces $(\Omega_i,\mathcal A_i)$, $i=1,2$, we denote by
$\Omega_1\sqcup\Omega_2$ their disjoint union, with $\sigma$–algebra
\[
\mathcal A_1 \oplus \mathcal A_2
:= \bigl\{A_1 \sqcup A_2 : A_1\in\mathcal A_1,\; A_2\in\mathcal A_2 \bigr\}.
\]
For measures $P_i\in\mathcal M(\Omega_i,\mathcal A_i)$, we write
\[
(P_1 \oplus P_2)(A_1 \sqcup A_2)
:= P_1(A_1) + P_2(A_2).
\]
This operator is used to express inference on independent subspaces.

\medskip
\noindent\textbf{Sub‑$\sigma$‑algebras and coarse‑grainings.}
For a sub‑$\sigma$‑algebra $\mathcal{A}' \subseteq \mathcal{A}$ (the \emph{coarsening order}, since a coarser $\sigma$‑algebra is a subset of a finer one), the restriction
(coarse‑graining) of $P$ to $\mathcal{A}'$ is denoted by $P|_{\mathcal{A}'}$.
For finite $\Omega$, $\sigma(\Omega)$ denotes the discrete $\sigma$‑algebra (the power set), the finest $\sigma$‑algebra of which every $\mathcal{A}$ is a sub‑$\sigma$‑algebra, and the trace of $\mathcal{A}$ on $A\in\mathcal{A}$ is $\mathcal{A}|_A\defeq\{B\cap A : B\in\mathcal{A}\}$.

\paragraph*{Distribution at resolution \(\mathcal{A}^\prime\).}
Let $\mathcal A'\subseteq \mathcal A$ be a sub–$\sigma$–algebra. For $P\in\mathcal M(\Omega,\mathcal A)$, the \emph{distribution at resolution $\mathcal A'$} is the
collection of conditionals probabilities
\[
P^{\mathcal A'} := \bigl(P(\,\cdot \mid a)\bigr)_{a\in \mathrm{At}(\mathcal A')} .
\]

\paragraph*{Lifting.}
Given $P,Q\in\mathcal M(\Omega,\mathcal A)$ and $\mathcal A'$ as above, the
\emph{lifting operator} redistributes the coarse masses of $P$ according to the distribution of $Q$ at resolution \(\mathcal{A}^\prime\):
\[
\bigl(P|_{\mathcal A'} \wedge Q^{\mathcal A'}\bigr)(B)
\;:=\; \sum_{a\in\mathrm{At}(\mathcal A')} P(a)\, Q(B\mid a),
\qquad B\in\mathcal A .
\]
In particular, $(P|_{\mathcal A'} \wedge Q^{\mathcal A'})(a)=P(a)$ for all $a\in\mathrm{At}(\mathcal A')$.

\medskip
\paragraph*{Linear Transformation.}
A measurable function \(f > 0\) and a total preorder \(\preceq\) on \((\Omega,\mathcal{A})\) induce a preorder \(\preceq^{f}\) defined by
    \begin{align}
        P_1 \preceq^{f} P_2 \iff f^{-1}\cdot P_1 \preceq f^{-1}\cdot P_2 \,,\quad \forall\, P_1,P_2\in\mathcal{M}(\Omega,\mathcal{A})\,,
    \end{align}
where the superscript records the moment function \(f\) held fixed, while the comparison transports the measures by \(f^{-1}\).

\paragraph*{Quotient preorder.}
If \(P_1 \preceq_{(\Omega,\mathcal{A},Q)} P_2 \iff \pi_f( P_1) \preceq_{(\Omega,\mathcal{A},\pi_f(Q))}  \pi_f(P_2)\) for all \(P_1,P_2, Q \in \mathcal{F}_f(c)\), \(c>0\), then \((\preceq^{f}_{(\Omega,\mathcal{A},cQ)})_{c>0}\) induces a quotient preorder \([\preceq_{(\Omega,\mathcal{A},[Q])}]_f\) on \(\mathcal{F}_f\) via \([P_1] \,[\preceq_{(\Omega,\mathcal{A},[Q])}]_f\, [P_2] \iff \pi_f( P_1) \preceq_{(\Omega,\mathcal{A},\pi_f(Q))}  \pi_f(P_2) \).
Throughout, square brackets carry two roles: on a measure they denote its equivalence class (\([P]\), \([Q]\), and likewise \([I]\) for an information set), while on a preorder or operator (\([\preceq]\), \([T]\)) they denote the induced quotient acting on classes.

\medskip

In the following axioms, we use the previous notation to define preorders that are coherent with the coarse-grained view of the measures.

\begin{axiomO}[Locality]\label{axiom:independence-subspace}
    Let \(A\in\mathcal{A}\) and let \((Q_0,P_0)\in\mathcal{M}(A^\complement,\mathcal{A}|_{A^\complement})^2\) be any fixed pair of measures on the complement.
    The preorder \(\preceq_{(\Omega,\mathcal{A})}\) comparing \(\bigl(Q_A\oplus Q_0,\,P_A\oplus P_0\bigr)_{(Q_A,P_A)\in\mathcal{M}(A,\mathcal{A}|_A)^2}\) does not depend on the pair \((Q_0,P_0)\), and coincides with the preorder \(\preceq_{(A,\mathcal{A}|_A)}\) on the subspace \(A\).
\end{axiomO}

\begin{axiomO}[Reference Measure]\label{axiom:reference-measure-consistency}
    \begin{align}\label{item:entropic-ordering}
         P|_{\mathcal{A}'} \wedge  Q^{\mathcal{A}'} \preceq_{(\Omega,\mathcal{A}, Q)}  P  \,, \quad \forall \,\mathcal{A}' \subseteq \mathcal{A},\; Q, P\,\in\mathcal{M}(\Omega,\mathcal{A}) \,,
    \end{align}
    with equivalence if and only if \( P^{\mathcal{A}'} = Q^{\mathcal{A}'}\).
\end{axiomO}

\begin{axiomO}[Upper Scale Consistency]\label{axiom:upper-scale-consistency}
    \begin{align}\label{item:scale-consistency}
            P|_{\mathcal{A}'} \wedge  Q^{\mathcal{A}'} \preceq_{(\Omega,\mathcal{A}, Q)}  P^{\prime}|_{\mathcal{A}'} \wedge  Q^{\mathcal{A}'} \iff  P|_{\mathcal{A}'} \preceq_{(\Omega,\mathcal{A}', Q|_{\mathcal{A}'})}  P^{\prime}|_{\mathcal{A}'}   \,, \quad \forall \,\mathcal{A}' \subseteq \mathcal{A},\; Q, P, P^\prime\,\in\mathcal{M}(\Omega,\mathcal{A}) \,.
    \end{align}
\end{axiomO}

\begin{axiomO}[Lower Scale Consistency]\label{axiom:lower-scale-consistency} For all \([P_1], [P_2],[Q]\in\mathcal{F}_{\boldsymbol{1}}\),  \(P(\Omega),Q(\Omega)>0\),
\begin{align}\label{item:lower-scale-independance}
             P(\Omega) \wedge   [P_1]  \preceq_{(\Omega,\sigma(\Omega),  Q(\Omega)\wedge [Q])}  P(\Omega) \wedge   [P_2] \iff   1 \wedge  [P_1]  \preceq_{(\Omega,\sigma(\Omega), 1 \wedge [Q])}  1 \wedge  [P_2]  \,,
        \end{align}
        forming a preorder \([\preceq_{(\Omega,\sigma(\Omega), [Q])}]\) on \(\mathcal{F}_{\boldsymbol{1}}\).
\end{axiomO}

\begin{axiomO}[Linear Map Invariance]\label{axiom:linear-constraint-consistency}
    For any measurable function \(f> 0\), \(\preceq^{f}_{(\Omega,\mathcal{A})}\) follows (O\ref{axiom:independence-subspace})--(O\ref{axiom:lower-scale-consistency}), and for all \(f_1, \cdots, f_n > 0\),
    \begin{align}\label{item:linear-constraint-consistency}
        [\preceq_{(\Omega,\mathcal{A},Q)}]_{f_1} = \cdots   = [\preceq_{(\Omega,\mathcal{A}, Q)}]_{f_n}   \,,\quad \text{on } \mathcal{F}_{f_1}\cap\cdots\cap \mathcal{F}_{f_n} 
    \end{align}
    with \(n>1\), defining a unique quotient preorder \([\preceq_{(\Omega,\mathcal{A}, \Min_{[\preceq_{(\Omega,\mathcal{A},Q)}]_{f_1}}(\mathcal{F}_{f_1}\cap\cdots\cap\mathcal{F}_{f_n}))}]_{f_1,\cdots,f_n}\) on \(\mathcal{F}_{f_1}\cap\cdots\cap \mathcal{F}_{f_n}\).
\end{axiomO}

\begin{remark}
Axioms (O\ref{axiom:upper-scale-consistency}) and (O\ref{axiom:lower-scale-consistency}) form a dual pair: 
(O\ref{axiom:upper-scale-consistency}) governs how fine-scale agreement propagates to coarse scales, 
while (O\ref{axiom:lower-scale-consistency}) governs how coarse-scale agreement leaves fine-scale comparisons unaffected. 
They are mirror conditions for upward and downward consistency across resolutions.
\end{remark}

\begin{remark}
Axiom (O\ref{axiom:lower-scale-consistency}) is a special case of (O\ref{axiom:linear-constraint-consistency}) where \(f\) has value \(1\) on the whole space.
\end{remark}

To extend the finite-space results to general measurable spaces, we must ensure that orderings defined on all finite \(\sigma\)-algebra \(\mathcal{A}\in \mathcal{A}_\nu\) of $(E,\mathcal{E})$
agree in the limit.  
This is achieved through a closure construction that aggregates comparisons across increasingly fine partitions.
\begin{definition}[Closure of Preorder]\label{def:closure-preordering}
We define the closure \(\closure{\preceq_{\mathscr{A}_{\nu},q}}\) of the family of measure-consistent preorders  \(\preceq_{\mathscr{A}_{\nu},q}\defeq (\preceq_{(E,\mathcal{A},q\d{\nu})})_{\mathcal{A}\in\mathscr{A}_{\nu}}\) as the preorder on the set \(S\subseteq L^1_+(\nu)\)
such that for all \(p\in S\) there exists a strictly increasing sequence \((\mathcal{A}_{p,k})_{k\in \mathbb{N}} \subseteq \mathscr{A}_{\nu}\) such that, for all other \(p^\prime\in S\),
\begin{align}\label{eq:closure-definition-equivalence}
    p \; \closure{\preceq_{\mathscr{A}_{\nu},q}}\; p^\prime \iff \left\{
\begin{array}{ll}
\forall (\mathcal{B}_k)_{k\in\mathbb{N}}\subseteq \mathscr{A}_{\nu}  & \text{ s.t. } \quad \; \mathcal{A}_{p,k} \subseteq \mathcal{B}_k \,, \text{ and }  \mathcal{A}_{p^\prime,k} \subseteq \mathcal{B}_k \,,\\&\\
\exists K \text{ s.t. } \forall k \geq K \,,& p \preceq_{(E,\mathcal{B}_k,q\d{\nu})} p^\prime \,.\\
\end{array}
\right.
\end{align}
\(\closure{\preceq_{\mathscr{A}_{\nu},q}}\) is a preorder on \(S\) but not necessarily a total preorder.
\end{definition}

The closure is unique where defined: any two preorders respecting the property of equation~\eqref{eq:closure-definition-equivalence} agree on every pair of densities they both compare (Lemma~\ref{lem:closure-preordering-well-defined}, Appendix~\ref{sec:convergence-proof}).
Its existence is not asserted in general: the closure is constructed explicitly wherever it is used, on the integrable densities by Lemma~\ref{lem:general-f-divergence-sup-of-discrete-f-divergence} and on the remaining densities by Lemma~\ref{lem:not-integrable-order}.

\begin{theorem}[Divergence Representation of Preorders]\label{th:ordering-representation-as-divergence}
    Let \((E,\mathcal{E},\nu)\) be a \(\sigma\)-finite space and let
    \(\mathcal{D}=\{\preceq_{(E,\mathcal{E},q\d{\nu})} : q\in L^1_+(\nu)\}\) be a measure-consistent
    preordering family whose finite coarse-grainings are coherent with their closure
    \(\closure{\preceq_{\mathscr{A}_{\nu},q}}\) (Definition~\ref{def:closure-preordering}).
    \begin{enumerate}
        \item[(i)] If \(\mathcal{D}\) follows (O\ref{axiom:independence-subspace})--(O\ref{axiom:upper-scale-consistency}),
        then \(\mathcal{D}\) is represented by an \(f\)-divergence: there is a continuous strictly convex
        \(f\) on \(]0,\infty[\), unique up to a positive affine transformation, such that
        \begin{align}\label{eq:main-f-divergence}
            D_q(p) = \int_E p(x)\, f\!\left(\frac{q(x)}{p(x)}\right) \d{\nu(x)}
        \end{align}
        induces \(\preceq_{(E,\mathcal{E},q\d{\nu})}\) for all \(q\in L^1_+(\nu)\), the integrand read by recession where \(p\) vanishes (Appendix~\ref{subsec:divergence-representation}).
        \item[(ii)] If \(\mathcal{D}\) additionally follows (O\ref{axiom:lower-scale-consistency}),
        then \(f\) is a power function and \(D_q\) is an \(\alpha\)-divergence.
        \item[(iii)] If \(\mathcal{D}\) additionally follows (O\ref{axiom:linear-constraint-consistency}),
        then \(f=-\log\) and \(D_q\) is the Kullback--Leibler divergence.
    \end{enumerate}
\end{theorem}

\begin{proof}[Proof sketch]
    The full argument is in Appendix~\ref{sec:axiomatic-finite-spaces}, proved first on finite spaces
    and then lifted to \((E,\mathcal{E},\nu)\) through the ordering closure (Appendix~\ref{sec:convergence-proof}).
    On a finite space, Locality (O\ref{axiom:independence-subspace}) makes the preorder additive across
    atoms by the cardinal representation theorem of \citet{debreuTopologicalMethodsCardinal1959},
    giving \(D_{(\Omega,\mathcal{A},Q)}(P)=\sum_a d(Q(a),P(a))\) (Lemma~\ref{lem:d-divergence-axioms}).
    Reference Measure (O\ref{axiom:reference-measure-consistency}) and Upper Scale Consistency
    (O\ref{axiom:upper-scale-consistency}) reduce \(d\) to \(d(up,p)=p\,f(u)\) with \(f\) strictly convex,
    yielding the \(f\)-divergence form~\eqref{eq:main-f-divergence}
    (Lemma~\ref{lem:help-f-divergence-proof}, Theorem~\ref{prop:f-divergence-axioms}).
    Adding Lower Scale Consistency (O\ref{axiom:lower-scale-consistency}) forces \(f\) to satisfy the
    scaling functional equation \(f(xu)=a(u)f(x)+b(u)\), whose only strictly convex solutions are powers
    and \(-\log\) (Lemma~\ref{lem:functional-equation-convexity}, Theorem~\ref{prop:reyni-divergence-axioms}).
    Adding Linear Map Invariance (O\ref{axiom:linear-constraint-consistency}) requires the minimiser
    under several moment constraints to keep the same prior across reweightings; among \(\alpha\)-divergences
    only KL does so, its minimiser being the exponential family \(\exp(\sum_i\lambda_i f_i)\cdot Q\)
    (Theorem~\ref{theo:kl-divergence-axioms}, via \cite[Th.~4.8]{borweinDualityRelationshipsEntropyLike1991}).
    The lift from the finite coarse-grainings to \(\preceq_{(E,\mathcal{E},q\d{\nu})}\) is
    Theorem~\ref{theo:general-f-divergence-axioms}.
\end{proof}

\subsection{Inference: axioms and representation}\label{subsec:inference-axioms-presentation}

\medskip
\noindent\textbf{Information set.}
An \emph{information set} is any subset $I \subseteq \mathcal{M}(\Omega,\mathcal{A})$
used as a constraint for an inference operator.

Information is not merely a list of admissible information sets: an inference theory must also be able to combine pieces of information (this \emph{and} that), to weigh alternatives (this \emph{or} that), and to rule them out (\emph{not} this).
These logical operations are exactly what the \(\sigma\)-algebra structure introduced below makes precise.

\medskip
\noindent\textbf{Admissible Information sets.}
An \emph{admissible information set} \(\mathcal{I}_{(\Omega,\mathcal A)}\) is a \(\sigma\)-algebra of information sets on \(\mathcal{M}(\Omega,\mathcal A)\).
The \(\sigma\)-algebra structure enforces the set equivalent of the logical operators \texttt{AND}, \texttt{OR}, \texttt{NOT}.

\begin{definition}[Admissible Constraints on Finite Discrete Space]\label{def:admissible-constraints-finite-discrete}
The admissible constraints \(\mathcal{I}_{(\Omega,\mathcal{A})}\) for any inference operator on a finite discrete space \((\Omega,\mathcal{A})\) are the sets \(I\subseteq\mathcal{M}(\Omega,\mathcal{A})\) such that \(\ev_{\star}(I)\) is a Borel set of \(\mathbb{R}^{\card{\mathrm{At}(\mathcal{A})}}\), where \(\ev\) is the evaluation map
\begin{align}\label{eq:mapping-ev}
    \ev\,\colon\, P\mapsto \bigl( P(a)\bigr)_{a\in\mathrm{At}(\mathcal{A})}\,,
\end{align}
and \(\ev_{\star}\) denotes the induced map on subsets of \(\mathcal{M}(\Omega,\mathcal{A})\).
\end{definition}

Since \(\mathcal{B}(\mathbb{R}^{\card{\mathrm{At}(\mathcal{A})}})\) is a \(\sigma\)-algebra and \(\ev\) is measurable, \(\mathcal{I}_{(\Omega,\mathcal{A})}\) is a \(\sigma\)-algebra on \(\mathcal{M}(\Omega,\mathcal{A})\), so Definition~\ref{def:admissible-constraints-finite-discrete} is correct.

\begin{definition}[Inference Operator by Preorder]\label{def:inference-operator-by-preorder}
Let $(\Omega,\mathcal A)$ be a measurable space and let $\mathcal{I}_{(\Omega,\mathcal A)}$ be
a \(\sigma\)-algebra on $\mathcal{M}(\Omega,\mathcal A)$.
An \emph{inference operator} is a map
\begin{align*}
T : \mathcal{I}_{(\Omega,\mathcal A)} \to 2^{\mathcal{M}(\Omega,\mathcal A)}
\end{align*}
such that $T(I)\subseteq I$ for every $I\in \mathcal{I}_{(\Omega,\mathcal A)}$: the inferred measures satisfy the given information.

We say that $T$ is an \emph{inference operator by preorder} if there exists
a total preorder $\preceq_T$ on $\mathcal{M}(\Omega,\mathcal A)$ such that
\begin{align*}
T(I)
= \Min_{\preceq_T}(I)
:= \{\,P\in I : \nexists\, P' \in I \text{ with } P' \prec_T P \,\},
\qquad \forall I\in\mathcal{I}_{(\Omega,\mathcal A)}.
\end{align*}
\end{definition}

\begin{remark}\label{rem:set-valued-inference}
Operations on measures apply to sets of measures elementwise, e.g.\ \(T(I)|_{\mathcal{A}'} = \{P|_{\mathcal{A}'} : P\in T(I)\}\); equalities between inferred sets are set identities.
\(T(I)\) may be empty: a preorder need not attain a least element on \(I\) (typically \(\ev_{\star}(I)\) open), and no operator represented by a continuous divergence avoids this on all of \(\mathcal{I}_{(\Omega,\mathcal A)}\).
Under the representation of Theorem~\ref{theo:inference-representation} on a finite space, \(T(I)\neq\emptyset\) whenever \(\ev_{\star}(I)\) is compact, and \(T(I)\) is a singleton when \(\ev_{\star}(I)\) is moreover convex.
\end{remark}

We deduce an axiomatic for inference by divergence minimisation by mirroring the ordering axioms on the divergence, namely translating ordering axioms into axioms on the minimal elements of well-chosen subsets of \(\mathcal{M}(\Omega,\mathcal{A})\).
We now introduce these key classes of subsets.

\paragraph*{Information classes.}
For the inference axioms, we make use of the following canonical
information classes, each capturing a different way in which information
can restrict a measure.

\medskip
\noindent\textbf{(1) Subset–restricted information.}
For $A\in\mathcal A$ we define
\[
\mathcal I_A
:= \Big\{\, I\subseteq \mathcal M(\Omega,\mathcal A) :
   P\in I,\; P'|_A = P|_A \;\Rightarrow\; P'\in I \,\Big\}.
\]
Thus $\mathcal I_A$ consists of all information that constrains the
behaviour of a measure only on the set $A$.

\medskip
\noindent\textbf{(2) Coarse–grain information.}
For a finite sub–$\sigma$‑algebra $\mathcal A'\subseteq\mathcal A$ we define
\[
\mathcal I_{\mathcal A'}
:= \Big\{\, I\subseteq \mathcal M(\Omega,\mathcal A) :
   P\in I,\; P'|_{\mathcal A'} = P|_{\mathcal A'} 
   \;\Rightarrow\; P'\in I \,\Big\}.
\]
This captures constraints expressed only at the resolution $\mathcal A'$,
i.e.\ constraints on the coarse‑grained description $P|_{\mathcal A'}$.

\medskip
\noindent\textbf{(3) fine-grain distribution information.}
For a finite sub–$\sigma$‑algebra $\mathcal A'=\{a_i\}$ we define
\[
\mathcal{I}^{\mathcal{A}^\prime}
\defeq \Big\{\, I\subseteq\mathcal M(\Omega,\mathcal A) :
   P\in I,\; P^{\prime\,,\,\mathcal{A}^\prime} = P^{\mathcal{A}^\prime}
   \Rightarrow P'\in I \,\Big\}.
\]
These constrain only the conditional distributions inside atoms of 
$\mathcal A'$, without restricting coarse masses $P(a_i)$.

\medskip
\noindent\textbf{(4) Linear‑transformation information.}
For a strictly positive measurable function $f>0$ on $(\Omega,\mathcal A)$ and an information set $I\subseteq\mathcal M(\Omega,\mathcal A)$, we define
\[
f\cdot I \defeq \{\, f\cdot P \,:\, P\in I\,\}\subseteq \mathcal M(\Omega,\mathcal A)\,.
\]
This captures how an information set transforms when the underlying measures are reweighted by $f$.
Moment constraints are then expressed similarly to the total mass constraint by
    \begin{align*}
        f^{-1}\cdot I = \{ P \,:\, f\cdot P (\Omega) \in \ev(I) \} = \{ P \,:\,  P \in \cup_{c \in \ev(I)}\mathcal{F}_f(c) \} \,,\quad \forall\,I\in\mathcal{I}_{\{\Omega,\emptyset\}}\,,
    \end{align*} 
and a \(f\)-distribution constraint on the \(f\)-weighted measure is expressed by
\begin{align*}
    f^{-1} \cdot I  \,,
    \qquad \forall\,I\in\mathcal{I}^{\{\Omega,\emptyset\}}\,,
\end{align*}
and the \(\sigma\)-algebra of \(f\)-distribution constraints denoted \(\mathcal{I}^{f}/\!\sim\), where \(\mathcal{I}^{f}\defeq\{\, f^{-1}\cdot I : I\in\mathcal{I}\,\}\) is the information \(\sigma\)-algebra transported by \(f\).

\medskip
\paragraph*{Linear-transformation inference.}
A measurable function \(f > 0\) and an inference operator \(T\) on \((\Omega,\mathcal{A})\) induce the inference operator \(T^{f}\) defined by
    \begin{align}
        T^{f}(I) \defeq T(f^{-1}\cdot I) \,,\quad \forall\, I\in\mathcal{I}_{(\Omega,\mathcal{A})}\,.
    \end{align}
The superscript \(f\) records that \(T^{f}\) performs the inference in the measure space linearly transformed by \(f\), where the inference itself stays invariant, and the inverse \(f^{-1}\) carries the constraint back so that the result is expressed in the initial space.

The following axioms enforce coherent processing of each class of constraints (1)--(4).
We index the inference operator \(T_{(\Omega, \mathcal{A}, Q)}\) by the measurable space \((\Omega, \mathcal{A})\) on which it is defined and the prior \(Q\).

\begin{axiomI}[Isolated System]\label{axiom:isolated-system}
For all \(T_{(\Omega,\mathcal{A}, Q)}\in\mathcal{T}\) ,  \(A\in \mathcal{A}\), 
    \begin{align}\label{item:isolated-system-axiom}
        T_{(\Omega, \mathcal{A}, Q)}(I_A \cap I_{A^\complement}) = T_{(A,\mathcal{A}|_A, Q|_A)}(I_A) \oplus T_{(A^\complement,\mathcal{A}|_{A^\complement}, Q|_{A^\complement})}( I_{A^\complement}) \,, \quad \forall \, I_A \in \mathcal{I}_A , \, I_{A^\complement} \in \mathcal{I}_{A^\complement} \,.
    \end{align}
\end{axiomI}

\begin{axiomI}[Prior Consistent]\label{axiom:prior}
     \begin{align}\label{item:upper-scale-does-not-influence-lower-scale}
        T_{(\Omega,\mathcal{A},Q)}(I)^{\mathcal{A}^{\prime}} &= T_{(\Omega,\mathcal{A},Q)}(I^\prime)^{\mathcal{A}^{\prime}} \,, \qquad\; \forall\, I,I^\prime \in \mathcal{I}_{\mathcal{A}^\prime} , \, \mathcal{A}^{\prime} \subseteq \mathcal{A} \,.
    \end{align}
\end{axiomI}

\begin{axiomI}[Coarse-grain Consistent]\label{axiom:lower-scale-conservation}
     \begin{align}\label{item:lower-scale-conservation}
        T_{(\Omega,\mathcal{A},Q)}(I)|_{\mathcal{A}^{\prime}} = T_{(\Omega,\mathcal{A}^\prime,Q|_{\mathcal{A}^\prime})}(I)  \,, \qquad\; \forall\, I \in \mathcal{I}_{\mathcal{A}^\prime} , \, \mathcal{A}^{\prime} \subseteq \mathcal{A} \,.
    \end{align}
\end{axiomI}

\begin{axiomI}[Fine-grain Consistent]\label{axiom:upper-scale-conservation}
        \begin{align}\label{item:upper-scale-conservation}
            T_{(\Omega, \mathcal{A})}\left(I^{\mathcal{A}^\prime} \cap I|_{ \mathcal{A}^\prime}\right)^{\mathcal{A}^\prime} = T_{(\Omega,\mathcal{A})}\left(I^{\mathcal{A}^\prime}\cap I'|_{ \mathcal{A}^\prime}\right)^{\mathcal{A}^\prime} \,, \quad \forall \, I^{\mathcal{A}^\prime} \in \mathcal{I}^{\mathcal{A}^\prime} \,, \; I|_{ \mathcal{A}^\prime} \in \mathcal{I}_{\mathcal{A}^\prime} \,, I'|_{ \mathcal{A}^\prime} \in \mathcal{I}_{\mathcal{A}^\prime} \,,
        \end{align}
    inducing an inference operator \([T_{(\Omega,\mathcal{A})}]\) on distributions \(\mathcal{F}_{\boldsymbol{1}}\) with admissible information sets \(\mathcal{I}/\!\sim\defeq \{[I] : I = I_m\cap I^d \,,\; I_m \in \mathcal{I}_{\{\Omega,\emptyset\}}\,,\; I^d\in \mathcal{I}^{\{\Omega,\emptyset\}} \} \).
\end{axiomI}

\begin{axiomI}[Linear Transformation Consistent]\label{axiom:linear-constraint-independance}
    For any measurable function \(f> 0\), \(T^{f}_{(\Omega,\mathcal{A})}\) follows (I\ref{axiom:isolated-system})--(I\ref{axiom:upper-scale-conservation}) with different priors, and
    \begin{align}\label{item:linear-constraint-coherence}
        [T_{(\Omega,\mathcal{A})}]_{f_1}\left(I\right) = \cdots   = [T_{(\Omega,\mathcal{A})}]_{f_n}(I) \,,\quad \forall\, I\in \mathcal{I}^{f_1}/\!\sim\cap\cdots\cap \mathcal{I}^{f_n}/\!\sim 
    \end{align}
    with \(n>1\), inducing an inference operator \([T_{(\Omega,\mathcal{A})}]_{f_1,\cdots,f_n}\) on \(\mathcal{F}_{f_1}\cap\cdots\cap \mathcal{F}_{f_n}\) with admissible information sets \(\mathcal{I}^{f_1}/\!\sim \cap \cdots \cap \mathcal{I}^{f_n}/\!\sim \).
\end{axiomI}

\begin{remark}
    The prior relations appearing in these axioms are not free choices.
    Which prior indexes each operator, the restrictions \(Q|_A\) and \(Q|_{\mathcal{A}^\prime}\) in (I\ref{axiom:isolated-system}) and (I\ref{axiom:lower-scale-conservation}) as much as the induced prior of the linear case, is itself derived in Section~\ref{sec:category-derivation}, each forced by requiring the inference operator to process information invariantly under equivalent reformulations of the same problem (functoriality).
\end{remark}

\begin{theorem}[Inference Representation]\label{theo:inference-representation}
    Let \((E,\mathcal{E},\nu)\) be a \(\sigma\)-finite space and let \(\mathcal{T}\) be an expressive,
    continuous family of inference operators (Definition~\ref{def:expressive-family-inference})
    on \((E,\mathcal{E},\nu)\), coherent with its finite discretisations.
    \begin{enumerate}
        \item[(i)] If \(\mathcal{T}\) follows (I\ref{axiom:isolated-system})--(I\ref{axiom:lower-scale-conservation}),
        then \(\mathcal{T}\) is inference by \(f\)-divergence minimisation: there is a continuous strictly
        convex \(f\), unique up to a positive affine transformation, with
        \(T_{(E,\mathcal{E},q\d{\nu})}(I)=\argmin_{p\in I} D_q(p)\) and \(D_q\) as in~\eqref{eq:main-f-divergence}.
        \item[(ii)] If \(\mathcal{T}\) additionally follows (I\ref{axiom:upper-scale-conservation}),
        then \(D_q\) is an \(\alpha\)-divergence.
        \item[(iii)] If \(\mathcal{T}\) additionally follows (I\ref{axiom:linear-constraint-independance}),
        then \(D_q\) is the Kullback--Leibler divergence.
    \end{enumerate}
\end{theorem}

\begin{proof}[Proof sketch]
    The full argument is in Appendix~\ref{subsec:inference-representation}, proved on finite spaces and
    lifted to \((E,\mathcal{E},\nu)\) by Theorem~\ref{theo:inference-on-general-space} and the closure of Definition~\ref{def:closure-preordering}).
    On a finite space, Isolated System (I\ref{axiom:isolated-system}) makes the induced preorder
    \(\preceq_{T}\) additive across atoms (Lemma~\ref{lem:isolate-system-axioms}); Prior Consistency
    (I\ref{axiom:prior}) and Coarse-grain Consistency (I\ref{axiom:lower-scale-conservation}) make
    \(\preceq_{T}\) satisfy the ordering axioms (O\ref{axiom:independence-subspace})--(O\ref{axiom:upper-scale-consistency}),
    so the operator minimises an \(f\)-divergence (Proposition~\ref{theo:MaxEnt-by-f-divergence-axioms},
    via Theorem~\ref{prop:f-divergence-axioms}).
    Fine-grain Consistency (I\ref{axiom:upper-scale-conservation}) supplies Lower Scale Consistency
    (O\ref{axiom:lower-scale-consistency}), forcing the \(\alpha\)-divergence
    (Theorem~\ref{theo:MaxEnt-by-reyni-divergence-axioms}); Linear Transformation Consistency
    (I\ref{axiom:linear-constraint-independance}) supplies Linear Map Invariance
    (O\ref{axiom:linear-constraint-consistency}), isolating KL
    (Theorem~\ref{theo:MaxEnt-by-kl-divergence-axioms}).
\end{proof}

\begin{corollary}[Bayesian Characterisation of \(f\)-Divergence Inference]\label{cor:bayes-identification}
    Let \(\mathcal{T}\) be a family of inference operators on finite spaces, following the technical assumptions of Appendix~\ref{subsec:inference-representation}, Isolated System (I\ref{axiom:isolated-system}) and Prior Consistency (I\ref{axiom:prior}).
    The following are equivalent:
    \begin{enumerate}
        \item[(i)] \(\mathcal{T}\) conditions Bayesianly: for every prior \(Q\) and every \(A\in\mathcal{A}\) with \(Q(A)>0\), on the constraint that all mass lies in \(A\) (the probability measures supported on \(A\)), the operator returns the prior conditional \(Q(\,\cdot\mid A)=Q|_A/Q(A)\) on \((A,\mathcal{A}|_A)\);
        \item[(ii)] \(\mathcal{T}\) is inference by \(f\)-divergence minimisation against its prior \(Q\).
    \end{enumerate}
    In particular, \(f\)-divergence minimisation is the unique additive-divergence inference compatible with Bayesian conditioning, recovering Cox's plausibility programme \cite{coxAlgebraProbableInference2001,coxProbabilityFrequencyReasonable1946,vanhornConstructingLogicPlausible2003} within the additive-divergence class.
\end{corollary}
\begin{proof}[Proof sketch]
Both directions are proved in Appendix~\ref{sec:Bayes}.
\end{proof}

\section{Geometry of Inference Theory}\label{sec:category-derivation}

\citet[Preface]{cencov1982} writes that "\textit{In any theory, a general law should be amenable to equivalent formulations}", a principle he formalises in his book through a category-theoretic development of statistical inference.
Statistical problems parametrised by $\Theta$, with Markov kernels transporting one problem into another, form the Markov geometry of \citet[Sec.~6]{cencov1982} under which risk functions are hereditarily invariant and decision rules are covariant.
Two objects are congruent (denoted \(A\leftrightarrow B\)) when they represent two formulations of the same object. 
Equivalence of statistical problems in the Markov geometry is captured by the following definition.

\begin{definition}[Equivalent Statistical Problem]\label{def:equivalent-statistical-problem}
Two statistical problems \(\{\mathbb{P}_{\theta}:\theta\in\Theta\}\) on \((\Omega,\mathcal{A})\) and \(\{\mathbb{P}'_{\theta}:\theta\in\Theta\}\) on \((\Omega',\mathcal{A}')\) are equivalent if and only if
there is a Markov kernel \(\Pi\) such that \(\mathbb{P}'_{\theta} = \mathbb{P}_{\theta} \Pi\) for all \(\theta\in\Theta\) and a Markov kernel \(\Pi'\) such that \(\mathbb{P}_{\theta} = \mathbb{P}'_{\theta} \Pi'\) for all \(\theta\in\Theta\).
\end{definition}

Our axiomatic follows the same principle: an inference theory must process information invariantly under equivalent reformulations of the same inference problem.
We take \(\mathfrak{M}\) to be the finite-measurable-space case of the category of statistical models of \citet{cencov1982}, and unfold its underlying construction, a step \citet{cencov1982} leaves implicit.
Reapplying the same construction to admissible information sets defines the category of equivalent information \(\mathfrak{I}\), the information geometry.
Reformulating an inference problem by relabelling, splitting, or coarsening (Section~\ref{sec:main-result}) is, formally, a measurable map between measurable spaces; these maps are the morphisms underlying \(\mathfrak{I}\) and \(\mathfrak{M}\).
The main result of the section shows that requiring the inference operator \(\mathcal{T}\) to be a covariant functor from \(\mathfrak{I}\) to \(\mathfrak{M}\) is equivalent to the consistency axioms of Section~\ref{subsec:inference-axioms-presentation}.
Inference by canonical divergence minimisation is therefore the unique inference operator invariant under equivalent reformulations. It is not an extra postulate added to statistical decision theory: it shares the same theoretical roots as statistical decision rules in \v{C}encov's geometry.

Throughout the section, \(\Omega\) is finite.
The objects of \(\mathfrak{M}\) are families of finite positive measures on \((\Omega,\mathcal{A})\) with Markov kernels as morphisms; the objects of \(\mathfrak{I}\) are admissible information sets \(I\subseteq\mathcal{M}(\Omega,\mathcal{A})\) with measurable transport maps as morphisms.
Figure~\ref{fig:cateory-information} depicts the functor \(\mathcal{T}\) connecting them.

\begin{figure}[htbp]
    \centering
    \[
    \begin{tikzcd}[row sep=large, column sep=normal]
    \mathfrak{I}  \arrow[r, phantom, ":"]  & (\Omega,\mathcal{A}),I  \arrow[r, "t_{\#}"] \arrow[d, "T_{(\Omega,\mathcal{A})}"'] & (\Omega', \mathcal{A}'),I' \arrow[d, "T_{(\Omega', \mathcal{A}')}"']\\
    \mathfrak{M}  \arrow[r, phantom, ":"]  & (\Omega,\mathcal{A}),\{P_{\theta}\} \arrow[r, "\Pi"]  & (\Omega',\mathcal{A}'),\{P_{\theta}'\}
    \end{tikzcd}
    \]

    \caption{Action of the inference operator \(\mathcal{T}\).
    A functor preserves the category structure between its source and target, encoding that the inference results transport consistently along reformulations of the underlying space that preserve the information.
    The inference operator \(T_{(\Omega,\mathcal{A})}\) sends each information \((\Omega,\mathcal{A}),I\) of \(\mathfrak{I}\) to a family of measures \((\Omega,\mathcal{A}),\{P_\theta\}\) of \(\mathfrak{M}\), and lifts each information morphism \(t\pf\) to a Markov kernel \(\Pi\) so that the square commutes, creating a functor.}
    \label{fig:cateory-information}
\end{figure}
  
\subsection{Statistical models category}\label{subsec:statis-model-category}

A reformulation of an inference problem is a measurable map \(t:(\Omega,\mathcal{A})\to(\Omega',\mathcal{A}')\) between the measurable spaces on which the problem is posed. 
When this reformulation is invertible through a map \(t':(\Omega',\mathcal{A}')\to (\Omega,\mathcal{A})\), the two formulations are equivalent.
The category of statistical models of \v{C}encov~\cite[Ch.~1]{cencov1982} records how such a reformulation acts on the measures: it associates to \(t\) a Markov kernel \(\Pi\) carrying a family of measures \(\{P_\theta\}\) from one space to the other.

\begin{lemma}[Generation of Measurable Maps on Finite Measurable Spaces]\label{lem:meas-factorisation}
    Every measurable map \(t:(\Omega,\mathcal{A})\to(\Omega',\mathcal{A}')\) between finite measurable spaces factors as
    \begin{align}
        t \;=\; t_{\mathrm{emb}} \circ t_{\mathrm{rel}} \circ t_{\mathrm{cg}} \,,
    \end{align}
    where \(t_{\mathrm{cg}}\) is a coarsening, \(t_{\mathrm{rel}}\) a relabelling, and \(t_{\mathrm{emb}}\) a subset embedding onto \(t(\Omega)\), in the senses of Recall~\ref{recall:congruent_measures}. Here "relabelling" is meant in the atom-bijection sense of Recall~\ref{recall:congruent_measures}: \(t_{\mathrm{rel}}\) induces a bijection between the atoms of its source and target \(\sigma\)-algebras, but is not required to be a bijection on points. 
\end{lemma}
\begin{proof}
    Set \(\mathcal{A}_0\defeq t^{-1}(\mathcal{A}')\subseteq\mathcal{A}\). 
    The identity-on-points map \(t_{\mathrm{cg}}\defeq\id_\Omega:(\Omega,\mathcal{A})\to(\Omega,\mathcal{A}_0)\) is a coarsening. 
    Atoms of \(\mathcal{A}_0\) are exactly the fibres of \(t\) and biject with \(t(\Omega)\), giving an atom-bijection \(t_{\mathrm{rel}}:\mathrm{At}(\mathcal{A}_0)\leftrightarrow\mathrm{At}(\mathcal{A}'\vert_{t(\Omega)})\). 
    Finally \(t(\Omega)\in\mathcal{A}'\) on a finite space, so \(t_{\mathrm{emb}}:(t(\Omega),\mathcal{A}'\vert_{t(\Omega)})\hookrightarrow(\Omega',\mathcal{A}')\) is a subset embedding. 
\end{proof}

We extend the Markov geometry of statistical problems to finite positive measures, creating a category \(\mathfrak{M}\) by replacing probability measures with finite positive measures in Definition~\ref{def:equivalent-statistical-problem}.

\begin{definition}[Category $\mathfrak{M}$] The category \(\mathfrak{M}\) has for objects \((\Omega,\mathcal{A},\{P\})_{\card{\Omega}<\infty,\,\mathcal{A}\subseteq \sigma(\Omega),\{P\} \subseteq \mathcal{M}(\Omega,\mathcal{A})}\).
A Markov kernel \(\Pi\) from \((\Omega,\mathcal{A})\) to \((\Omega',\mathcal{A}')\) is a morphism from \(((\Omega,\mathcal{A}),\{P\})\) to \(((\Omega',\mathcal{A}'),\{P'\})\) if and only if there is \(\Pi'\) such that \(((\Omega,\mathcal{A}),\{P\}) \leftrightarrow_{(\Pi,\Pi')} ((\Omega',\mathcal{A}'),\{P'\})\) in the sense of Definition~\ref{def:equivalent-statistical-problem}.
\end{definition}

\begin{recall}[Congruent families of measures: \cite{cencov1982} Sect. 9] \label{recall:congruent_measures}

Let \(((\Omega,\mathcal{A},Q),\{P_\theta : \theta \in \Theta\}) \leftrightarrow_{(\Pi,\Pi')} ((\Omega,\mathcal{A}',Q'),\{P'_\theta : \theta \in \Theta\})\)  in \(\mathfrak{M}\). 
Then, there exists \(\mathcal{A}_{\Pi}\subseteq \mathcal{A}\), \(\mathcal{A}_{\Pi'}\subseteq \mathcal{A}'\)  with $ \card{\mathrm{At}(\mathcal{A}_{\Pi})}= \card{\mathrm{At}(\mathcal{A}_{\Pi'})}$, $Q \in \mathcal{M}(\Omega, \mathcal{A})$ and $Q' \in \mathcal{M}(\Omega', \mathcal{A}')$ such that:
    \begin{align}
        P_\theta = P_\theta|_{\mathcal{A}_{\Pi}} \wedge Q^{\mathcal{A}_{\Pi}}, \quad P'_\theta = P'_\theta|_{\mathcal{A}_{\Pi'}}  \wedge Q'^{\mathcal{A}_{\Pi'}} \, \text{ and } P_\theta|_{\mathcal{A}_{\Pi}}(a) = P_\theta \Pi(\phi(a)) =  P'_\theta|_{\mathcal{A}_{\Pi'}}(\phi(a)) \,,\quad \forall \theta \in \Theta \,, \label{eq:A1}
    \end{align}
    with \(\phi: \mathrm{At}(\mathcal{A}_{\Pi}) \mapsto \mathrm{At}(\mathcal{A}_{\Pi'})\) a bijection.

These morphisms are generated by three elementary families of morphisms:
\begin{itemize}
    \item \emph{Relabelling}, \(\Pi : \mathcal{M}(\Omega',\mathcal{A}') \mapsto  \mathcal{M}(\Omega,\mathcal{A})\) with \(\card{\mathrm{At}(\mathcal{A})} = \card{\mathrm{At}(\mathcal{A}')}\) and  \(P\Pi(a) =  P(\phi_{\Pi}(a))\) for all \(a\in\mathrm{At}(\mathcal{A})\).
    \item \emph{Subset embedding},  \(\Pi  : \mathcal{M}(A,\mathcal{A}|_A) \hookrightarrow  \mathcal{M}(\Omega,\mathcal{A})\) with \(A \in\mathcal{A}\), and \(P\Pi =  P\oplus \boldsymbol{O}|_{A^{\complement}}\) for all \(P\in \mathcal{M}(A,\mathcal{A}|_A)\).
    \item \emph{Coarsening}, \(\Pi : \mathcal{M}(\Omega,\mathcal{A}') \hookrightarrow  \mathcal{M}(\Omega,\mathcal{A})\) with \(\mathcal{A}' \subsetneq \mathcal{A}\), and \(P\Pi = P \wedge Q^{\mathcal{A}'}\) for a \(Q\in\mathcal{M}(\Omega,\mathcal{A})\).
\end{itemize}
with their respective duals (pseudo-inverses) \(\Pi'\).
\end{recall}
\begin{proof}
    The result is a rewriting of \cite[Th.~9.1]{cencov1982}.
\end{proof}

\emph{Notational convention.} When \(\Pi\) is viewed as a morphism of \(\mathfrak{M}\), we also write \(\Pi(P)\defeq P\Pi\) for a single measure and \(\Pi(S)\defeq \{P\Pi : P\in S\}\) for a set of measures \(S\subseteq\mathcal{M}(\Omega,\mathcal{A})\).

The generating family of the morphisms of \(\mathfrak{M}\) is induced by the generating family of measurable maps.
The measurable map is thus the primitive notion of reformulation, and the Markov-kernel morphism is the action it induces on a family of measures.

\begin{lemma}[From Measurable Map to Markov Morphism]\label{lem:markov-morphism-as-measurable-map}
    For every morphism \((\Pi,\Pi')\) of the generating family of \(\morph(\mathfrak{M})\), either \(\Pi:\mathcal{M}(\Omega,\mathcal{A})\to\mathcal{M}(\Omega',\mathcal{A}')\) is related to a measurable map \(t:(\Omega,\mathcal{A})\to(\Omega',\mathcal{A}')\) by
    \begin{align}
        P\Pi = t_{\star}P \,, \quad \forall P\in\mathcal{M}(\Omega,\mathcal{A}) \,,
    \end{align}
    where \(t_\star\) is the measure pushforward \((t_\star P)(B)\defeq P(t^{-1}(B))\),
    or \(\Pi'\) is related to a measurable map \(t':(\Omega',\mathcal{A}')\to(\Omega,\mathcal{A})\) the same way.
    We denote \((\Pi_t,\Pi'_{t})\) when the morphism is related to \(t\).
\end{lemma}
\begin{proof}
    The result is a direct implication of Recall~\ref{recall:congruent_measures} and Lemma~\ref{lem:meas-factorisation} when considering \(t_{\mathrm{rel}}\) for \(\Pi\) relabelling, \(t_{\mathrm{emb}}\) for \(\Pi\) subset embedding, and \(t_{\mathrm{cg}}\) for the dual \(\Pi'\) of a coarsening.
\end{proof}

The subset-embedding family of morphisms of \(\mathfrak{M}\) (Recall~\ref{recall:congruent_measures}) already uses the \emph{disjoint-union operator} \((\sqcup,\oplus)\), writing \(P\Pi = P\oplus\boldsymbol{O}|_{A^\complement}\) for the inclusion \(\mathcal{M}(A,\mathcal{A}|_A)\hookrightarrow\mathcal{M}(\Omega,\mathcal{A})\).
This makes \((\sqcup,\oplus)\) the canonical way to combine disjoint measurable spaces inside \(\mathfrak{M}\), with subset embeddings as the corresponding injections, and \(\mathfrak{M}\) therefore carries a canonical symmetric monoidal structure.

\begin{proposition}[Symmetric Monoidal Structure on $\mathfrak{M}$]\label{prop:monoidal-M}
    The category \(\mathfrak{M}\) carries a symmetric monoidal structure \((\mathfrak{M},\sqcup,\mathbf{1})\) defined by:
    \begin{itemize}
        \item \emph{Coproduct on objects.}
        \begin{align}\label{eq:monoidal-tensor-objects}
            ((\Omega_1,\mathcal{A}_1),\{P_{1,\theta}\}_{\theta\in\Theta}) \sqcup ((\Omega_2,\mathcal{A}_2),\{P_{2,\theta'}\}_{\theta'\in\Theta'}) \defeq ((\Omega_1\sqcup\Omega_2,\mathcal{A}_1\oplus\mathcal{A}_2),\{P_{1,\theta}\oplus P_{2,\theta'}\}_{(\theta,\theta')\in\Theta\times\Theta'}) \,.
        \end{align}
        \item \emph{Coproduct on morphisms.} For Markov kernels \(\Pi_i:(\Omega_i,\mathcal{A}_i)\to(\Omega'_i,\mathcal{A}'_i)\) in \(\mathfrak{M}\), \(\Pi_1\sqcup\Pi_2\) is the Markov kernel on \(\Omega_1\sqcup\Omega_2\) acting as \(\Pi_i\) on each component.
        \item \emph{Unit object.} \(\mathbf{1}\defeq((\emptyset,\{\emptyset\}),\{\boldsymbol{O}\})\), the empty measurable space with the singleton family carrying its unique (zero) measure.
        \item \emph{Associator, unitors, symmetry.} The canonical bijections of tagged disjoint unions, which are relabelling-isomorphisms of \(\mathfrak{M}\) (Recall~\ref{recall:congruent_measures}).
    \end{itemize}
\end{proposition}
\begin{proof}
    Each of the associator, unitor and symmetry bijections is a bijection between finite sets that preserves atoms of the disjoint-union \(\sigma\)-algebras, hence a relabelling (Recall~\ref{recall:congruent_measures}) and therefore an isomorphism of \(\mathfrak{M}\).
    The pentagon, triangle and hexagon coherence diagrams commute already at the level of tagged sets and of parameter sets (via the canonical bijections \(\Theta\times\Theta'\cong\Theta'\times\Theta\) and \((\Theta\times\Theta')\times\Theta''\cong\Theta\times(\Theta'\times\Theta'')\)), and the pushforward of measures along the disjoint-union bijections is the corresponding coordinate permutation of \(\oplus\)-decompositions, so they commute in \(\mathfrak{M}\) as well.
    For morphisms, \(\Pi_1\sqcup\Pi_2\) is block-diagonal in the atom decomposition of \(\mathcal{A}_1\oplus\mathcal{A}_2\), hence a Markov kernel; identities and composition are preserved component-wise, so \(\sqcup\) is a bifunctor.
    The unit axioms follow from the canonical bijection \(\Omega\sqcup\emptyset \cong \Omega\) and the identity \(P\oplus\boldsymbol{O}_\emptyset = P\) on \(\mathcal{M}(\Omega\sqcup\emptyset)\), together with the bijection \(\Theta\times\{*\}\cong\Theta\) on the parameter side.
\end{proof}

\subsection{Invariance of inference under \v{C}encov morphisms}\label{subsec:information-category-and-functor}

At the measurable space level, problem reformulation is encoded by measurable maps \(t : (\Omega,\mathcal{A}) \to (\Omega',\mathcal{A}')\), the same maps that organise statistical models in \(\mathfrak{M}\).
We carry these reformulation morphisms one layer up, from measurable space to information: each \(t\) induces a transport of admissible information sets from one measurable space to another. 
The category generated plays for inference the role that \v{C}encov's Markov category plays for statistical models.
This subsection constructs this information category \(\mathfrak{I}\), defines the inference functor \(\mathcal{T} : \mathfrak{I} \to \mathfrak{M}\), and shows that requiring functoriality is equivalent to the inference axioms of Section~\ref{subsec:inference-axioms-presentation}.

We identify three minimal requirements for \(\morph(\mathfrak{I})\).
First, as seen in Definition~\ref{def:admissible-constraints-finite-discrete}, \(\mathcal{I}_{(\Omega,\mathcal{A})}\) carries the logical operators \texttt{OR}, \texttt{AND}, \texttt{NOT} through the set operations \(\cup\), \(\cap\) and \(\complement\).
A morphism as a mapping must respect the logical operator \(\implies\) on its definition domain.
As \(\implies\) logically interacts with \texttt{OR}, \texttt{AND}, \texttt{NOT}, the mapping morphism must interact with \(\cup\), \(\cap\) and \(\complement\) the same way.
Second, the absence of information is always transported faithfully, meaning that it always maps to another absence of information.
Third, in inference theory, an information \(I\) on \(\Omega,\mathcal{A}\) is determined independently of the reference/prior measure \(Q\), so an information category must be defined on measurable spaces, not measured spaces.

\begin{definition}[Information Category]\label{def:category-of-information}
    An \emph{information category} is a category whose objects are admissible information sets \(I\in\mathcal{I}_{(\Omega,\mathcal{A})}\) and whose morphisms \(t_{\#}:\mathcal{I}_{(\Omega,\mathcal{A})}\to\mathcal{I}_{(\Omega',\mathcal{A}')}\) satisfy:
    \begin{enumerate}[label=(\roman*)]
        \item \emph{logical compatibility}: \(t_{\#}\) commutes with countable unions, intersections, and complement,
        \begin{align*}
            t_{\#}\!\bigl(\textstyle\bigcup_{i\in\mathbb{N}} I_i\bigr) = \textstyle\bigcup_{i\in\mathbb{N}} t_{\#}(I_i)\,,\quad t_{\#}\!\bigl(\textstyle\bigcap_{i\in\mathbb{N}} I_i\bigr) = \textstyle\bigcap_{i\in\mathbb{N}} t_{\#}(I_i)\,, \quad t_{\#}(I^{\complement}) = t_{\#}(I)^{\complement}\,.
        \end{align*}
        \item \emph{Preservation of the absence of information}: \(t_{\#}\bigl(\mathcal{M}(\Omega,\mathcal{A})\bigr)=\mathcal{M}(\Omega',\mathcal{A}')\).
    \end{enumerate}
    We omitted the measurable space when referring to objects for readability.
\end{definition}

\begin{lemma}[Induced Structure of Information Morphisms on Finite Spaces]\label{lem:info-morphism-induced-structure}
    Let \((\Omega,\mathcal{A})\), \((\Omega',\mathcal{A}')\) be finite measurable spaces and let \(t_{\#}:\mathcal{I}_{(\Omega,\mathcal{A})}\to\mathcal{I}_{(\Omega',\mathcal{A}')}\) be a morphism of an information category in the sense of Definition~\ref{def:category-of-information}.
    Then there exists a Borel measurable function \(f_{t\pf}:[0,\infty)^{\card{\mathrm{At}(\mathcal{A}')}}\setminus\{0\}\to[0,\infty)^{\card{\mathrm{At}(\mathcal{A})}}\setminus\{0\}\) such that
    \begin{align}\label{eq:cencov-push}
        t\pf I  = \ev_{(\Omega',\mathcal{A}')}^{-1}\!\bigl(f^{-1}_{t\pf}\bigl(\ev_{(\Omega,\mathcal{A})}(I)\bigr)\bigr) \,,\quad \forall I\in \mathcal{I}_{(\Omega,\mathcal{A})}\,,
    \end{align}
    where \(f_{t\pf}^{-1}(\cdot)\) and \(\ev_{(\Omega',\mathcal{A}')}^{-1}(\cdot)\) denote set-theoretic preimages.
\end{lemma}
\begin{proof}
    As the measures of \(\mathcal{M}(\Omega,\mathcal{A})\) are finite of positive mass, \(\ev_{(\Omega,\mathcal{A})}\) induces a bijection between \(\mathcal{I}_{(\Omega,\mathcal{A})}\) and the Borel \(\sigma\)-algebra of the punctured closed orthant \(\mathcal{B}\bigl([0,\infty)^{\card{\mathrm{At}(\mathcal{A})}}\setminus\{0\}\bigr)\), restricting to the open orthant \((0,\infty)^{\card{\mathrm{At}(\mathcal{A})}}\) on the strictly positive interior.
    Under these bijections, conditions (i) and (ii) make \(t_{\#}\) a Boolean \(\sigma\)-homomorphism from \(\mathcal{B}\bigl([0,\infty)^{\card{\mathrm{At}(\mathcal{A})}}\setminus\{0\}\bigr)\) to \(\mathcal{B}\bigl([0,\infty)^{\card{\mathrm{At}(\mathcal{A}')}}\setminus\{0\}\bigr)\).
    The punctured closed orthant is locally closed in \(\mathbb{R}^{\card{\mathrm{At}(\mathcal{A})}}\), hence a Borel subset, so it is an absolute Borel space.
    Sikorski's point-realisation theorem \cite[Th.~32.5]{booleanAlgebra1969}, applied with trivial \(\sigma\)-ideals, then realises this \(\sigma\)-homomorphism as the set preimage \(t\pf I = f_{t\pf}^{-1}\bigl(\ev_{(\Omega,\mathcal{A})}(I)\bigr)\) of a point mapping \(f_{t\pf}\); since \(f_{t\pf}^{-1}\) carries Borel sets to Borel sets, \(f_{t\pf}\) is Borel measurable, which is equation~\eqref{eq:cencov-push}.
\end{proof}

Inference must be invariant under equivalent reformulation, and reformulation is carried out by a measurable map.
The notion of information equivalence is then obtained from the isomorphisms generated by a measurable space map.

\begin{definition}[\v{C}encov Information Categories \(\mathfrak{I}\) and \(\mathfrak{I}_r\)]\label{def:cencov-category-of-information}
    The \emph{\v{C}encov information category with relabelling} \(\mathfrak{I}_r\) is the groupoid of the information category (Definition~\ref{def:category-of-information})
    restricted to the generation family of morphisms \((t\pf,t^{\#})\) such that there exists a transport \(t\) with
    \begin{subequations}\label{eq:transport-condition-for-info-morph}
    \begin{align}
        \ev_{(\Omega,\mathcal{A})} \;=\; f_{t_{\#}} \circ \ev_{(\Omega',\mathcal{A}')} \circ \, t_\star \,,& \text{ if } \mathbb{R}^{\card{\mathrm{At}(\mathcal{A})}} \leq \mathbb{R}^{\card{\mathrm{At}(\mathcal{A}')}}, \\
        f_{t^{\#}} \circ \ev_{(\Omega,\mathcal{A})} \;=\; \ev_{(\Omega',\mathcal{A}')} \circ \, t_\star\,, & \text{ if } \mathbb{R}^{\card{\mathrm{At}(\mathcal{A})'}} \leq \mathbb{R}^{\card{\mathrm{At}(\mathcal{A})}}\,.
    \end{align}
    \end{subequations}
    The \emph{\v{C}encov information category} \(\mathfrak{I}\) is the subgroupoid generated by the coarsening and subset-embedding transports (Theorem~\ref{theo:cencov-uniqueness}), that is \(\mathfrak{I}_r\) deprived of the relabelling transports.
\end{definition}

\begin{theorem}\label{theo:cencov-uniqueness}
    A morphism of the \v{C}encov information category with relabelling \(\mathfrak{I}_r\) is a pair \((t\pf,t^{\#})\) of arrows defined by equations~\eqref{eq:cencov-push} from a measurable map \(t:(\Omega,\mathcal{A})\to(\Omega',\mathcal{A}')\) and satisfying the mutual inversion conditions
    \begin{align}\label{eq:cencov-iso-conditions}
        t^{\#}\circ t\pf\, I \;=\; I \,,\qquad t\pf\circ t^{\#}\, I' \;=\; I' \,,
    \end{align}
    with \(I' = t\pf I\); we write \((\Omega,\mathcal{A},I)\leftrightarrow_t(\Omega',\mathcal{A}',I')\).
    Every such isomorphism is a finite composition of arrows drawn from the following three elementary families:
    \begin{itemize}
        \item \emph{Relabelling}: \(t_{\mathrm{rel}}:(\Omega,\mathcal{A})\to(\Omega',\mathcal{A}')\) is a measurable bijection (equivalently, \(\card{\mathrm{At}(\mathcal{A})}=\)\newline\(\card{\mathrm{At}(\mathcal{A}')}\)), \(f_{t_{\mathrm{rel}}}\) is the induced coordinate permutation, and \(I' = t_{\mathrm{rel}\#} I = \{t_{\mathrm{rel}\#}P : P\in I\}\).
        \item \emph{Subset embedding}: for \(A\in\mathcal{A}'\), the source \((\Omega,\mathcal{A}) = (A,\mathcal{A}'|_A)\), \(t_{\mathrm{emb},A}:(A,\mathcal{A}'|_A)\hookrightarrow(\Omega',\mathcal{A}')\) is the inclusion, \(f_{t_{\mathrm{emb},A}}\) is the whole-real-line extension on the coordinates of \(\mathrm{At}(\mathcal{A}'|_{A^\complement})\), and \(I' = t_{\mathrm{emb},A\,\#} I = \{P\in\mathcal{M}(\Omega',\mathcal{A}') : P|_A\in I\}\in\mathcal{I}_A\subset\mathcal{I}_{(\Omega',\mathcal{A}')}\).
        \item \emph{Coarsening}: for \(\mathcal{A}'\subseteq\mathcal{A}\) on a common finite set \(\Omega = \Omega'\), \(t_{\mathrm{cg},\mathcal{A}'}:(\Omega,\mathcal{A})\to(\Omega,\mathcal{A}')\) is the identity, \(f_{t_{\mathrm{cg},\mathcal{A}'}}\) is the partition-sum matrix on \(\mathrm{At}(\mathcal{A}')\), and \(I = t_{\mathrm{cg},\mathcal{A}'}^{\#} I' = \{P\in\mathcal{M}(\Omega,\mathcal{A}) : P|_{\mathcal{A}'}\in I'\}\in\mathcal{I}_{\mathcal{A}'}\subset\mathcal{I}_{(\Omega,\mathcal{A})}\).
    \end{itemize}
    The relabelling-free subgroupoid \(\mathfrak{I}\) (Definition~\ref{def:cencov-category-of-information}) is generated by the subset-embedding and coarsening families of isomorphisms.
\end{theorem}

\begin{proof}
    Equation~\eqref{eq:cencov-iso-conditions} is the expression of an isomorphism of category, which is automatic in any groupoid.
    By Lemma~\ref{lem:meas-factorisation}, every measurable map between finite measurable spaces factors as \(t = t_{\mathrm{emb}}\circ t_{\mathrm{rel}}\circ t_{\mathrm{cg}}\), so by functoriality it suffices to verify equations~\eqref{eq:transport-condition-for-info-morph} and exhibit the explicit \(f_t\) for each of the three families.
    In each case we work on the atomic basis \((e_a)_{a\in\mathrm{At}(\mathcal{A})}\) of \(\mathbb{R}^{\card{\mathrm{At}(\mathcal{A})}}\) on which \(\ev_{(\Omega,\mathcal{A})}(P) = \bigl(P(a)\bigr)_{a\in\mathrm{At}(\mathcal{A})}\) by Definition~\ref{def:admissible-constraints-finite-discrete}.

    \emph{Relabelling.}
    Let \(t_{\mathrm{rel}}:(\Omega,\mathcal{A})\to(\Omega',\mathcal{A}')\), and write \(\bar t:\mathrm{At}(\mathcal{A})\to\mathrm{At}(\mathcal{A}')\) for the induced atom bijection.
    Define \(f_{t_{\mathrm{rel}}}\) as the coordinate permutation \(e_a\mapsto e_{\bar t(a)}\).
    For every \(P\in\mathcal{M}(\Omega,\mathcal{A})\), \((t_{\mathrm{rel},\star} P)(\bar t(a)) = P(a)\), hence \(\ev_{(\Omega',\mathcal{A}')}\circ t_{\mathrm{rel},\star} = f_{t_{\mathrm{rel}}}\circ\ev_{(\Omega,\mathcal{A})}\,,\) which is both lines of equations~\eqref{eq:transport-condition-for-info-morph} by considering the inverse function \(f^{-1}_{t_{\mathrm{rel}}}\).
    The map \(f_{t_{\mathrm{rel}}}\) is a linear bijection, so the isomorphism conditions~\eqref{eq:cencov-iso-conditions} follow from equation~\eqref{eq:cencov-push}.

    \emph{Subset embedding.}
    Let \(t_{\mathrm{emb},A}:(A,\mathcal{A}'|_A)\hookrightarrow(\Omega',\mathcal{A}')\) the inclusion with \(A\in\mathcal{A}'\).
    Here \(\mathrm{At}(\mathcal{A}'|_A)\) embeds into \(\mathrm{At}(\mathcal{A}')\), so \(\card{\mathrm{At}(\mathcal{A}'|_A)}\le\card{\mathrm{At}(\mathcal{A}')}\) and we are in the upper case of equations~\eqref{eq:transport-condition-for-info-morph}.
    Define \(f_{t_{\mathrm{emb},A}\#}:\mathbb{R}^{\card{\mathrm{At}(\mathcal{A}')}}\to\mathbb{R}^{\card{\mathrm{At}(\mathcal{A}'|_A)}}\) as the projection onto the coordinates of \(\mathrm{At}(\mathcal{A}'|_A)\); its set-theoretic preimage is the whole-real-line extension on the coordinates of \(\mathrm{At}(\mathcal{A}'|_{A^\complement})\), as stated in the theorem.
    For \(P\in\mathcal{M}(A,\mathcal{A}'|_A)\), \((t_{\mathrm{emb},A,\star} P)(a) = P(a)\) for \(a\subseteq A\) and \(0\) otherwise, whence \(f_{t_{\mathrm{emb},A}\#}\circ\ev_{(\Omega',\mathcal{A}')}\circ t_{\mathrm{emb},A,\star} = \ev_{(\Omega,\mathcal{A})}\,,\) verifying equations~\eqref{eq:transport-condition-for-info-morph}.
    The set-preimage \(f_{t_{\mathrm{emb},A}\#}^{-1}\) is a bijection between \(\ev_{(\Omega,\mathcal{A})}(\mathcal{I}_{(\Omega,\mathcal{A})})\) and \(\ev_{(\Omega',\mathcal{A}')}(\mathcal{I}_A)\), so equations~\eqref{eq:cencov-iso-conditions} hold on \(\mathcal{I}_A\subset\mathcal{I}_{(\Omega',\mathcal{A}')}\), and equation~\eqref{eq:cencov-push} yields the announced description \(I' = \{P\in\mathcal{M}(\Omega',\mathcal{A}') : P|_A\in I\}\).

    \emph{Coarsening.}
    Let \(\mathcal{A}'\subseteq\mathcal{A}\) on \(\Omega'\), and \(t_{\mathrm{cg},\mathcal{A}'}: (\Omega',\mathcal{A})\to(\Omega',\mathcal{A}')\) the identity on points.
    Each atom \(a'\in\mathrm{At}(\mathcal{A}')\) is a disjoint union of atoms of \(\mathcal{A}\), giving \(\card{\mathrm{At}(\mathcal{A}')}\le\card{\mathrm{At}(\mathcal{A})}\), the lower case of equations~\eqref{eq:transport-condition-for-info-morph}.
    Define \(f_{t_{\mathrm{cg},\mathcal{A}'}^{\#}}:\mathbb{R}^{\card{\mathrm{At}(\mathcal{A})}}\to\mathbb{R}^{\card{\mathrm{At}(\mathcal{A}')}}\) as the partition-sum matrix \((f_{t_{\mathrm{cg},\mathcal{A}'}^{\#}}x)_{a'} = \sum_{a\in\mathrm{At}(\mathcal{A}),\,a\subseteq a'} x_a\).
    For every \(P\in\mathcal{M}(\Omega,\mathcal{A})\), 
    \begin{align}
        (t_{\mathrm{cg},\mathcal{A}',\star} P)(a') = \sum_{a\subseteq a'} P(a)\; \text{ so } \;f_{t_{\mathrm{cg},\mathcal{A}'}^{\#}}\circ\,\ev_{(\Omega,\mathcal{A})} = \ev_{(\Omega,\mathcal{A}')}\circ\, t_{\mathrm{cg},\mathcal{A}',\star}\,,
    \end{align} 
    which is equations~\eqref{eq:transport-condition-for-info-morph}.
    The partition-sum matrix is surjective with the affine fibres \(\{x: \sum_{a\subseteq a'} x_a = y_{a'}\}\), so its set-preimage is a bijection between \(\ev_{(\Omega,\mathcal{A}')}(\mathcal{I}_{(\Omega,\mathcal{A}')})\) and \(\ev_{(\Omega,\mathcal{A})}(\mathcal{I}_{\mathcal{A}'})\), and equations~\eqref{eq:cencov-iso-conditions} restrict to \(\mathcal{I}_{\mathcal{A}'}\subset\mathcal{I}_{(\Omega,\mathcal{A})}\) with \(I = \{P\in\mathcal{M}(\Omega,\mathcal{A}) : P|_{\mathcal{A}'}\in I'\}\).

    Composing the three cases through Lemma~\ref{lem:meas-factorisation} and the functoriality of \(t_\star\) reproduces the general morphism, completing the proof.
\end{proof}

Given two information objects $(\Omega_i,\mathcal{A}_i,I_i)$, the subset embeddings $(\Omega_i,\mathcal{A}_i)\hookrightarrow(\Omega_1\sqcup\Omega_2,\mathcal{A}_1\oplus\mathcal{A}_2)$ of Theorem~\ref{theo:cencov-uniqueness} act as canonical injections of the two summands into the disjoint-union space, and the intersection of their pushforwards combines $I_1$ and $I_2$ into a single information set on $\Omega_1\sqcup\Omega_2$.
These data already assemble into a canonical symmetric monoidal structure on \(\mathfrak{I}\).

\begin{proposition}[Symmetric Monoidal Structure on $\mathfrak{I}$]\label{prop:monoidal-I}
    The category \(\mathfrak{I}\) carries a symmetric monoidal structure \((\mathfrak{I},\sqcup,\mathbf{1}_{\mathfrak{I}})\) defined by:
    \begin{itemize}
        \item \emph{Coproduct on objects.}
        \begin{align}\label{eq:monoidal-tensor-objects-I}
            (\Omega_1,\mathcal{A}_1,I_1)\sqcup(\Omega_2,\mathcal{A}_2,I_2) \defeq \bigl(\Omega_1\sqcup\Omega_2,\,\mathcal{A}_1\oplus\mathcal{A}_2,\, t_{\mathrm{emb},\Omega_1,\#}I_1 \cap t_{\mathrm{emb},\Omega_2,\#}I_2\bigr) \,,
        \end{align}
        where by Theorem~\ref{theo:cencov-uniqueness}, \(t_{\mathrm{emb},\Omega_1,\#}I_1 \cap t_{\mathrm{emb},\Omega_2,\#}I_2 = \{P_1\oplus P_2 : P_1\in I_1,\,P_2\in I_2\}\).
        \item \emph{Coproduct on morphisms.} For measurable maps \(t_i:(\Omega_i,\mathcal{A}_i)\to(\Omega'_i,\mathcal{A}'_i)\), the map \(t_1\sqcup t_2:\Omega_1\sqcup\Omega_2\to\Omega'_1\sqcup\Omega'_2\) acts as \(t_i\) on each component, and the associated dual pair \(((t_1\sqcup t_2)\pf,(t_1\sqcup t_2)^{\#})\) is a morphism of \(\mathfrak{I}\) by Theorem~\ref{theo:cencov-uniqueness}.
        \item \emph{Unit object.} \(\mathbf{1}_{\mathfrak{I}}\defeq(\emptyset,\{\emptyset\},\{\boldsymbol{O}\})\), the empty measurable space carrying its unique (zero) measure as information set.
        \item \emph{Associator, unitors, symmetry.} The associativity bijection \((\Omega_1\sqcup\Omega_2)\sqcup\Omega_3\cong\Omega_1\sqcup(\Omega_2\sqcup\Omega_3)\), the unit bijections \(\mathbf{1}_{\mathfrak{I}}\sqcup\Omega\cong\Omega\cong\Omega\sqcup\mathbf{1}_{\mathfrak{I}}\) and the symmetry \(\Omega_1\sqcup\Omega_2\cong\Omega_2\sqcup\Omega_1\) only re-tag the disjoint-union slots and leave the underlying atoms fixed. They are therefore isomorphisms of \(\mathfrak{I}\) by Theorem~\ref{theo:cencov-uniqueness}.
    \end{itemize}
\end{proposition}
\begin{proof}
    Each of the associator, unitor and symmetry bijections is a bijection between finite sets that preserves the underlying atoms of the disjoint-union \(\sigma\)-algebras (acting only on the tags), hence a canonical identification of tagged disjoint unions and an isomorphism of \(\mathfrak{I}\); the induced pushforwards commute with intersection, so the information-level structure is preserved.
    Since they fix the underlying atoms, these identifications belong to \(\mathfrak{I}\) and are distinct from the atom-permuting relabellings that distinguish \(\mathfrak{I}_r\).
    The pentagon, triangle and hexagon coherence diagrams commute at the level of tagged sets and of information sets (intersection of pushforwards is commutative and associative), so they commute in \(\mathfrak{I}\).
    For morphisms, \((t_1\sqcup t_2)\pf\) and \((t_1\sqcup t_2)^{\#}\) act block-diagonally on the \(\oplus\)-decomposition of measures and on the \(\cap\)-decomposition of informations, so identities and composition are preserved component-wise and \(\sqcup\) is a bifunctor.
    The unit axioms follow from the canonical bijection \(\Omega\sqcup\emptyset \cong \Omega\) and \(t_{\mathrm{emb},\emptyset,\#}\{\boldsymbol{O}\} = \mathcal{M}(\Omega,\mathcal{A})\) (since \(P|_\emptyset\) is the only measure on the empty space), so \(t_{\mathrm{emb},\Omega,\#}I \cap t_{\mathrm{emb},\emptyset,\#}\{\boldsymbol{O}\} = t_{\mathrm{emb},\Omega,\#}I\), which corresponds to \(I\) under the unit bijection.
\end{proof}

\begin{remark}
    Proposition~\ref{prop:monoidal-I} adds no new structure: it exposes the one already supplied by Theorem~\ref{theo:cencov-uniqueness}.
    The space-level \(\sqcup\) is read in the same tagged-disjoint-union sense as for \(\mathfrak{M}\), so the bifunctor is total on objects of \(\mathfrak{I}\), and the information-level combination is the intersection \(\cap\) of pushforwards, which realises the logical \texttt{AND} of two informations on disjoint parts of the universe.
\end{remark}

Inference theory maps each constraint set \(I\in\mathcal{I}_{(\Omega,\mathcal{A})}\) to one or multiple measures \(T_{(\Omega,\mathcal{A}, Q)}(I)\subseteq\mathcal{M}(\Omega,\mathcal{A})\).
Inference theory is defined independently of the formulation of the problem (in the sense of \citet{cencov1982}'s theory) if the mappings \(\mathcal{T}\) form a covariant functor from \(\mathfrak{I}\) to \(\mathfrak{M}\).

\begin{theorem}\label{theo:functoriality}
    \(\mathcal{T}\) is a monoidal functor from \(\mathfrak{I}\) to \(\mathfrak{M}\) such that the functor is the inference operator \(F_\mathcal{T}((\Omega,\mathcal{A},I)) = T_{(\Omega,\mathcal{A})}(I)\) for every \(I\in\mathcal{I}_{(\Omega,\mathcal{A})}\), and the natural transformation is the identity, meaning \(F_\mathcal{T}(I_1\sqcup I_2) = F_\mathcal{T}(I_1) \sqcup F_\mathcal{T}(I_2)\) for every \(I_1, I_2\) information on disjoint measurable spaces
    if and only if \((T_{(\Omega,\mathcal{A})})_{(\Omega,\mathcal{A})}\) is a family of inference operators that follows axioms (I\ref{axiom:isolated-system})--(I\ref{axiom:lower-scale-conservation}).
    Moreover, \(\mathcal{T}\) extends to a monoidal functor from \(\mathfrak{I}_r\) to \(\mathfrak{M}\) (functoriality also under relabelling) if and only if, in addition, the prior \(Q\) is a uniform measure on the atoms of each space.
\end{theorem}
\begin{proof}
    We start with the \(\Rightarrow\) direction.

    \emph{Relabelling} (the \(\mathfrak{I}_r\) case), for a measurable bijection \(t_{\mathrm{rel}}:(\Omega,\mathcal{A})\to(\Omega',\mathcal{A}')\), the associated morphism \((t_{\mathrm{rel}\#},t_{\mathrm{rel}}^{\#})\) in \(\mathfrak{I}_r\) is mapped by the functor to \((\Pi_{t_{\mathrm{rel}}}, \Pi'_{t_{\mathrm{rel}}})\) in \(\mathfrak{M}\) as it is the only morphism of \(\mathfrak{M}\) satisfying \(T_{(\Omega',\mathcal{A}')}(\{t_{\mathrm{rel},\star}P\}) = \Pi_{t_{\mathrm{rel}}}\bigl(T_{(\Omega,\mathcal{A})}(\{P\})\bigr)\) for every \(P\in\mathcal{M}(\Omega,\mathcal{A})\).
    So \((T_{(\Omega,\mathcal{A})})_{(\Omega,\mathcal{A})}\) is a measure-consistent family of inference operators.

    \emph{Coarsening}, for a fixed measurable space \((\Omega,\sigma(\Omega))\), every sub-\(\sigma\)-algebra \(\mathcal{A}\) is associated with a coarsening transport \(t_{\mathrm{cg},\mathcal{A},\#}\), itself associated with a morphism \((t_{\mathrm{cg},\mathcal{A},\#},t^{\#}_{\mathrm{cg},\mathcal{A}})\) in \(\mathfrak{I}\) creating an isomorphism between \(\mathcal{I}_{\mathcal{A}}\subset \mathcal{I}_{(\Omega,\sigma(\Omega))}\) and \(\mathcal{I}_{(\Omega,\mathcal{A})}\).
    By application of the functorial commutativity on a singleton \(\{P\}\in \mathcal{I}_{(\Omega,\mathcal{A})}\) (cf.\ Figure~\ref{fig:cateory-information}) and Theorem~\ref{theo:cencov-uniqueness},
    \begin{align}\label{eq:theo:functoriality:1}
        T_{(\Omega,\sigma(\Omega))}(\{P' : P'|_{\mathcal{A}} = P\}) \leftrightarrow_{F_{\mathcal{T}}((t_{\mathrm{cg},\mathcal{A},\#},t^{\#}_{\mathrm{cg},\mathcal{A}}))} T_{(\Omega,\mathcal{A})}(\{P\}) = \{P\} \,.
    \end{align} 
    The only morphisms of \(\mathfrak{M}\) that respect equation~\ref{eq:theo:functoriality:1} for all \(P\in\mathcal{M}(\Omega,\mathcal{A})\) are \(\{(\Pi_{t_{\mathrm{cg},\mathcal{A}}}, \Pi'_{t_{\mathrm{cg},\mathcal{A}}})\}\), so there is, for each sub-\(\sigma\)-algebra \(\mathcal{A}\), \(Q_{\mathcal{A}}\in\mathcal{M}(\Omega,\sigma(\Omega))\) such that \(F_{\mathcal{T}}((t_{\mathrm{cg},\mathcal{A},\#},t^{\#}_{\mathrm{cg},\mathcal{A}})) = (\Pi_{t_{\mathrm{cg},\mathcal{A}\to\sigma(\Omega)},Q_{\mathcal{A}}}, \Pi'_{t_{\mathrm{cg},\mathcal{A}\to\sigma(\Omega)},Q_{\mathcal{A}}})\).
    By composition of the morphisms and functoriality, we have for every \(\mathcal{A}'\subsetneq \mathcal{A}\) the following commutative diagram in \(\mathfrak{M}\):
    \begin{center}
        \begin{tikzcd}[row sep=normal, column sep=huge]
            (\Omega,\{\Omega,\emptyset\},\{P\})  \arrow[rr, leftrightarrow, bend left=18, "t^{\#}_{\mathrm{cg},\{\Omega,\emptyset\}\to\mathcal{A}}"] \arrow[r,  leftrightarrow, "t^{\#}_{\mathrm{cg},\{\Omega,\emptyset\}\to\sigma(\Omega)}"] \arrow[d, "T_{(\Omega,\{\Omega,\emptyset\})}"'] & (\Omega,\sigma(\Omega),\{P' : P'|_{\{\Omega,\emptyset\}} = P\}) \arrow[r,  leftrightarrow, "t^{\#}_{\mathrm{cg},\mathcal{A}\to\sigma(\Omega)}"] \arrow[d, "T_{(\Omega,\sigma(\Omega))}"'] & (\Omega,\mathcal{A},\{P' : P'|_{\{\Omega,\emptyset\}} = P\}) \arrow[d, "T_{(\Omega,\mathcal{A})}"']\\
            (\Omega,\{\Omega,\emptyset\},\{P\}) \arrow[rr, leftrightarrow, bend right=15, "\Pi_{\mathrm{cg},Q_{\{\Omega,\emptyset\}\to\mathcal{A}}}"'] \arrow[r,  leftrightarrow, "\Pi_{\mathrm{cg},Q_{\{\Omega,\emptyset\}\to\sigma(\Omega)}}"] &  (\Omega,\sigma(\Omega), \{P_{\sigma(\Omega)}\}) \arrow[r,  leftrightarrow, "\Pi_{\mathrm{cg},Q_{\mathcal{A}\to\sigma(\Omega)}}"] & (\Omega,\mathcal{A}, \{P_{\mathcal{A}}\})  \,,
        \end{tikzcd}
    \end{center}
    where the pseudo inverse morphisms are omitted for readability. 
    The commutativity of the diagram implies that \(P_{\sigma(\Omega)} = P \wedge Q_{\{\Omega,\emptyset\}\to\sigma(\Omega)}\) and \(P_{\mathcal{A}} = P \wedge Q_{\{\Omega,\emptyset\}\to\mathcal{A}}\).
    Then by composition of morphisms, we have
    \[P \wedge Q_{\{\Omega,\emptyset\}\to \sigma(\Omega)} = (P \wedge Q_{\{\Omega,\emptyset\}\to \mathcal{A}})\wedge Q_{\mathcal{A}\to \sigma(\Omega)}\,,\] 
    so \([Q_{\{\Omega,\emptyset\}\to\mathcal{A}} \wedge Q_{\mathcal{A}\to \sigma(\Omega)}] = [Q_{\{\Omega,\emptyset\}\to\sigma(\Omega)}]\) for every \(\mathcal{A}\subseteq \sigma(\Omega)\).
    So \((T_{(\Omega,\mathcal{A})})_{(\Omega,\mathcal{A})}\) satisfies the axioms (I\ref{axiom:prior})--(I\ref{axiom:lower-scale-conservation}) with the prior proven to be \(Q = Q_{\{\Omega,\emptyset\}\to \sigma(\Omega)}\).
    In the \(\mathfrak{I}_r\) case, applying the relabelling identity above to a relabelling automorphism \(\tau\) of \((\Omega,\sigma(\Omega))\) forces \(\tau_\star Q = Q\) for every atom permutation \(\tau\), so \(Q\) is the uniform (counting) measure on the atoms; for \(\mathfrak{I}\) no such constraint arises and \(Q\) stays free.

    \emph{Embedding}, for a fixed measurable space \((\Omega',\mathcal{A}')\), every subset \(A\in\mathcal{A}'\) is associated with an embedding transport \(t_{\mathrm{emb},A}:(A,\mathcal{A}'|_A)\hookrightarrow(\Omega',\mathcal{A}')\), itself associated with a morphism \((t_{\mathrm{emb},A,\#},t^{\#}_{\mathrm{emb},A})\) in \(\mathfrak{I}\) creating an isomorphism between \(\mathcal{I}_{(A,\mathcal{A}'|_A)}\) and \(\mathcal{I}_A\subset \mathcal{I}_{(\Omega',\mathcal{A}')}\).
    Consider the same embedding for \(A^{\complement}\). 
    Since \(A\) and \(A^{\complement}\) partition \(\Omega\), the map \((t_{\mathrm{emb},A},t_{\mathrm{emb},A^{\complement}}) : A\sqcup A^{\complement} \to \Omega\) is the canonical identification of \(\Omega\) with the coproduct of \(A\) and \(A^{\complement}\); preserving the underlying atoms, it is an isomorphism of \(\mathfrak{I}\) (resp. \(\mathfrak{M}\)) by Proposition~\ref{prop:monoidal-I}.
    Applying monoidal functoriality of \(F_{\mathcal{T}}\) along this isomorphism gives the commutative diagram
    \begin{center}
        \begin{tikzcd}[row sep=normal, column sep=huge]
            (A,\mathcal{A}|_A,I_A) \sqcup (A^{\complement},\mathcal{A}|_{A^\complement},I_{A^\complement}) \arrow[r, leftrightarrow, "{(t_{\mathrm{emb},A},t_{\mathrm{emb},A^{\complement}})}"] \arrow[d, "T_{(A,\mathcal{A}|_A)} \oplus T_{(A^{\complement},\mathcal{A}|_{A^\complement})}"'] & (\Omega,\mathcal{A},I_A\cap I_{A^\complement}) \arrow[d, "T_{(\Omega,\mathcal{A})}"] \\
            T_{(A,\mathcal{A}|_A)}(I_A) \oplus T_{(A^{\complement},\mathcal{A}|_{A^\complement})}(I_{A^\complement}) \arrow[r, leftrightarrow, "{(\Pi_{\mathrm{emb},A},\Pi_{\mathrm{emb},A^{\complement}})}"] & T_{(\Omega,\mathcal{A})}(I_A\cap I_{A^\complement})  \,,
        \end{tikzcd}
    \end{center}
    where \(I_A\cap I_{A^\complement}\) on the right denotes the intersection \(t_{\mathrm{emb},A,\#}(I_A) \cap t_{\mathrm{emb},A^\complement,\#}(I_{A^\complement})\) of the pushforwards into \((\Omega,\mathcal{A})\) (cf. Proposition~\ref{prop:monoidal-I}).
    Then by application of the previous result to the special case \(I_A\in\mathcal{I}_{\{A,\emptyset\}}\) and \(I_{A^\complement}\in\mathcal{I}_{\{A^\complement,\emptyset\}}\), we have \([Q]^{\sigma(A,A^{\complement})} = (T_{(\Omega,\mathcal{A})}(I_A\cap I_{A^\complement}))^{\sigma(A,A^{\complement})} = (T_{(A,\mathcal{A}|_A)}(I_A) \oplus T_{(A^{\complement},\mathcal{A}|_{A^\complement})}(I_{A^\complement}))^{\sigma(A,A^{\complement})}=([Q_{A,\mathcal{A}|_A}] \oplus [Q_{A^{\complement},\mathcal{A}|_{A^\complement}}])^{\sigma(A,A^{\complement})}\).
    So \((T_{(\Omega,\mathcal{A})})_{(\Omega,\mathcal{A})}\) satisfies the axioms (I\ref{axiom:isolated-system}) with the relation between priors proven.

    The \(\Leftarrow\) direction is straightforward by application of the axioms and the definition of measure-consistent family of inference operators to deduce functoriality on the generators of morphisms of \(\mathfrak{I}\) (coarsening and subset embedding), and on the additional relabelling generators in the case of \(\mathfrak{I}_r\), then by composition and monoidal functoriality to deduce functoriality on all morphisms.
\end{proof}

\subsection{Quotient inference functor on moment information}\label{subsec:quotient-operators}

The inference problem can also be read through the moment information a measure carries, the total mass being the moment against the constant function and a general moment the integral \(\langle P,f\rangle\) against a strictly positive reweighting \(f\).
This view raises a new invariance requirement, namely that the inferred distribution be processed independently of such moment information, which takes formal shape as an inference functor between two quotients.
Its source is the quotient information category, where informations sharing a distribution but differing in their moment are identified; its target is the quotient measure category, where a measure is specified only by its distribution and its moment value discarded.
Building this quotient functor singles out the \(\alpha\)-divergences for the total mass and the Kullback-Leibler divergence for a general moment.

\begin{definition}[Quotient Category \(\mathfrak{M}_{/\!\sim }\)]\label{def:quotient-measures-category}
For a sub-\(\sigma\)-algebra \(\mathcal{A}_d\subsetneq\mathcal{A}\), let \([P]^{\mathcal{A}_d} \defeq \{\, g\cdot P \,:\, g>0 \text{ } \mathcal{A}_d\text{-measurable} \,\}\) be the class of \(P\) under \(\mathcal{A}_d\)-measurable positive rescaling, equivalently the family \(P^{\mathcal{A}_d}\) of conditional distributions of \(P\) inside the atoms of \(\mathcal{A}_d\).
We define the category \(\mathfrak{M}_{/\!\sim }\) from \(\mathfrak{M}\) by taking as objects the \((\Omega,\mathcal{A},\{[P]^{\mathcal{A}_d}\})\) for every \(\mathcal{A}_d\subsetneq\mathcal{A}\), and generating the morphisms by the relation
 \(\{P\} \leftrightarrow_{(\Pi,\Pi')} \{P'\} \implies \{[P]^{\mathcal{A}_d}\} \leftrightarrow_{(\Pi,\Pi')} \{[P']^{\mathcal{A}'_d}\}\) if \(\mathcal{A}_d\subseteq\mathcal{A}_\Pi\),
where \(\mathcal{A}_\Pi\) is the sub-\(\sigma\)-algebra on whose atoms \(\Pi\) acts as a bijection.
\end{definition}

\begin{remark}
    The class \([P]^{\mathcal{A}_d}\) rescales each atom of \(\mathcal{A}_d\) independently, so the quotient is applied disjointly on the atoms of \(\mathcal{A}_d\).
    This defines a notion of distribution on a subset of the total space, not only on the full space, and is what makes the quotient compatible with the monoidal structure \(\sqcup\) of \(\mathfrak{M}\).
\end{remark}

\begin{lemma}
    \(\mathfrak{M}_{/\!\sim }\) is a well-defined symmetric monoidal category under the product \(\sqcup\) inherited from \(\mathfrak{M}\).
\end{lemma}
\begin{proof}
    \emph{Composition.}
    A generating morphism satisfies \(\mathcal{A}_d\subseteq\mathcal{A}_\Pi\), so \(\Pi\) induces a bijection of the atoms of \(\mathcal{A}_d\) and intertwines the \(\mathcal{A}_d\)-measurable rescaling, \(\Pi(g\cdot P) = g'\cdot \Pi(P)\) with \(g'\) the transport of \(g\).
    Hence \(\{P\} \leftrightarrow_{(\Pi,\Pi')} \{P'\}\) determines \([P']^{\mathcal{A}'_d}\) from \([P]^{\mathcal{A}_d}\) alone.
    The composite of two morphisms respecting \(\mathcal{A}_d\) again respects it, hence if \([P'_1]^{\mathcal{A}'_d}=[P_2]^{\mathcal{A}_d}\) then \(\{[P_1]^{\mathcal{A}_d}\} \leftrightarrow_{(\Pi_1\Pi_2,\Pi_1'\Pi_2')} \{[P'_2]^{\mathcal{A}'_d}\}\) is well defined.

    \emph{Monoidal product.}
    On the disjoint union the quotient is taken along \(\mathcal{A}_{d,1}\oplus\mathcal{A}_{d,2}\), whose atoms refine the summand partition \(\{\Omega_1,\Omega_2\}\).
    The class therefore absorbs an independent positive rescaling of each summand, so \([\lambda_1 P_1\oplus\lambda_2 P_2]^{\mathcal{A}_{d,1}\oplus\mathcal{A}_{d,2}} = [P_1\oplus P_2]^{\mathcal{A}_{d,1}\oplus\mathcal{A}_{d,2}}\) for all \(\lambda_1,\lambda_2>0\), and \([P_1]^{\mathcal{A}_{d,1}}\sqcup[P_2]^{\mathcal{A}_{d,2}}\defeq[P_1\oplus P_2]^{\mathcal{A}_{d,1}\oplus\mathcal{A}_{d,2}}\) depends only on \([P_1]^{\mathcal{A}_{d,1}}\) and \([P_2]^{\mathcal{A}_{d,2}}\).
    The associator, unitor and symmetry of \(\mathfrak{M}\) are relabellings respecting \(\mathcal{A}_d\), hence descend, and their coherence diagrams commute in \(\mathfrak{M}_{/\!\sim }\) since they commute in \(\mathfrak{M}\).
\end{proof}

\begin{remark}
    For \(\mathcal{A}_d\) restrictied to \(\{\Omega,\emptyset\}\) the class \([P]^{\{\Omega,\emptyset\}}=[P]\) is the ray \(\{cP:c>0\}\) and \(\mathfrak{M}_{/\!\sim }\) is the category of statistical models of \citet{cencov1982} on finite spaces (cf.\ \cite[Sec.~6]{cencov1982}): identifying each \([P]\) with its probability-measure representative recovers his category exactly.
\end{remark}

If we restrain \(\obj(\mathfrak{I})\) to the couples of constraints \((I^d\in \mathcal{I}^{\{\Omega,\emptyset\}}, I_m\in \mathcal{I}_{\{\Omega,\emptyset\}})\), and apply the quotient operator to the measures that compose each information set, we obtain a class of equivalent information regarding the distribution, independent of \(I_m\):
\begin{align}\label{eq:quotient-information-set}
    [I^d\cap I_m] = [I^d] \,,\quad \forall\, I_m \neq\emptyset \in \mathcal{I}_{\{\Omega,\emptyset\}}\,,
\end{align}
motivating the following definition.

\begin{definition}[Distribution Information Category \(\mathfrak{I}_{/\!\sim }\)]\label{def:distribution-information-category}\label{def:distribution-information-space}
We define the category \(\mathfrak{I}_{/\!\sim }\) from the \v{C}encov information category \(\mathfrak{I}\) by taking as objects the representative elements of the subspace \(\cup_{(\Omega,\mathcal{A})}(I^d\cap I_m)_{I^d\in\mathcal{I}^{\mathcal{A}_d},I_m\in\mathcal{I}_{\mathcal{A}_d},\mathcal{A}_d\subsetneq \mathcal{A}}\subseteq \mathcal{I}\) quotiented by
\begin{align}
    [I^{d}\cap I_m] = [I^d]  \,,\quad \forall\, I_m \neq\emptyset \in \mathcal{I}_{\mathcal{A}_d},\ I^d\in \mathcal{I}^{\mathcal{A}_d}\,,
\end{align}
and generating the morphisms by
 \(I^d \cap I_m \leftrightarrow_{(t_{\#},t^{\#})} I'^d\cap I'_m \implies [I^d] \leftrightarrow_{(t_{\#},t^{\#})} [I'^d]\). 
\end{definition}

\begin{remark}
    Here \([\,\cdot\,]\) denotes the canonical projection to the quotient, applied uniformly to measures, information sets, and later on inference operators.
\end{remark}

\begin{lemma}
    \(\mathfrak{I}_{/\!\sim }\) is a well-defined symmetric monoidal category under the product \(\sqcup\) inherited from \(\mathfrak{I}\).
\end{lemma}
\begin{proof}
    \emph{Composition.}
    Suppose \(I^d_1 \cap I_{m,1}  \leftrightarrow_{(t_{1,\#},t_1^{\#})} I'^d_1 \cap I'_{m,1} \) and \(I^d_2 \cap I_{m,2} \leftrightarrow_{(t_{2,\#},t_2^{\#})} I'^d_1 \cap I'_{m,2}\) with \([I'^d_1 \cap I'_{m,1}]=[I^d_2 \cap I_{m,2}]\); we show \([I^d_1 \cap I_{m,1}] \leftrightarrow_{(t_{1,\#}\circ t_{2,\#},\,t_1^{\#}\circ t_2^{\#})} [I'^d_2 \cap I'_{m,2}]\).
    The morphisms \(t_i\) are bijections between the atoms of sub-\(\sigma\)-algebras \(\mathcal{A}_{t_i}\), and as \(\mathcal{A}_d\subsetneq \mathcal{A}_{t_i}\) they induce a bijection between the atoms of \(\mathcal{A}_{d}\).
    So \(I^d_i \cap I_{m,i}  \leftrightarrow_{(t_{i,\#},t_i^{\#})} I'^d_i \cap I'_{m,i} \) for \(i=1,2\) implies \(I^d_i \leftrightarrow_{(t_{i,\#},t_i^{\#})} I'^d_i \) for \(i=1,2\).
    As \([I'^d_1 \cap I'_{m,1}]=[I^d_2 \cap I_{m,2}]\) implies \(I'^d_1 = I^d_2\), we conclude \([I^d_1] \leftrightarrow_{(t_{1,\#}\circ t_{2,\#},\,t_1^{\#}\circ t_2^{\#})} [I'^d_2]\).

    \emph{Monoidal product.}
    By Proposition~\ref{prop:monoidal-I} the join of two informations on disjoint spaces is the intersection of their pushforwards, \(t_{\mathrm{emb},\Omega_1,\#}I_1\cap t_{\mathrm{emb},\Omega_2,\#}I_2\) on \(\Omega_1\sqcup\Omega_2\).
    The quotient there is taken along \(\mathcal{A}_{d,1}\oplus\mathcal{A}_{d,2}\), whose atoms refine the summand partition \(\{\Omega_1,\Omega_2\}\), so the mass informations \(I_{m,1},I_{m,2}\) are absorbed independently and \([I^d_1]\sqcup[I^d_2]\defeq[\,t_{\mathrm{emb},\Omega_1,\#}I^d_1\cap t_{\mathrm{emb},\Omega_2,\#}I^d_2\,]\) depends only on \([I^d_1]\) and \([I^d_2]\).
    The associator, unitor and symmetry of \(\mathfrak{I}\) are the canonical structural isomorphisms of \(\mathfrak{I}\) (Proposition~\ref{prop:monoidal-I}) respecting \(\mathcal{A}_d\), hence descend, and their coherence diagrams commute in \(\mathfrak{I}_{/\!\sim }\) since they commute in \(\mathfrak{I}\).
\end{proof}

From Definition~\ref{def:distribution-information-space}, we can define a notion of inference on distribution.
An inference operator family \(\mathcal{T}\) processes distribution information independently of the total mass information if and only if there is an inference functor \(\mathcal{T}_{/\!\sim}\) from \(\mathcal{I}_{/\!\sim}\) to \(\mathcal{M}_{/\!\sim}\) such that applying the quotient operator to the output of \(\mathcal{T}\) is equivalent to applying the quotient operator to the input of \(\mathcal{T}\) and then applying the inference functor \(\mathcal{T}_{/\!\sim}\).

\begin{definition}[Inference on Distribution Induced by an Inference Family]\label{def:inference-on-distribution-induced-by-inference-family}
    Let \(\mathcal{T}\) be an expressive family of inference operators.
    We say that \(\mathcal{T}\) induces an inference on distribution family \(\mathcal{T}_{/\!\sim}\) if the following diagram is commutative
    \[
        \begin{tikzcd}[row sep=large, column sep=large]
            F_{/\!\sim}^{-1}(\obj(\mathfrak{I}_{/\!\sim})) \arrow[d, "\mathcal{T}"'] \arrow[r, "F_{/\!\sim}"] & \mathfrak{I}_{/\!\sim} \arrow[d, "\mathcal{T}_{/\!\sim}"] \\
            \mathfrak{M} \arrow[r, "F_{/\!\sim}"'] & \mathfrak{M}_{/\!\sim}  
        \end{tikzcd}
    \]
    where \(F_{/\!\sim}\) denotes the quotient functor, applied componentwise on each category and \(F_{/\!\sim}^{-1}(\obj(\mathfrak{I}_{/\!\sim}))\) is subset of \(\mathcal{I}\) composed of pairs of mass information and distribution information intersected as in Definition~\ref{def:distribution-information-space}.
\end{definition}

\begin{theorem}\label{theo:functoriality-distrib-info}
    \(\mathcal{T}\) induces an expressive family of inference on distribution operators \(\mathcal{T}_{/\!\sim}\) if and only if \(\mathcal{T}\) satisfies axiom~(I\ref{axiom:upper-scale-conservation}).
    Moreover, if \(\mathcal{T}\) induces a monoidal covariant functor from \(\mathfrak{I}\) to \(\mathfrak{M}\) (as in Theorem~\ref{theo:functoriality}), then so does \(\mathcal{T}_{/\!\sim}\) from \(\mathfrak{I}_{/\!\sim}\) to \(\mathfrak{M}_{/\!\sim}\) where the functor on morphisms acts as the functor \(\mathcal{T}\).
\end{theorem}
\begin{proof}
    Throughout, \(F_{/\!\sim}\) denotes the quotient functor of Definition~\ref{def:distribution-information-space} applied componentwise, and \(\mathcal{A}_d\subsetneq\mathcal{A}\) the sub-\(\sigma\)-algebra along which the quotient is taken (the \(\mathcal{A}^\prime\) of axiom~(I\ref{axiom:upper-scale-conservation}), with \(I^d=I^{\mathcal{A}_d}\in\mathcal{I}^{\mathcal{A}_d}\) and \(I_m=I|_{\mathcal{A}_d}\in\mathcal{I}_{\mathcal{A}_d}\)).

    \emph{Well-definedness of \(\mathcal{T}_{/\!\sim}\).}
    By Definition~\ref{def:inference-on-distribution-induced-by-inference-family} the square commutes if and only if \([T_{(\Omega,\mathcal{A},Q)}(I)]\) depends on \(I\) only through its class \([I]\).
    Since the quotient identifies \(I^d\cap I_m\) with \(I^d\) for every nonempty \(I_m\in\mathcal{I}_{\mathcal{A}_d}\) (Definition~\ref{def:distribution-information-space}), this is exactly
    \begin{align}\label{eq:distrib-info-class-criterion}
        [T_{(\Omega,\mathcal{A},Q)}(I^d\cap I_m)] = [T_{(\Omega,\mathcal{A},Q)}(I^d)] \,,\quad \forall\, I^d\in\mathcal{I}^{\mathcal{A}_d},\ I_m\in\mathcal{I}_{\mathcal{A}_d}\,.
    \end{align}

    \emph{Equivalence with \textup{(I\ref{axiom:upper-scale-conservation})}.}
    Two measures carry the same class \([\,\cdot\,]\) precisely when they share their distribution part \((\,\cdot\,)^{\mathcal{A}_d}\), the conditional distributions inside the atoms of \(\mathcal{A}_d\) (Definition~\ref{def:quotient-measures-category}).
    Hence equation~\eqref{item:upper-scale-conservation}, \(T_{(\Omega,\mathcal{A})}(I^d\cap I_m)^{\mathcal{A}_d} = T_{(\Omega,\mathcal{A})}(I^d\cap I'_m)^{\mathcal{A}_d}\), is the equality of classes \([T_{(\Omega,\mathcal{A})}(I^d\cap I_m)] = [T_{(\Omega,\mathcal{A})}(I^d\cap I'_m)]\) for all \(I_m,I'_m\in\mathcal{I}_{\mathcal{A}_d}\).
    Taking \(I'_m=\mathcal{M}(\Omega,\mathcal{A})\) (the trivial mass constraint) yields \eqref{eq:distrib-info-class-criterion}, and conversely \eqref{eq:distrib-info-class-criterion} applied to \(I_m\) and to \(I'_m\) gives the axiom.

    \emph{Monoidal functoriality.}
    Assume in addition that \(\mathcal{T}\) induces the monoidal covariant functor \(F_{\mathcal{T}}:\mathfrak{I}\to\mathfrak{M}\) of Theorem~\ref{theo:functoriality}, so \(\mathcal{T}_{/\!\sim}\) is defined by the first part.
    Set \(F_{\mathcal{T}_{/\!\sim}}([I^d])\defeq[F_{\mathcal{T}}(I^d)]=[T_{(\Omega,\mathcal{A})}(I^d)]\) on objects, and let \(F_{\mathcal{T}_{/\!\sim}}\) act on morphisms as \(F_{\mathcal{T}}\), that is \(F_{/\!\sim}\circ F_{\mathcal{T}} = F_{\mathcal{T}_{/\!\sim}}\circ F_{/\!\sim}\).
    The object map is well defined by \eqref{eq:distrib-info-class-criterion}.
    Every morphism of \(\mathfrak{I}_{/\!\sim}\) is the image under \(F_{/\!\sim}\) of a morphism \((t_{\#},t^{\#})\) of \(\mathfrak{I}\) acting bijectively on the atoms of \(\mathcal{A}_d\) (Definition~\ref{def:distribution-information-category}); by Theorem~\ref{theo:functoriality} its image \((\Pi,\Pi')=F_{\mathcal{T}}((t_{\#},t^{\#}))\) acts bijectively on the same atoms, so \(\mathcal{A}_d\subseteq\mathcal{A}_{\Pi}\) and \((\Pi,\Pi')\) generates a morphism of \(\mathfrak{M}_{/\!\sim}\) (Definition~\ref{def:quotient-measures-category}).
    As \(F_{/\!\sim}\) is full and surjective on objects and \(F_{\mathcal{T}}\) is a functor, the commuting square transports identities and composition, hence \(F_{\mathcal{T}_{/\!\sim}}\) is a covariant functor.
    Both quotient functors are strict monoidal, \([X_1]\sqcup[X_2]=[X_1\sqcup X_2]\) in \(\mathfrak{I}\) and in \(\mathfrak{M}\) with coherence inherited from \(\mathfrak{I}\) and \(\mathfrak{M}\) (Definitions~\ref{def:distribution-information-category} and~\ref{def:quotient-measures-category}), so monoidality of \(F_{\mathcal{T}}\) descends,
    \begin{align*}
        F_{\mathcal{T}_{/\!\sim}}([I_1]\sqcup[I_2]) = [F_{\mathcal{T}}(I_1\sqcup I_2)] = [F_{\mathcal{T}}(I_1)\sqcup F_{\mathcal{T}}(I_2)] = F_{\mathcal{T}_{/\!\sim}}([I_1])\sqcup F_{\mathcal{T}_{/\!\sim}}([I_2])\,.
    \end{align*}
    Thus \(F_{\mathcal{T}_{/\!\sim}}\) is a monoidal covariant functor from \(\mathfrak{I}_{/\!\sim}\) to \(\mathfrak{M}_{/\!\sim}\) acting on morphisms as \(\mathcal{T}\).
\end{proof}

The map \(t_f : P \mapsto f\cdot P\) is a self-homeomorphism of \(\mathcal{M}(\Omega,\mathcal{A})\).
We enrich \(\morph(\mathfrak{M})\) with the linear transformations \(t_{f}\), \(f>0\), forming \(\mathfrak{M}_{\text{linear}}\).

\begin{definition}[Linear Morphism of \(\mathfrak{M}_{\text{linear}}\)]\label{def:linear-morphism-category-M}
    For every \(f:\Omega\to \mathbb{R}^+\) measurable, we define the linear morphism \(t_f\) as the mapping that maps \((\Omega,\mathcal{A},Q),\{P\}\) to \((\Omega,\mathcal{A},f\cdot Q),\{f\cdot P\}\).
    We then define the category \(\mathfrak{M}_{\text{linear}}\) as the category with the same objects as \(\mathfrak{M}\) and morphisms generated by the morphisms of \(\mathfrak{M}\) and the linear morphisms.
\end{definition}

Processing the \(f\)-distribution independently of the moment \(\langle f, P\rangle\), as required in the section introduction, motivates the following generalisation of Definition~\ref{def:distribution-information-space} to any statistic \(f>0\).

\begin{definition}[Linear Transformation Information as Quotient Space \(\mathcal{I}_{\text{linear}}\)]\label{def:linear-constraint-information-space}
    We define the quotient space \(\mathcal{I}_{\text{linear}}\) from the subspace \(\cup_{(\Omega,\mathcal{A})}(f^{-1}\cdot I^d\cap f^{-1}\cdot I_m)_{f>0,I^d\in\mathcal{I}^{\mathcal{A}_d},I_m\in\mathcal{I}_{\mathcal{A}_d},\mathcal{A}_d\subsetneq \mathcal{A}}\subseteq \mathfrak{I}\) quotiented as follows:
    \begin{align}
        [f^{-1}\cdot I^{d}\cap f^{-1}\cdot I_m] = [ I^d]_f  \,,\quad \forall\,f>0,\, I_m \neq \emptyset \in \mathcal{I}_{\mathcal{A}_d},\, I^d\in \mathcal{I}^{\mathcal{A}_d}\,.
    \end{align}
\end{definition}

A single information can satisfy several linear constraints at once, typically when it lies on the fibre of several expectations \(\langle f_1, P\rangle, \langle f_2, P\rangle, \dots, \langle f_n, P\rangle\), and so admits several presentations \(f^{-1}\cdot I^d\cap f^{-1}\cdot I_m\) for distinct statistics \(f\).
The category \(\mathfrak{I}_{\text{linear}}\) must therefore identify these presentations through a richer notion of equivalence than the one carried by \(\mathfrak{I}_{/\!\sim}\).

\begin{definition}[Decorated Linear Information Category \(\widetilde{\mathfrak{I}}_{\text{lin}}\)]\label{def:decorated-linear-info-category}
    We define \(\widetilde{\mathfrak{I}}_{\text{lin}}\) as the unique thin category (at most one morphism between any two objects) whose objects are tuples \((\Omega,\mathcal{A},f,I^d,I_m)\) with \(f>0\) an \(\mathcal{A}\)-measurable function, \(\mathcal{A}_d\subsetneq \mathcal{A}\), \(I^d\in\mathcal{I}^{\mathcal{A}_d}\) and \(I_m\in\mathcal{I}_{\mathcal{A}_d}\), and whose morphisms are characterised by the Hom-reflection property
    \begin{align}\label{eq:decorated-linear-info-hom-reflection}
        \Hom_{\widetilde{\mathfrak{I}}_{\text{lin}}}\!\bigl((\Omega,\mathcal{A},f,I^d,I_m),(\Omega',\mathcal{A}',f',I'^d,I'_m)\bigr)\neq\emptyset \iff \Hom_{\mathfrak{I}}\!\bigl(U_{\obj}(\cdot),U_{\obj}(\cdot)\bigr)\neq\emptyset\,,
    \end{align}
    where \(U_{\obj}:(\Omega,\mathcal{A},f,I^d,I_m)\mapsto (\Omega,\mathcal{A},f^{-1}\cdot I^d\cap f^{-1}\cdot I_m)\) is a forgetful object map to \(\obj(\mathfrak{I})\).
    The map \(U_{\obj}\) is in general many-to-one, since a single information of \(\mathfrak{I}\) admits several presentations whenever it lies on multiple linear constraints.
\end{definition}

\begin{definition}[Linear Transformation Information Category \(\mathfrak{I}_{\text{linear}}\)]\label{def:linear-transformation-information-category-2}
    We define \(\mathfrak{I}_{\text{linear}}\) as the unique category with object class \(\mathcal{I}_{\text{linear}}\) (Definition~\ref{def:linear-constraint-information-space}) such that the object-level assignment
    \begin{align}
        F_{/\!\sim,\,\obj}:(\Omega,\mathcal{A},f,I^d,I_m)\mapsto (\Omega,\mathcal{A},[I^d]_f)
    \end{align}
    extends to a fully faithful covariant functor \(F_{/\!\sim}:\widetilde{\mathfrak{I}}_{\text{lin}}\to\mathfrak{I}_{\text{linear}}\).
\end{definition}

\begin{remark}[Role of \(\widetilde{\mathfrak{I}}_{\text{lin}}\)]\label{rem:role-of-decorated-linear-info-category}
    The decorated category \(\widetilde{\mathfrak{I}}_{\text{lin}}\) serves a purely instrumental purpose: it is the intermediate object through which the \v{C}encov categorical structure of \(\mathfrak{I}\) is transported to \(\mathfrak{I}_{\text{linear}}\).
    Direct transport along the quotient of Definition~\ref{def:linear-constraint-information-space} fails because the quotient creates duplicates: a single information of \(\mathfrak{I}\) yields several classes \([I^d]_f\) when it lies on multiple linear constraints.
\end{remark}

\begin{lemma}\label{lem:uniqueness-isomorph-category}
    Let \(\mathfrak{C}\) be a category with at most one morphism between any two objects. 
    For any surjective mapping \(F_{\obj{}}:\obj{\mathfrak{C}} \mapsto \obj{\mathfrak{C}'}\), there exists at most one couple \((\morph{\mathfrak{C}'},F_{\morph}:\morph{\mathfrak{C}}\mapsto\morph{\mathfrak{C}'})\) such that
    \((F_{\obj{}},F_{\morph})\) forms a fully faithful covariant functor from \(\mathfrak{C}\) to \(\mathfrak{C}'\).
\end{lemma}
\begin{proof}
    Fix \(F_{\obj{}}\) and two objects \(X,Y\in\obj{\mathfrak{C}}\).
    Full faithfulness means that \(F_{\morph}\) restricts to a bijection \(\Hom_{\mathfrak{C}}(X,Y)\to\Hom_{\mathfrak{C}'}(F_{\obj{}}(X),F_{\obj{}}(Y))\).
    Since \(\Hom_{\mathfrak{C}}(X,Y)\) has at most one element by hypothesis, so does \(\Hom_{\mathfrak{C}'}(F_{\obj{}}(X),F_{\obj{}}(Y))\); when the former is non-empty, the latter consists of the image under \(F_{\morph}\) of its unique element.
    Hence both the set of morphisms of \(\mathfrak{C}'\) between objects in the image of \(F_{\obj{}}\) and the action of \(F_{\morph}\) on every morphism of \(\mathfrak{C}\) are uniquely determined by \(F_{\obj{}}\).
\end{proof}

\begin{lemma}
     \(\widetilde{\mathfrak{I}}_{\text{lin}}\) and \(\mathfrak{I}_{\text{linear}}\) are well-defined.
\end{lemma}
\begin{proof}
        Existence and uniqueness of \(\widetilde{\mathfrak{I}}_{\text{lin}}\), \(\mathfrak{I}_{\text{linear}}\), and of the morphism map \(F_{/\!\sim,\,\morph}\) follow from Lemma~\ref{lem:uniqueness-isomorph-category}, since \(\widetilde{\mathfrak{I}}_{\text{lin}}\) is thin by Definition~\ref{def:decorated-linear-info-category}.
\end{proof}

The morphisms of \(\mathfrak{I}_{\text{linear}}\) are made explicit in the following Lemma.

\begin{lemma}[Characterisation of \(\mathfrak{I}_{\text{linear}}\)]\label{lem:linear-transformation-information-category}
\(\mathfrak{I}_{\text{linear}}\) is characterised by the objects \(\mathcal{I}_{\text{linear}}\) and the morphisms generated by the two families:
\begin{itemize}
    \item \emph{(Transport.)} For \((t\pf,t^{\#})\) a morphism of \(\mathfrak{I}\) with underlying maps \(t\), and \(f>0\) measurable with respect to the sub-\(\sigma\)-algebra \(\mathcal{A}_{t}\) on whose atoms \(t^{-1}\) acts bijectively,
    \begin{align}
        f^{-1}\cdot I^d \cap f^{-1}\cdot I_m \leftrightarrow_{(t\pf,t^{\#})} f^{-1} \circ  t^{-1} \cdot I'^d\cap f^{-1} \circ  t^{-1} \cdot I'_m \implies [I^d]_f \leftrightarrow_{(t\pf,t^{\#})} [I'^d]_{f \circ  t^{-1}}\,.
    \end{align}
    \item \emph{(Re-presentation.)} For \(f_1,f_2>0\) measurable on \((\Omega,\mathcal{A})\) and mass parts \(I_1^m,I_2^m\),
    \begin{align}
        f_1^{-1}\cdot I_1^d \cap f_1^{-1}\cdot I_1^m =  f_2^{-1}\cdot I_2^d \cap f_2^{-1}\cdot I_2^m \implies   [I_1^d]_{f_1} \leftrightarrow_{(\,t_{f_2}\circ t_{f_1}^{-1},\ t_{f_1}\circ t_{f_2}^{-1}\,)} [I_2^d]_{f_2}  \,,
    \end{align}
    where \(t_{f}\) is the linear self-homeomorphism of Definition~\ref{def:linear-morphism-category-M} and the connecting morphism \(t_{f_2}\circ t_{f_1}^{-1}\) maps \(I_1^d\cap I_1^m\) to \(I_2^d\cap I_2^m\) as the two presentations carry the same underlying \v{C}encov information.
\end{itemize}
\end{lemma}
\begin{proof}
    By Definition~\ref{def:linear-transformation-information-category-2}, \(F_{/\!\sim}:\widetilde{\mathfrak{I}}_{\text{lin}}\to\mathfrak{I}_{\text{linear}}\) is fully faithful and surjective on objects.
    Writing \(U_i\defeq f_i^{-1}\cdot I_i^d\cap f_i^{-1}\cdot I_i^m\) for the image of the presentation \((\Omega,\mathcal{A},f_i,I_i^d,I_i^m)\) under \(U_{\obj}\), the Hom-reflection~\eqref{eq:decorated-linear-info-hom-reflection} gives
    \begin{align*}
        \Hom_{\mathfrak{I}_{\text{linear}}}\!\bigl([I_1^d]_{f_1},[I_2^d]_{f_2}\bigr)\;\cong\;\Hom_{\mathfrak{I}}\!\bigl(U_1,U_2\bigr)\,,
    \end{align*}
    so every morphism of \(\mathfrak{I}_{\text{linear}}\) is the image under \(F_{/\!\sim}\) of a unique \(m\in\Hom_{\mathfrak{I}}(U_1,U_2)\).

    \emph{Both families are morphisms.}
    The first is the \(F_{/\!\sim}\)-image of a reflected \(\mathfrak{I}\)-morphism, the statistic transporting as \(f\mapsto f\circ t^{-1}\).
    For the second, the hypothesis is exactly \(U_1=U_2\); writing \(J\) for this common information, \(\mathrm{id}_J\in\Hom_{\mathfrak{I}}(J,J)\) reflects to the asserted isomorphism, realised on representatives by \(t_{f_2}\circ t_{f_1}^{-1}\), which sends \(I_1^d\cap I_1^m=t_{f_1}(J)\) to \(t_{f_2}(J)=I_2^d\cap I_2^m\).

    \emph{They generate.}
    Given a morphism, let \(m\in\Hom_{\mathfrak{I}}(U_1,U_2)\) be the associated \(\mathfrak{I}\)-morphism, \((t\pf,t^{\#})\) its forward and backward morphisms and \(\mathcal{A}_{t}\) the sub-\(\sigma\)-algebra on whose atoms \(t\) acts bijectively.
    Transporting the statistic \(f_1\) along \(m\) (first family) reaches a presentation \([I_1'^d]_{f_1\circ t^{-1}}\) with underlying information \(U_2\), the codomain of \(m\).
    As \(U_2\) is also the underlying information of \([I_2^d]_{f_2}\), the second family identifies the two, so \(F_{/\!\sim}(m)\) factors as a re-presentation after a transport.
    Hence the two families generate \(\morph(\mathfrak{I}_{\text{linear}})\).
\end{proof}

\begin{theorem}[Quotient Geometry of Inference Theory]\label{theo:functoriality-moment-info}
    If \(\mathcal{T}\) induces a monoidal covariant functor from \(\mathfrak{I}\) to \(\mathfrak{M}\) (as in Theorem~\ref{theo:functoriality}), then \(\mathcal{T}_{\text{linear}}\) is well-defined and induces a monoidal covariant functor from \(\mathfrak{I}_{\text{linear}}\) to \(\mathfrak{M}_{\text{linear}}\) if and only if \(\mathcal{T}\) satisfies axiom~(I\ref{axiom:linear-constraint-independance}).
    In that case the induced operator \([T_{(\Omega,\mathcal{A})}]_{f_1,\cdots,f_n}\) on \(\mathcal{F}_{f_1}\cap\cdots\cap\mathcal{F}_{f_n}\) carries the prior
    \begin{align}\label{item:linear-constraint-coherence-prior}
        [Q^\star] = \argmin_{\mathcal{F}_{f_1}\cap\cdots\cap \mathcal{F}_{f_n}}  [T_{(\Omega,\mathcal{A},[Q])}]_{f_1} = \cdots   = \argmin_{\mathcal{F}_{f_1}\cap\cdots\cap \mathcal{F}_{f_n}}  [T_{(\Omega,\mathcal{A},[Q])}]_{f_n}  \,.
    \end{align}
\end{theorem}
\begin{proof}
    By Lemma~\ref{lem:linear-transformation-information-category} the morphisms of \(\mathfrak{I}_{\text{linear}}\) are generated by two families, transport and re-presentation, treated in turn.

    \emph{Transport.}
    Fix a statistic \(f>0\).
    The reweighting \(t_f\) identifies the objects and transport morphisms of \(\mathfrak{I}_{\text{linear}}\) carrying that statistic with \(\mathfrak{I}_{/\!\sim}\), and the first clause of axiom~(I\ref{axiom:linear-constraint-independance}) is exactly the hypothesis of Theorem~\ref{theo:functoriality-distrib-info} for the reweighted operator \(T^{f}\), namely that it follows (I\ref{axiom:isolated-system})--(I\ref{axiom:upper-scale-conservation}).
    Verbatim as in that theorem, \([T]_{f}\) is well-defined and \(F_{\mathcal{T}}\) descends to a monoidal covariant functor on the transport morphisms, the strict monoidality of the linear quotient functors being inherited from \(\mathfrak{I}_{\text{linear}}\) and \(\mathfrak{M}_{\text{linear}}\); each step there is an equivalence, so this part holds if and only if the first clause of (I\ref{axiom:linear-constraint-independance}) does.

    \emph{Re-presentation.}
    What Theorem~\ref{theo:functoriality-distrib-info} does not cover is the re-presentation family, relating two presentations \([I]_{f_1}\) and \([I']_{f_2}\) of one underlying information \(f_1^{-1} \cdot I = f_2^{-1} \cdot I' \in\mathcal{I}^{f_1}/\!\sim\cap \mathcal{I}^{f_2}/\!\sim\).
    The operator \(\mathcal{T}_{\text{linear}}\) respects such a morphism precisely when the square
    \begin{center}
        \begin{tikzcd}[row sep=normal, column sep=huge]
            (\Omega,\mathcal{A},f_1^{-1} \cdot I) \arrow[rrr, leftrightarrow, bend left=12, "{=}"] \arrow[r, "{/\!\sim f_1}"] \arrow[d, "{T_{(\Omega,\mathcal{A})}}"'] & (\Omega,\mathcal{A},[I]_{f_1})  \arrow[r, leftrightarrow, "{(t_{\mathrm{lin},f_1},t_{\mathrm{lin},f_2})}"] \arrow[d, "{[T_{(\Omega,\mathcal{A})}]_{f_1}}"'] & (\Omega,\mathcal{A},[I']_{f_2}) \arrow[d, "{[T_{(\Omega,\mathcal{A})}]_{f_2}}"] & (\Omega,\mathcal{A},f_2^{-1} \cdot I') \arrow[d, "{T_{(\Omega,\mathcal{A})}}"'] \arrow[l, "{/\!\sim f_2}"'] \\
            T_{(\Omega,\mathcal{A})}(f_1^{-1} \cdot I) \arrow[rrr, leftrightarrow, bend right=12, "{=}"] \arrow[r, "{/\!\sim f_1}"]  & {}[T_{(\Omega,\mathcal{A})}]_{f_1}([I]_{f_1})  \arrow[r, leftrightarrow, "{(\,t_{f_2}\circ t_{f_1}^{-1},\ t_{f_1}\circ t_{f_2}^{-1}\,)}"] &  [T_{(\Omega,\mathcal{A})}]_{f_2}([I']_{f_2}) &  T_{(\Omega,\mathcal{A})}(f_2^{-1} \cdot I') \arrow[l, "{/\!\sim f_2}"']  \,,
        \end{tikzcd}
    \end{center}
    commutes, that is, when the image \(t_{f_2}\circ t_{f_1}^{-1}\) under \(F_{\mathcal{T}}\) of the re-presentation isomorphism \((t_{\mathrm{lin},f_1},t_{\mathrm{lin},f_2})\) matches the two quotient operators,
    \begin{align*}
        [T_{(\Omega,\mathcal{A})}]_{f_1}(I) = [T_{(\Omega,\mathcal{A})}]_{f_2}(I) \,,\quad \forall\, I\in\mathcal{I}^{f_1}/\!\sim\cap \mathcal{I}^{f_2}/\!\sim \,.
    \end{align*}
    Extended to \(n\) statistics this is exactly equation~\eqref{item:linear-constraint-coherence} of axiom~(I\ref{axiom:linear-constraint-independance}), the agreement of \([T]_{f_1},\dots,[T]_{f_n}\) on \(\mathcal{I}^{f_1}/\!\sim\cap\cdots\cap \mathcal{I}^{f_n}/\!\sim\); the prior of the induced operator on \(\mathcal{F}_{f_1}\cap\cdots\cap\mathcal{F}_{f_n}\) is then the common minimiser~\eqref{item:linear-constraint-coherence-prior} shared by these agreeing operators.

    Combining the two families, \(\mathcal{T}_{\text{linear}}\) is well-defined and induces a monoidal covariant functor from \(\mathfrak{I}_{\text{linear}}\) to \(\mathfrak{M}_{\text{linear}}\) if and only if \(\mathcal{T}\) satisfies axiom~(I\ref{axiom:linear-constraint-independance}).
\end{proof}

Theorem~\ref{theo:functoriality-moment-info} enables us to characterise inference by Kullback-Leibler divergence as inference operators that preserve the existence of a quotient inference operator \(\mathcal{T}_{/\!\sim }\) after linear transformation.

\begin{corollary}
    Let \(\mathcal{T}\) be a continuous family of inference operators.
    Then, \(\mathcal{T}\) induces a monoidal covariant functor from \(\mathfrak{I}_{\text{linear}}\) to \(\mathfrak{M}_{\text{linear}}\) if and only if \(\mathcal{T}\) is the family of inference operators induced by the Kullback-Leibler divergence.
\end{corollary}
\begin{proof}
    By Theorems~\ref{theo:functoriality}, \ref{theo:functoriality-distrib-info} and \ref{theo:functoriality-moment-info} inducing a monoidal covariant functor from \(\mathfrak{I}_{\text{linear}}\) to \(\mathfrak{M}_{\text{linear}}\) is equivalent to satisfying axioms (I\ref{axiom:isolated-system})--(I\ref{axiom:linear-constraint-independance}).
    Theorem~\ref{theo:MaxEnt-by-kl-divergence-axioms} then identifies the induced axioms with the Kullback-Leibler divergence, and the converse is immediate.
\end{proof}

\section{Discussion}\label{sec:discussion}

The results of this paper admit a reading that goes beyond their proof.
This section distinguishes the free-prior and uniform-prior regimes the hierarchy contains, relates it to existing axiomatisations and to information geometry, and states the limits of its scope.

\subsection{Free prior versus uniform prior}\label{subsec:free-vs-uniform}
The reference measure \(Q\) is left free throughout the hierarchy, so the inferred operator updates a prior by minimum relative entropy rather than maximising an absolute entropy.
Requiring in addition invariance under relabelling of the components, that is functoriality over the \v{C}encov information category with relabelling \(\mathfrak{I}_r\) (Theorem~\ref{theo:functoriality}), forces \(Q\) to be the uniform measure and recovers maximum-entropy inference.
Imposing relabelling is therefore not a neutral reformulation but exactly the principle of insufficient reason \citep[Sec.~VIII]{aynesPriorProbabilities1968}: with no ground to distinguish the labels, all are assigned equal mass.
Minimum relative entropy against a free prior and maximum entropy against the uniform prior are often conflated in the literature, we locate their difference in a single morphism family (Table~\ref{tab:main-result-summary}).

\subsection{Relation to existing axiomatisations}\label{subsec:relation-axiomatisations}

Table~\ref{tab:axiom-correspondence} maps our consistency axioms onto the axiomatisations of \citet{shoreAxiomaticDerivationPrinciple1980} and \citet{csiszar_why_1991}, all stated on the same inference operator.
The question is which axiom isolates Kullback--Leibler.
Shore--Johnson reach it only through system independence on product spaces, and even then do not isolate it, the R\'enyi family remaining admissible \citep{uffinkCanMaximumEntropy1995}.
Csisz\'ar does isolate the \(I\)-divergence, through his \emph{statistical} and \emph{transitivity} postulates.
The statistical postulate is a purely numerical condition, whereas our axiom~(I\ref{axiom:lower-scale-conservation}) is a reformulation invariance and reads as an information-processing principle rather than a numerical convenience.
Our axiom~(I\ref{axiom:linear-constraint-independance}), given prior uniqueness, entails his transitivity.
In Csisz\'ar the \(I\)-divergence theorems are not a stronger axiomatic version of the \(\alpha\)-divergence family, whereas in our framework they are.
This lets us claim that the Kullback--Leibler divergence is not a different rationality from the other canonical divergences, but the one satisfying the most reformulation invariances.

The prior, finally, is a primitive of information processing for both comparators.
Our work derives it instead: reformulation invariance requires every inference operator to carry a reference measure, so updating against a prior is a necessity of the theory rather than a historical convention.
Its uniform choice in particular is forced by relabelling invariance (Section~\ref{subsec:free-vs-uniform}), so maximum entropy is not a neutral choice among priors but a definite logical commitment.

\begin{table}[htbp]
  \centering
  \caption{Correspondence between the consistency axioms imposed here on the inference operator \(\mathcal{T}\) and the two enumerable axiomatisations of the same operator, \citet{shoreAxiomaticDerivationPrinciple1980} and \citet{csiszar_why_1991}.
  Both comparators are finite and discrete. A dash (---) marks an axiom with no counterpart in the corresponding axiomatisation.}
  \label{tab:axiom-correspondence}
  \small
  \begin{tabularx}{\linewidth}{@{}>{\raggedright\arraybackslash}p{0.215\linewidth} >{\raggedright\arraybackslash}p{0.145\linewidth} >{\raggedright\arraybackslash}p{0.145\linewidth} >{\raggedright\arraybackslash}X@{}}
    \toprule
    Axiom on \(\mathcal{T}\) (this paper) & \cite{shoreAxiomaticDerivationPrinciple1980} & \cite{csiszar_why_1991} & Notes \\
    \midrule
    Well-definedness (expressive, continuous family)
      & Uniqueness (I), Invariance (II)
      & Regularity (consistency, distinctness, continuity)
      & Coordinate invariance (II) is automatic in the coordinate-free measure setting. \\
    \addlinespace
    Isolated System \mbox{(I\ref{axiom:isolated-system})}
      & Subset Independence (IV)
      & Locality
      & Csisz\'ar notes the Shore--Johnson form is stronger than his locality. \\
    \addlinespace
    Prior Consistent \mbox{(I\ref{axiom:prior})}
      & ---
      & ---
      & The notion of \emph{prior} is part of our axiomatic, while stated as a definition in the comparators. \\
    \addlinespace
    Coarse-grain Consistent \mbox{(I\ref{axiom:lower-scale-conservation})}
      & ---
      & Statistical
      & Both lack the \(\sigma\)-algebra primitive on which resolution is phrased, so Csisz\'ar patches it with a numerical axiom.  \\
    \addlinespace
    Fine-grain Consistent \mbox{(I\ref{axiom:upper-scale-conservation})}
      & ---
      & Scale invariance + Transitivity
      & Selects the \(\alpha\)-divergences, \(\alpha\)-family; the family Shore--Johnson cannot exclude. \\
    \addlinespace
    Linear Transformation Consistent \mbox{(I\ref{axiom:linear-constraint-independance})}
      & ---
      & Transitivity
      & Isolates Kullback--Leibler. Corollary of Thm~3 in Csisz\'ar. \\
    \addlinespace
    ---
      & System Independence (III)
      & Product consistency (Def.~7)
      & A separate product-space route to Kullback--Leibler, unrelated to reweighting: Shore--Johnson's route (their axioms still fail to isolate it, Uffink) and Csisz\'ar's second derivation (Thm~5). Requires spaces with a product structure. \\
    \bottomrule
  \end{tabularx}
\end{table}

\subsection{Relation to information geometry}\label{subsec:relation-info-geometry}
The recasting of inference as selection under a preorder on measures mirrors the invariance-forces-uniqueness pattern of \citet{cencov1982}.
Invariance under Markov kernels alone leaves a whole family of compatible metrics on statistical models, and the Fisher metric is singled out only once invariance under linear reparametrisation is added.
The same final step, our Linear Transformation Consistency (I\ref{axiom:linear-constraint-independance}), isolates Kullback--Leibler among the \(f\)-divergences, so the KL divergence plays for inference operators the role the Fisher metric plays for statistical metrics.

\subsection{Scope and limitations}\label{subsec:scope-limitations}
The lift to general measurable spaces through the ordering closure holds only on the free-prior branch: a general measurable space carries no canonical uniform reference, so the maximum-entropy specialisation remains finite.
The generality over the sample space is obtained within a dominated setting: every measure compared is absolutely continuous with respect to a fixed \(\sigma\)-finite reference, its density lying in \(L^1_+(\nu)\).
Measures with strictly smaller support are not excluded but recovered as closure limits of strictly positive densities (Lemma~\ref{lem:general-f-divergence-sup-of-discrete-f-divergence}), sitting at finite divergence whenever \(\lim_{t\to0}f(t)=0\), as for Kullback--Leibler.
What falls outside the framework is mass singular to the reference, carrying no density.

\section{Conclusions}\label{sec:conclusion}

This work recasts inference by divergence minimisation as the unique rational way of processing information that remains invariant under equivalent reformulations of a problem.
Rather than treating the classical divergences as numerical conveniences with desirable analytic properties, we identify them as the necessary representation of a logic of information processing: any inference method that respects the equivalence of problems carrying the same information must, up to representation, minimise such a divergence.

The broader impact lies in the minimality of the assumptions required to reach this conclusion.
By relying only on the basics of measure, order, and category theory, the development is accessible to the information theory, statistics, and engineering communities, and exposes precisely which structural assumption selects which divergence family.
A direct consequence is that the axiomatic applies on arbitrary \(\sigma\)-algebras, not only on finite or discrete ones, so that the same logical foundation governs inference on countable, continuous, and abstract measurable spaces alike.
In doing so, it provides a common foundation on which information geometry, maximum entropy methods, and variational inference can be discussed and compared.

Several directions follow naturally.
The framework invites extensions to non-commutative and quantum divergences, and to inference problems with structured constraints arising in coding, estimation, and learning.
More broadly, the present axiomatic suggests reading any principled inference procedure as the representation of an underlying logic of information.
Different procedures are then distinguished not by their loss functions but by the type of information their axioms commit to processing faithfully.

\section*{Acknowledgments}
The authors are indebted to Mykola Lukashchuk for insightful discussions on information geometry.

\section*{Funding}
This publication is part of the project ``ROBUST: Trustworthy AI-based Systems for Sustainable Growth'' with project number KICH3.LTP.20.006, which is (partly) financed by the Dutch Research Council (NWO), GN Hearing, and the Dutch Ministry of Economic Affairs and Climate Policy (EZK) under the program LTP KIC 2020-2023.

\bibliography{refsup}

\appendix

\section{Axiomatic Representation on Finite Spaces}\label{sec:axiomatic-finite-spaces}

\subsection{Divergence representation}\label{subsec:divergence-representation}

This appendix proves the divergence representation theorem (Theorem~\ref{th:ordering-representation-as-divergence}) on finite spaces, working with the preorder underlying a divergence rather than its value.

As we expect continuity of our inference operator regarding both the information set and the prior \(Q\), we must consider them jointly.
So, for notational convenience, we define preorders on pairs of measures \((Q,P)\) where \(Q\) is the prior and \(P\) the measure to be evaluated.
\(\preceq_{(\Omega,\mathcal{A},Q)}\) of Section~\ref{subsec:divergence-axioms-presentation} is then the preorder \(\preceq_{(\Omega,\mathcal{A})}\) restricted in its first argument to \(Q\).

\begin{definition}[Continuous Family]\label{def:continuous-family}
    A family \((\preceq_k)_{k\in K}\) of preorders on finite measurable spaces is \textit{continuous} if and only if 
    \begin{align*}
        \left\{  (Q^\prime,P^\prime) \;: \;  (Q^\prime,P^\prime) \bpreceq_{(\Omega,\mathcal{A})}  (Q,P) \right\} \text{ and } \left\{  (Q^\prime,P^\prime) \;: \;  (Q,P) \bpreceq_{(\Omega,\mathcal{A})}  (Q^\prime,P^\prime) \right\}
    \end{align*}
    are closed subsets of \(\mathbb{R}^{2n}\) under the evaluation map
    \(\ev\,\colon\, Q,P \mapsto ( Q(a), P(a))_{a\in\mathrm{At}(\mathcal{A})}\).
\end{definition}

\begin{definition}[Non-Trivial Preorder]\label{def:non-trivial-family}
    A preorder on finite measurable spaces is non-trivial if and only if either \(\card{\mathrm{At}(\mathcal{A})} \leq 2\), or \(\preceq\) is not constant when the measure is fixed on a non-trivial subset of \(\Omega\), meaning that
    \begin{align*}
        \forall \, A\subseteq\Omega, A\neq \Omega \,, \quad  \exists \,  P, P^\prime \in \mathcal{M}(\Omega,\mathcal{A}) \text{ such that }  P|_A =  P^\prime|_ A \text{ and }  P \not\sim  P^\prime \,.
    \end{align*}
\end{definition}

The non-triviality assumption ensures that no subset of the measurable space is order-irrelevant.

\begin{definition}[Measure-Consistent Family]\label{def:measure-consistent-family}
    A family \((\preceq_k)_{k\in K}\) of preorders on measurable spaces is measure-consistent if and only if for any \(\preceq_{(\Omega,\mathcal{A})}\) and \(\preceq^\prime_{(\Omega^\prime,\mathcal{A}^\prime)}\) in the family such that there exists a bijection \(t:\mathrm{At}(\mathcal{A})\mapsto\mathrm{At}(\mathcal{A}^\prime)\) with \(Q^\prime = t_\star(Q)\),
    then,
    \begin{align}\label{eq:isomorphic-measurable-space-equivalent-oder}
         (t_\star(Q), t_\star(P_1)) \preceq_{(\Omega^\prime,\mathcal{A}^\prime)}  (t_\star(Q), t_\star(P_2)) \iff (Q, P_1)  \preceq_{(\Omega,\mathcal{A})} (Q, P_2)\,,  \quad \forall \, P_1, P_2, Q \in\mathcal{M}(\Omega,\mathcal{A})  \,.
    \end{align}
\end{definition}

Let \(\mathcal{D} \defeq \left\{ \preceq_{(\Omega,\mathcal{A})} \;:\; \card{\Omega} < \infty  ,\, \mathcal{A}\subseteq \sigma(\Omega)   \right\}\) be a measure-consistent family of non-trivial continuous preorders.

\begin{lemma}\label{lem:d-divergence-axioms}
    Let \(\mathcal{D}\) follow (O\ref{axiom:independence-subspace}).
    Then there is a function \(d\,:\, (0,\infty)^2 \mapsto \mathbb{R}\) continuous and unique up to a common positive affine transformation \(d \mapsto \alpha d + \beta\) with a single \(\alpha>0\) shared across atoms, such that for all \(\preceq_{(\Omega,\sigma(\Omega), Q)}\in\mathcal{D}\),
    \begin{align}\label{eq:d-divergence-representation}
        D_{(\Omega,\mathcal{A}, Q)} (P) = \sum_{a\in \mathrm{At}(\mathcal{A})}d\left( Q(a), P(a)\right) \,,
    \end{align}
    induces \(\preceq_{(\Omega,\sigma(\Omega), Q)}\).
\end{lemma}

\begin{proof}
    Let \(\preceq_{(\Omega,\mathcal{A})}\in\mathcal{D}\) with \(n = \card{\mathrm{At}(\mathcal{A})}\geq 3\).
    We apply the additive (cardinal) representation theorem of \citet[third Th., p.~10]{debreuTopologicalMethodsCardinal1959}, whose hypotheses we discharge in the present setting as follows.
    Each factor of the product carries a pair \((Q(a),P(a))\in(0,\infty)^2\) indexed by an atom \(a\in\mathrm{At}(\mathcal{A})\), so every factor is a connected topological space.
    The continuity hypothesis of \(\mathcal{D}\) supplies the topological continuity of \(\preceq_{(\Omega,\mathcal{A})}\) required by the theorem.
    Debreu's independence assumption is exactly Axiom (O\ref{axiom:independence-subspace}).
    Debreu's essentialness assumption is our non-trivial preorder assumption.
    The theorem then yields a family \((d_a:(0,\infty)^2 \to \mathbb{R})_{a\in\mathrm{At}(\mathcal{A})}\) of continuous functions such that \(\sum_{a\in\mathrm{At}(\mathcal{A})} d_a(Q(a),P(a))\) induces \(\preceq_{(\Omega,\mathcal{A})}\), determined up to \(d_a \mapsto \alpha\, d_a + \beta_a\) with a \emph{single} multiplicative constant \(\alpha>0\) shared across all atoms.
    By measure-consistency of \(\mathcal{D}\) (equation~\eqref{eq:isomorphic-measurable-space-equivalent-oder}), \(d_a=d_{a^\prime}\) for all \(a,a^\prime\in \mathrm{At}(\mathcal{A})\), collapsing the family to a single function \(d^{(n)}\) unique up to a common positive affine transformation.

    All remaining cases reduce to \(n=3\) through Axiom (O\ref{axiom:independence-subspace}).
    For \(n>3\), let \(A\in\mathcal{A}\) gather three atoms and freeze the pair \((Q,P)\) on \(A^{\complement}\): Axiom (O\ref{axiom:independence-subspace}) identifies the frozen preorder with \(\preceq_{(A,\mathcal{A}|_A)}\), a three-atom member of \(\mathcal{D}\) by measure-consistency, so \(d^{(n)}\) and \(d^{(3)}\) provide two additive representations of it; Debreu's uniqueness gives \(d^{(n)} = \alpha\, d^{(3)} + \beta\), and we normalise \(d^{(n)} = d^{(3)}\), written \(d\).
    For \(n\leq 2\), Debreu's theorem does not apply (it requires three essential factors) and the same argument runs backward.
    Equation~\eqref{eq:d-divergence-representation} therefore holds for every member of \(\mathcal{D}\) with a single \(d\), unique up to a common positive affine transformation.
\end{proof}

\begin{theorem}[\(f\)-Divergence Axioms]\label{prop:f-divergence-axioms}
    Let \(\mathcal{D}\) follow (O\ref{axiom:independence-subspace})--(O\ref{axiom:upper-scale-consistency}).
    Then \(\mathcal{D}\) is an \(f\)-divergence, meaning that there is a continuous strictly convex function \(f\) on \(]0,\infty[\) unique up to a positive affine transformation such that
    \begin{align}\label{eq:f-divergence}
       D_{(\Omega,\mathcal{A}, Q)}( P) \defeq  \sum_{a \text{ atom of } \mathcal{A}}  P(a_i) f\left(\frac{ Q(a_i)}{ P(a_i)}\right) \,,
    \end{align}
    defines a divergence that induces the preordering \(\preceq_{(\Omega,\mathcal{A})}\) for all \(\preceq_{(\Omega,\mathcal{A})}\in\mathcal{D}\).
\end{theorem}

\begin{proof}
    By Lemma~\ref{lem:d-divergence-axioms}, there is a continuous function \(d\,:\, (0,\infty)^2 \mapsto \mathbb{R}\), unique up to a common positive affine transformation, such that equation \eqref{eq:d-divergence-representation} induces \(\mathcal{D}\).
    We now establish that \(\mathcal{D}\) is an \(f\)-divergence.
    By Lemma~\ref{lem:help-f-divergence-proof}, the function \(\widetilde{d}(q,p) \defeq d(q,p) - d(q,q) + q\,d(1,1)\) satisfies

    \begin{align}
        \widetilde{d}(u(p_1 + p_2), p_1 + p_2) &=   \widetilde{d}(u p_1  , p_1 ) + \widetilde{d}(u p_2, p_2) \,,\quad \forall u,p_1,p_2>0\,.
    \end{align}

    Substituting \(\widetilde{d}\) for \(d\) in equation~\eqref{eq:d-divergence-representation} shifts \(D_{(\Omega,\mathcal{A},Q)}\) by \(\sum_{a}\left(Q(a)\,d(1,1)-d(Q(a),Q(a))\right)\), a constant depending only on \(Q\) and \(\mathcal{A}\), so each preorder of \(\mathcal{D}\) is preserved; we therefore rename \(\widetilde{d}\) as \(d\).
    By defining \(h(u,p) \defeq d(u p , p)\),
    \begin{align*}
        h(u, p_1 + p_2) &=   h(u , p_1 ) + h(u , p_2)  \,, \quad \forall u,p_1,p_2>0\,.
    \end{align*}
 
    So, at \(u\) fixed, \(h(u,\,.\,)\) is linear and continuous, which implies that there exists \(f\) such that

    \begin{align}
        d(u p  , p ) = h(u,p) = f(u)p \,,\quad \forall u,p>0\,.
    \end{align}
    
    As \(f(u) = h(u,1) = d(u,1)\) and \(d\) is continuous, \(f\) is continuous on \(\mathbb{R}_+^\star\), unique up to a positive affine transformation, and
    \begin{align}\label{eq:pre-divergence-2}
       D_{(\Omega,\mathcal{A}, Q)}( P) \defeq \sum_{a \text{ atom of } \mathcal{A}}   P(a) f\left(\frac{ Q(a)}{ P(a)}\right) \,,
    \end{align}
    induces \(\preceq_{(\Omega,\mathcal{A}, Q)}\).    
    Finally, let us show that \(f\) is strictly convex.
    Let \(x_1,\cdots,x_k\) be strictly positive real numbers and \((p_1,\cdots,p_k)\) be a strictly positive probability vector, \(k\geq 1\).
    Let us consider the space  \(\Omega= \{e_1,\cdots,e_k\}\), the \(\sigma\)-algebra \(\mathcal{A} = \{\Omega,\emptyset\}\),  and the measure \( Q\), \( P\) defined by \( Q(e_i) = p_i x_i\), \( P(e_i) = p_i\).
    Then, due to the assumption \eqref{item:entropic-ordering}, we have
    \begin{align}\notag
        \sum_{i=1}^k p_i f\left(x_i\right) &= \sum_{i=1}^k  P(e_i) f\left(\frac{ Q(e_i)}{ P(e_i)}\right) = D_{(\Omega, \sigma(\Omega), Q)}( P) \\\label{eq:f-convexity-proof}
        &\geq D_{(\Omega, \sigma(\Omega), Q)} \left( P|_{\{\Omega,\emptyset\}}\wedge  Q^{\{\Omega,\emptyset\}}\right) =  P(\Omega)\, f\left( \frac{ Q(\Omega)}{ P(\Omega)}\right) = f \left(\sum_{i=1}^k p_i x_i \right) \,,
    \end{align}

    since \( P(\Omega) = \sum_{i=1}^k p_i = 1\) and \( Q(\Omega) = \sum_{i=1}^k p_i x_i\), with equality if and only if \(x_i = \frac{ Q(e_i)}{ P(e_i)} = \frac{ Q(\Omega)}{ P(\Omega)} = \sum_{j=1}^k p_j x_j  \) for all \(1\leq i\leq k\).
    So, \(f\) is strictly convex on \(\mathbb{R}_+^\star\).
\end{proof}

\medskip
\noindent\emph{Extension to reduced support.}
The construction determines \(f\) on \(]0,\infty[\) from strictly positive measures.
For a variable measure \(P\) with \(P(a)=0\) on some atoms, the corresponding term of the \(f\)-divergence~\eqref{eq:f-divergence} is read by its recession value \(0\cdot f(\infty)\defeq Q(a)\,f'(\infty)\), where \(f'(\infty)\defeq\lim_{u\to\infty} f(u)/u\), and is finite exactly when \(f'(\infty)<\infty\).
By continuity of the family (Definition~\ref{def:continuous-family}), this extension represents the preorder on the boundary; on a general space the placement of reduced-support densities is Lemma~\ref{lem:not-integrable-order}.

\begin{theorem}\label{prop:reyni-divergence-axioms}
    Let \(\mathcal{D}\) follow (O\ref{axiom:independence-subspace})--(O\ref{axiom:lower-scale-consistency}).
    Then, the family of preorders \(\mathcal{D}\) is induced by an \(\alpha\)-divergence, meaning that there exists \(\alpha \in \mathbb{R} \setminus [0,1]\) such that
    \begin{align}\label{eq:reyni-divergence}
        P \mapsto \sum_{a \in \mathrm{At}(\mathcal{A})}  P(a) \left( \frac{ Q(a) }{ P(a) } \right)^{\alpha} \,,
    \end{align}

    or   \(\alpha \in ]0,1[\), such that
    \begin{align}\label{eq:reyni-divergence-2}
        P \mapsto - \sum_{a \in \mathrm{At}(\mathcal{A})}  P(a) \left( \frac{ Q(a) }{ P(a) } \right)^{\alpha} \,, 
    \end{align}

    or

    \begin{align}\label{eq:KL-divergence}
        P \mapsto \sum_{a \in \mathrm{At}(\mathcal{A})} P(a) \ln\left( \frac{ P(a) }{ Q(a) } \right) \,,
    \end{align}

    is the unique additive divergence, up to strictly positive affine transformation, inducing the preordering \(\preceq_{(\Omega,\mathcal{A})}\).
\end{theorem}

\begin{proof}
    By application of Proposition~\ref{prop:f-divergence-axioms}, \(\mathcal{D}\) is representable by an \(f\)-divergence \(D\) defined in equation~\eqref{eq:f-divergence}.
    By equation~\eqref{item:lower-scale-independance}, for any \(p^{\star}>0\),
    both functionals \( P \mapsto D_{(\Omega, \sigma(\Omega), Q)}( p^{\star}\times  P)\) and \( P \mapsto D_{(\Omega,\sigma(\Omega), \frac{Q}{Q(\Omega)})}(  P)\) are \(f\)-divergences representing the same preorder on the set \(\mathcal{F}_{\boldsymbol{1}}\) on \((\Omega,\sigma(\Omega))\).
    As \(\mathcal{F}_{\boldsymbol{1}}\) is connected, the additive representation of the preorder is unique up to positive affine transformation, so there are two real-valued continuous functions \(a>0,b\) such that for any probability vector \((p_i)_{1\leq i \leq n}\), \((q_i)_{1\leq i \leq n}\), \(n\geq 2\) and any \(U=\frac{ Q(\Omega)}{ P(\Omega)}>0\):

    \begin{align}
        \sum_{i=1}^n   P(\Omega) p_i f\left( U \frac{q_i}{p_i} \right)  = \sum_{i=1}^n  P(\Omega) p_i f\left(\frac{ Q(\Omega)}{ P(\Omega)}\frac{q_i}{p_i}\right) = \sum_{i=1}^n   P(\Omega) p_i \left( a(U) f\left( \frac{q_i}{p_i} \right) + b(U)  \right)  \,.
    \end{align}

    So \(f\) satisfies the functional equation:

    \begin{align}
        f(xU) = a(U) f(x) + b(U)  \,, \quad \forall x,U>0 \,.
    \end{align}
     
    Lemma~\ref{lem:functional-equation-convexity} enables us to conclude.
\end{proof}

\begin{definition}[Family of Quotient Preorders]\label{def:quotient-preorder}
    For every distinct \(f_1,\cdots,f_n>0\), a family of preorders \(\mathcal{D}^{f_1,\cdots,f_n}/\!\sim\) on \(\mathcal{F}_{f_1}\cap\cdots\cap\mathcal{F}_{f_n}\) is said to be expressive if and only if for every \([Q]\in\mathcal{F}_{f_1}\cap\cdots\cap\mathcal{F}_{f_n}\), there is a unique quotient preorder \([\preceq_{(\Omega,\mathcal{A}, [Q])}]_{f_1,\cdots,f_n}\in\mathcal{D}^{f_1,\cdots,f_n}/\!\sim\) such that \(\argmin_{P\in\mathcal{F}_{f_1}\cap\cdots\cap\mathcal{F}_{f_n}}  [\preceq_{(\Omega,\mathcal{A},[Q])}]_{f_1,\cdots,f_n} = \{[Q]\}\).
\end{definition}

\begin{theorem}\label{theo:kl-divergence-axioms}
    Let \(\mathcal{D}\) follow (O\ref{axiom:independence-subspace})--(O\ref{axiom:upper-scale-consistency}), and (O\ref{axiom:linear-constraint-consistency}), and induce a unique expressive family of quotient preoders \(\mathcal{D}^{f_1,\cdots,f_n}/\!\sim\) on \(\mathcal{F}_{f_1}\cap\cdots\cap\mathcal{F}_{f_n}\) for all \((f_1,\cdots,f_n)>0\). 
    Then, the preorder family \(\mathcal{D}\) is induced by the KL divergence of equation~\eqref{eq:KL-divergence}.
\end{theorem}

\begin{proof}
    Let \(\preceq_{(\Omega,\mathcal{A}, [Q])}\in\mathcal{D}{/\!\sim f}\) on \(\mathcal{F}_f\) with \(f>0\) on \((\Omega,\mathcal{A})\). 
    Then for every distinct \(f_1,\cdots,f_n>0\), \(\preceq_{(\Omega,\mathcal{A}, [Q])}\) follows the same order as \(\preceq_{(\Omega,\mathcal{A}, [Q]_{f_1,\cdots,f_n})}\) on the intersection of the fibres \(\mathcal{F}_{f_1}\cap\cdots\cap\mathcal{F}_{f_n}\) with least element \([Q]_{f_1,\cdots,f_n}\).
    By uniqueness of the additive representation up to strictly positive affine transformation applied on the connected set \(\mathcal{F}_{f_1}\cap\cdots\cap \mathcal{F}_{n}\),
    and uniqueness of \(\preceq^{f_1,\cdots,f_n}_{(\Omega,\mathcal{A})}\) up to the prior \([Q_{f_1,\cdots,f_n}]\), we have for all \([Q]_{f_1,\cdots,f_n} \in\mathcal{F}_{f_1}\cap \cdots\cap \mathcal{F}_{f_n}\),

    \begin{align}
        &\sum_{a \in \mathrm{At}(\mathcal{A})} d(Q_{f_1,\cdots,f_n}(a),P(a)) = c_Q \sum_{a \in \mathrm{At}(\mathcal{A})}  d^{\alpha}\left( Q(a) , P(a)\right) + b_Q \,,\quad c_Q>0, b_Q\in\mathbb{R} \\
        &\quad \forall\, Q \text{ such that } [\argmin_{P \in \mathcal{F}_{f_1}\cap \cdots\cap \mathcal{F}_{f_n}}D^{\alpha}_Q(P)] = \{[Q_{f_1,\cdots,f_n}]\} \,,
    \end{align}

    so, 
    \begin{align}\label{eq:prior-identiy-moment-constraint}
        &D^{\alpha}_{\argmin_{P \in \mathcal{F}_{f_1}\cap \cdots\cap \mathcal{F}_{f_n}}D^{\alpha}_Q(P)}(P) = D^{\alpha}_{(\Omega,\mathcal{A},[Q]_{f_1,\cdots,f_n})}(P) =  c_Q D^{\alpha}_{\frac{Q}{\langle Q,f_i \rangle}}(P) + b_Q \\ &\quad \forall\, Q\in\mathcal{M}(\Omega,\mathcal{A}),P\in\mathcal{F}_{f_1}\cap \cdots\cap \mathcal{F}_{f_n}\,.
    \end{align}

    The only \(\alpha\)-divergence satisfying equation~\eqref{eq:prior-identiy-moment-constraint} for all \(Q\) is the KL-divergence.
    Indeed, by \cite[Th.~4.8]{borweinDualityRelationshipsEntropyLike1991} the solution of \(\argmin_{P \in \mathcal{F}_{f_1}\cap \cdots\cap \mathcal{F}_{f_n}}D^{\alpha}_Q(P)\) for \(\alpha \neq 1\) is of the form \(((\lambda_i f_i)^{\frac{\alpha}{\alpha-1}} ) \cdot Q\) with \(\lambda_i>0\) which does not lead to the affine relation of equation~\ref{eq:prior-identiy-moment-constraint}.
    For the KL-divergence (\(\alpha = 1\)), the solution is of the form \(\exp( \sum_{i=1}^n \lambda_i f_i(a))_{a\in\mathrm{At}(\mathcal{A})} \cdot Q\) with \(\lambda_i\in\mathbb{R}\), leading to 
    \begin{align}
        \KL{\argmin_{P \in \mathcal{F}_{f_1}\cap \cdots\cap \mathcal{F}_{f_n}}D^{\alpha}_Q(P)}{P} & = D^{\alpha}_{(\Omega,\mathcal{A},\exp(\sum_{i=1}^n \lambda_i f_i) \cdot Q)}(P) \\
        &  = \sum_{a \in \mathrm{At}(\mathcal{A})}  \ln\frac{P(a)}{e^{\lambda_i f_i(a)}Q(a)} \\
        &  = \sum_{a \in \mathrm{At}(\mathcal{A})}  \ln\frac{P(a)}{Q(a)} - \sum_{a \in \mathrm{At}(\mathcal{A})} \sum_{i=1}^n \lambda_i f_i(a) P(a) \\
        &  = \sum_{a \in \mathrm{At}(\mathcal{A})}  \ln\frac{P(a)}{Q(a)}  + \sum_{i=1}^n \lambda_i  \langle f_i, P \rangle \,, \\ \label{eq:prior-identiy-moment-constraint-KL}
        &  = \KL{Q}{P}  + \sum_{i=1}^n \lambda_i  \langle f_i, P \rangle \,,
    \end{align}
    as \(P\in\mathcal{F}_{f_1}\cap \cdots\cap \mathcal{F}_{f_n}\) the second term of \eqref{eq:prior-identiy-moment-constraint-KL} is constant, so the KL-divergence satisfies equation~\eqref{eq:prior-identiy-moment-constraint}, which concludes the proof.
\end{proof}

This section showed sufficient conditions on the preorder for the existence of specific families of divergences inducing them.
The fact that our conditions are also necessary is well known, see \cite{renyiMeasuresEntropyInformation1961},\cite{amariInformationGeometryIts2016}[Sec. 3.2 and 3.5].

\subsection{Inference representation}\label{subsec:inference-representation}

This appendix proves the inference representation theorem (Theorem~\ref{theo:inference-representation}) on finite spaces, working with the preorder underlying an inference operator rather than the divergence it minimises.

Let the family 
\(\widetilde{\mathcal{T}} \defeq \left\{ T_{(\Omega,\mathcal{A},Q)} \;:\; \card{\Omega} < \infty  ,\, \mathcal{A}\subseteq \sigma(\Omega)  ,\,Q\in\mathcal{M}(\Omega,\mathcal{A}) \right\}\)
be a family of inference operators on finite measurable spaces such that the associated family of preorders \(\{\preceq_T\}_{T\in\mathcal{T}}\) forms a continuous, measure-consistent 
family in the sense of Section~\ref{subsec:divergence-representation}.

As each preorder  \(\preceq_T\) is continuous, \cite[Th.~2]{debreuRepresentationPreferenceOrdering1983} ensures the existence of a real-valued functional \(D_T\) on 
$\mathcal{M}(\Omega,\mathcal A)$ that represents it. The usual formulations of inference by divergence minimisation:

\begin{align}
    T(I) = \argmin_{ P\in I} D_T( P) \,,\quad \forall I\in\mathcal{I}_{(\Omega,\mathcal A)} \,,
\end{align}
is then recovered.

\begin{lemma}[Inference by Additive Divergence]\label{lem:isolate-system-axioms}
    Let \(\mathcal{T}\) follow (I\ref{axiom:isolated-system}).
    Then there is a function \(d\,:\, (0,\infty)^2 \mapsto \mathbb{R}\)  continuous and unique up to a strictly positive affine transformation in its second variable such that for all \(T_{(\Omega,\mathcal{A}, Q)}\in\mathcal{T}\), \(\preceq_{T_{(\Omega,\mathcal{A}, Q)}}\) is representable by the divergence defined by equation~\eqref{eq:d-divergence-representation}.
\end{lemma}

\begin{proof}
    Let \(\Omega\) with \(\card{\Omega}\geq 2\), \( Q\in\mathcal{M}(\Omega,\mathcal{A})\), \(A\subseteq \Omega\). 
    Consider the inference operator \(T\) on \(\mathcal{M}(\Omega,\mathcal{A}, Q)\).
    Let \(I_{A} = \{ P \,\colon\,  P |_A =  P_{A,1} \text{ or }   P |_A =  P_{A,2} \} \)  with \( P_{A,1},  P_{A,2}\in \mathcal{M}(A,\mathcal{A})\) and \(I_{A^\complement} = \{ P \,\colon\,  P |_{A^\complement} =  P_{A^\complement}\} \)  with \( P_{A^\complement}\in \mathcal{M}(A^\complement,\mathcal{A}^\complement)\).
    Then by hypothesis \eqref{item:isolated-system-axiom}:

    \begin{align}
        \Min_{\preceq_T}\{ P_{A,1}\oplus  P_{A^\complement},  P_{A,2}\oplus P_{A^\complement}\} &= T(I_{A}\cap I_{A^\complement})\\
            &= T_{(A, Q|_A)}(I_A)   \oplus T_{(A^\complement, Q|_{A^\complement})}( I_{A^\complement})  \\
        &= \Min_{\preceq_{T_{A, Q|_A}}}\{ P_{A,1} ,  P_{A,2}\} \oplus  P_{A^\complement} \,.
    \end{align}

    So, the order \(\preceq_{T_{(\Omega, \mathcal{A})}}\) at \((Q|_{A^\complement}, P|_{A^\complement})\) fixed is independent of the fixed sub-measures  and is equal to the preorder \(\preceq_{T_{(A, \mathcal{A}|_A)}}\).
    The result follows from the application of  Lemma~\ref{lem:d-divergence-axioms}.
\end{proof}

\begin{definition}[Prior Indexed Family of Inference Operator]\label{def:prior-index}
    Suppose \(T\) returns a single distribution profile on the total-mass constraint \(\{P : P(\Omega)=1\}\). 
    The \emph{prior} of \(T\) is that profile,
    \begin{align}
        [Q_T] \defeq \left[ T\bigl(\{ P : P(\Omega) = 1 \}\bigr) \right] \,.
    \end{align}
    A family \(\mathcal{T}\) is \emph{indexed by its priors} when its members are labelled \(T_{(\Omega,\mathcal{A},Q)}\) where \([Q]\) is the prior.
\end{definition}

\begin{definition}[Expressive Family of Inference Operators]\label{def:expressive-family-inference}
    A family of inference operators \(\mathcal{T}\) is \emph{expressive} when, on each space \((\Omega,\mathcal{A})\), every profile \([Q]\in\mathcal{M}(\Omega,\mathcal{A})/\!\sim\) is the prior (Definition~\ref{def:prior-index}) of exactly one member \(T_{(\Omega,\mathcal{A},Q)}\) of \(\mathcal{T}\).
\end{definition}

\begin{theorem}[Inference by \(f\)-Divergence]\label{theo:MaxEnt-by-f-divergence-axioms}
    Let \(\mathcal{T}\) be expressive and follow (I\ref{axiom:isolated-system})--(I\ref{axiom:lower-scale-conservation}).
    Then, there is a continuous strictly convex function \(f\) on \(]0,\infty[\) such that for all \(T_{(\Omega,\mathcal{A}, Q)}\in\mathcal{T}\), \(\preceq_{T_{(\Omega,\mathcal{A}, Q)}}\) is induced by a divergence defined by equation~\eqref{eq:f-divergence} unique up to a strictly positive affine transformation.
\end{theorem}

\begin{proof}
    We first prove that the inference operators \((T_{(\Omega,\mathcal{A},Q)})_{\mathcal{A}\subseteq \sigma(\Omega)}\) have a common reference/prior measure.
    By application of axiom (I\ref{axiom:prior}) equation~\eqref{item:upper-scale-does-not-influence-lower-scale} to the \(\sigma\)-algebra \(\{\Omega,\emptyset\}\), we have for every constraint on the total mass a unique distribution \([P^\star] \in \mathcal{F}_{\boldsymbol{1}}\) such that
    \begin{subequations}\label{eq:ax-apppli-1}   
        \begin{align}
            T_{(\Omega,\sigma(\Omega),Q)}(I) &= T_{(\Omega,\sigma(\Omega),Q)}(I)|_{\{\Omega,\emptyset\}} \wedge [P^\star]  \,,\quad    \text{ by \eqref{item:upper-scale-does-not-influence-lower-scale},}\\
            T_{(\Omega,\sigma(\Omega),Q)}(I) &= T_{(\Omega,\{\Omega,\emptyset\},Q)}(I) \wedge [P^\star]  \,, \quad  \text{ by \eqref{item:lower-scale-conservation},}\quad  \forall\, I \in \mathcal{I}_{\{\Omega,\emptyset\}} \,.
        \end{align}
    \end{subequations} 

    For all \(\mathcal{A}\supseteq \{\Omega,\emptyset\}\), we can reproduce the same reasoning that in equation~\eqref{eq:ax-apppli-1}, and as for any \(\mathcal{A}' \subseteq \mathcal{A}\), \(I\in\mathcal{I}_{\mathcal{A}'}\) implies \(I\in\mathcal{I}_{\mathcal{A}}\), \([P^\star]\) is the same for all \(\sigma\)-algebra of \(\Omega\), giving

    \begin{align}\label{eq:lower-scale-conservation-rephrase}
        T_{(\Omega,\mathcal{A},Q)}(I) = T_{(\Omega,\mathcal{A}',Q)}(I) \wedge [P^{\star}]^{\mathcal{A}'} \,,\quad \forall I\in\mathcal{I}_{\mathcal{A}'},\, \mathcal{A}' \subseteq \mathcal{A} \subseteq \sigma(\Omega) \,.
    \end{align}
    By indexing convention of Definition~\ref{def:prior-index}, the reference profile \([\mathbb{P}^\star]\) derived above \emph{is} the prior, \([Q]\defeq[\mathbb{P}^\star]\).
    By Lemma~\ref{lem:isolate-system-axioms}, there is a function \(d\,:\, (0,\infty)^2 \mapsto \mathbb{R}\) continuous in its second variable such that the divergence defined by equation~\ref{eq:d-divergence-representation} induces \((\preceq_{T_{(\Omega,\mathcal{A}, Q)}})_{T_{(\Omega,\mathcal{A}, Q)}\in\mathcal{T}}\).

    Let \(\Omega\) with \(\card{\Omega}\geq 2\), \( Q\in\mathcal{M}(\Omega,\sigma(\Omega))\), \(\mathcal{A}\subseteq \sigma(\Omega)\).
    Let \(I_{p^\mathcal{A}} = \{ P\,:\,  P|_\mathcal{A} = p|_{\mathcal{A}} \} \in \mathcal{I}_\mathcal{A}\) with \(p|_{\mathcal{A}}\in\mathcal{M}(\Omega,\mathcal{A})\).
    We prove that \(\preceq_{T_{(\Omega,\sigma(\Omega), Q)}}\) follows (O\ref{axiom:reference-measure-consistency}), 
    
    \begin{align}
        \Min_{\preceq_{T_{(\Omega,\sigma(\Omega), Q)}}}\{ P\,:\,  P|_\mathcal{A} = p|_{\mathcal{A}} \} & = T_{(\Omega,\sigma(\Omega), Q)}(I_\mathcal{A})\\
        & = T_{(\Omega,\mathcal{A}, Q)}(I_\mathcal{A}) \wedge  Q^\mathcal{A}  \; \text{ by~\eqref{eq:lower-scale-conservation-rephrase}} \\
         &= \Min_{\preceq_{T_{(\Omega,\mathcal{A}, Q)}}}\{p|_{\mathcal{A}} \} \wedge  Q^\mathcal{A}\\
        & = p|_{\mathcal{A}} \wedge  Q^\mathcal{A}  \,. 
    \end{align}
    We prove that \(\preceq_{T_{(\Omega,\sigma(\Omega), Q)}}\) follows (O\ref{axiom:upper-scale-consistency}).
    Let \(I_{\mathcal{A}} = \{ P\,:\,  P|_\mathcal{A} = p_1 \text{ or }  P|_\mathcal{A} = p_2 \} \in \mathcal{I}^\mathcal{A}\) with \(p_1,p_2\in\mathcal{M}(\Omega,\mathcal{A})\),

    \begin{align}
        \Min_{\preceq_{T_{(\Omega,\sigma(\Omega), Q)}}}\{ p_1\wedge  Q^\mathcal{A} \,,\; p_2\wedge  Q^\mathcal{A}\} &= \Min_{\preceq_{T_{(\Omega,\sigma(\Omega), Q)}}}\{  P \,:\,  P|_\mathcal{A} = p_1 \text{ or }  P|_\mathcal{A} = p_2\} \quad \text{by \eqref{item:entropic-ordering}}\\
        & = T_{(\Omega,\sigma(\Omega), Q)}(I^\mathcal{A})\\
        & = T_{(\Omega,\mathcal{A}, Q)}(I^\mathcal{A}) \wedge  Q^\mathcal{A}  \; \text{ by~\eqref{eq:lower-scale-conservation-rephrase}} \\
         &= \Min_{\preceq_{T_{(\Omega,\mathcal{A}, Q)}}}\{p_1, p_2\} \wedge  Q^\mathcal{A}\,.
    \end{align}

    We conclude by Theorem~\ref{prop:f-divergence-axioms}.
\end{proof}

\begin{definition}[Expressive Inference Family on Linear Constraints]\label{def:expressive-family-linear-constraints-inference}
    For distinct \(f_1,\cdots,f_n>0\), a family of inference operators \(\mathcal{T}^{(f_1,\cdots,f_n)} /\!\sim \) on \(\mathcal{F}_{f_1}\cap\cdots\cap\mathcal{F}_{f_n}\) is said to be \emph{expressive} if for every \((\Omega,\mathcal{A})\), \([Q]\in\mathcal{F}_{f_1}\cap\cdots\cap\mathcal{F}_{f_n}\), there is a unique inference operator \([T_{(\Omega,\mathcal{A})}]\in\mathcal{T}^{(f_1,\cdots,f_n)} /\!\sim\) such that \([T_{(\Omega,\mathcal{A})}](\mathcal{F}_{f_1}\cap\cdots\cap\mathcal{F}_{f_n}) = \{[Q]\}\).
\end{definition}

\begin{theorem}\label{theo:MaxEnt-by-reyni-divergence-axioms}
    Let \(\mathcal{T}\) be an expressive family following (I\ref{axiom:isolated-system})--(I\ref{axiom:upper-scale-conservation}) inducing an expressive family \(\mathcal{T}^{\boldsymbol{1}} /\!\sim \).
    Then, there is an \(\alpha\)-divergence, defined by equation~\eqref{eq:reyni-divergence}, \eqref{eq:reyni-divergence-2} or \eqref{eq:KL-divergence}, inducing the order \(\preceq_{T_{(\Omega,\mathcal{A}, Q)}}\) for all \(T_{(\Omega,\mathcal{A}, Q)}\in\mathcal{T}\).
\end{theorem}

\begin{proof}
    On \((\Omega,\sigma(\Omega),Q)\), consider \(\mathcal{A}=\{\Omega,\emptyset\}\), \(p^\star>0\), \([P_{1}], [P_{2}] \in \mathcal{F}_{\boldsymbol{1}}\) and the constraint sets
    \(I|_{\mathcal{A}} = \left\{  P \in \mathcal{M}(\Omega,\sigma(\Omega)) \;:\;  P(\Omega) = p^\star \right\}\),  \(I_1 = \left\{  P \in \mathcal{M}(\Omega,\sigma(\Omega)) \;:\;  P(\Omega) = 1 \right\}\), and \newline \(I^\mathcal{A} = \left\{  P \in \mathcal{M}(\Omega,\sigma(\Omega)) \;:\;  P^\mathcal{A} = [P_{1}] \,\text{ or }\,  P^\mathcal{A} = [P_{2}] \right\}\).
    We prove that \(\preceq_{\mathcal{T}}\) follows (O\ref{axiom:lower-scale-consistency}):
    \begin{align}
        \left(\Min_{\preceq_{T_{(\Omega,\sigma(\Omega), Q)}}}\{ p^\star \wedge [P_{1}]  \,,\; p^\star \wedge [P_{2}]\}\right)^{\{\Omega,\emptyset\}} &= T_{(\Omega,\sigma(\Omega), Q)}\left(I^{\mathcal{A}} \cap I|_{\mathcal{A}}\right)^{\{\Omega,\emptyset\}}  \\
        &= T_{(\Omega,\sigma(\Omega), Q)}\left(I^{\mathcal{A}} \cap I_1\right)^{\{\Omega,\emptyset\}} \quad \text{by (I\ref{axiom:upper-scale-conservation})} \\
        &= \left(\Min_{\preceq_{T_{(\Omega,\sigma(\Omega),Q)}}}\{1\wedge [P_{1}], 1\wedge [P_{2}]\}\right)^{\{\Omega,\emptyset\}}  \,,
    \end{align}
    moreover, the relation between priors in equation~\eqref{item:lower-scale-independance} follows from expressivity assumption on \(\mathcal{T}^{\boldsymbol{1}} /\!\sim \): the prior inputs \(\{P:P(\Omega)=Q(\Omega)\}\) and \(\{P:P(\Omega)=1\}\) have the same class \(\mathcal{F}_{\boldsymbol{1}}\), and \([T_{(\Omega,\mathcal{A},Q)}(\{P:P(\Omega)=Q(\Omega)\})]=[T_{(\Omega,\mathcal{A},1\wedge[Q])}(\{P:P(\Omega)=1\})]=\{[Q]\}\), so \([T_{(\Omega,\mathcal{A},Q)}] = [T_{(\Omega,\mathcal{A},1\wedge[Q])}]\).
    We conclude by application of Theorems~\ref{prop:reyni-divergence-axioms} and \ref{theo:MaxEnt-by-f-divergence-axioms}.
\end{proof}

\begin{theorem}\label{theo:MaxEnt-by-kl-divergence-axioms}
    Let \(\mathcal{T}\) follow (I\ref{axiom:isolated-system})--(I\ref{axiom:linear-constraint-independance}) inducing an expressive family \(\mathcal{T}/\!\sim (f_1,\cdots,f_n)\) for every distinct \(f_1,\cdots,f_n>0\).
    Then, the Kullback-Leibler divergence of equation~\eqref{eq:KL-divergence} induces the order \(\preceq_{T_{(\Omega,\mathcal{A}, Q)}}\)  for all \(T_{(\Omega,\mathcal{A}, Q)} \in \mathcal{T}\).
\end{theorem}

\begin{proof}
    By axiom (I\ref{axiom:linear-constraint-independance}) the induced operators agree on \(\mathcal{I}^{f_1}/\!\sim\cap\cdots\cap\mathcal{I}^{f_n}/\!\sim\) (equation~\eqref{item:linear-constraint-coherence}), so the preorders they induce coincide there; this is equation~\eqref{item:linear-constraint-consistency} (O\ref{axiom:linear-constraint-consistency}) with the prior excluded.
    The relation between the priors is the common minimiser~\eqref{item:linear-constraint-coherence-prior} of these agreeing operators, established by the re-presentation square in the proof of Theorem~\ref{theo:functoriality-moment-info}.
    Hence \(\preceq_{\mathcal{T}}\) follows (O\ref{axiom:linear-constraint-consistency}), and Theorem~\ref{theo:kl-divergence-axioms} concludes.
\end{proof}

In this section we showed sufficient conditions to express an inference operator as a constrained optimisation of specific divergences.
We did not explicitly check that our conditions were necessary. Their necessity is a straightforward implication of the functional arithmetic properties of the divergences.

\subsection{Proof of Corollary~\ref{cor:bayes-identification}}\label{sec:Bayes}

For \(A\in\mathcal{A}\) with \(Q(A)>0\), the Bayesian conditioning constraint is the set of probability measures supported on \(A\),
\begin{align}\label{eq:Bayesian-constraint}
    B_A \defeq \left\{ P\in\mathcal{M}(\Omega,\mathcal{A}) \,:\, P(A) = 1 \,,\; P(\Omega) = 1 \right\} \,,
\end{align}
on which Bayesian inference returns the conditional probability \(Q(\,\cdot\mid A) = Q|_A/Q(A)\).

\begin{proof}[Proof of Corollary~\ref{cor:bayes-identification}]
\emph{Bayesian conditioning fixes the prior, and mass rescaling.}
Clause (i) at \(A=\Omega\) reads \(T_{(\Omega,\mathcal{A},Q)}(\{P : P(\Omega)=1\}) = \{Q/Q(\Omega)\}\): the inference from the bare probability constraint is the normalised prior.
Hence for \(c>0\) the constraints \(\{P : P(\Omega)=c\}\) and \(\{P : P(\Omega)=1\}\) lie in \(\mathcal{I}_{\{\Omega,\emptyset\}}\); Axiom (I\ref{axiom:prior}) equates their inferred distributions and \(T(I)\subseteq I\) fixes the mass, so
\begin{align}\label{eq:mass-rescaling}
    T_{(\Omega,\mathcal{A},Q)}\bigl(\{P : P(\Omega)=c\}\bigr) = \left\{ c\,\tfrac{Q}{Q(\Omega)} \right\}\,.
\end{align}

\emph{\((ii)\Rightarrow(i)\).}
An \(f\)-divergence minimisation against its prior is coarse-grain consistent (I\ref{axiom:lower-scale-conservation}) and returns its prior on each fibre, so the following applies.
\(B_A\in\mathcal{I}_{\mathcal{A}^\prime}\) for \(\mathcal{A}^\prime = \sigma(\{A,A^\complement\})\); by (I\ref{axiom:lower-scale-conservation}), \(T_{(\Omega,\mathcal{A},Q)}(B_A)|_{\mathcal{A}^\prime} = T_{(\Omega,\mathcal{A}^\prime,Q|_{\mathcal{A}^\prime})}(B_A)\).
\(B_A\) factors across \(\{A,A^\complement\}\) with mass \(1\) on \(A\) and \(0\) on \(A^\complement\), and (I\ref{axiom:isolated-system}) fixes that split.
The distribution inside \(A\) is then obtained by (I\ref{axiom:prior}) on \(T_{(A,\mathcal{A}|_A,Q|_A)}\), which returns its prior.
Hence \(T_{(\Omega,\mathcal{A},Q)}(B_A) = \{1\wedge [Q|_A] \oplus 0 \} = \{Q(\,\cdot\mid A)\}\).

\emph{\((i)\Rightarrow(ii)\).}
The derivation of equation~\eqref{eq:mass-rescaling} uses only clause (i) and (I\ref{axiom:prior}), so it applies verbatim on every space; in particular on each subspace \((A,\mathcal{A}|_A,Q|_A)\) it rescales the conditional \(Q(\,\cdot\mid A)\) recovered at full support to any mass \(c>0\).
Let \(\mathcal{A}\subseteq\sigma(\Omega)\), \(P^\star\in\mathcal{M}(\Omega,\mathcal{A})\), and consider the slice constraint \(I_{P^\star}\defeq\{P : P|_{\mathcal{A}} = P^\star\} = \bigcap_{a\in\mathrm{At}(\mathcal{A})}\{P : P(a)=P^\star(a)\}\in\mathcal{I}_{\mathcal{A}}\).
Axiom (I\ref{axiom:isolated-system}) factors the inference across the atoms of \(\mathcal{A}\); on each subspace \((a,\sigma(a),Q|_a)\), the Bayesian recovery hypothesis rescaled by equation~\eqref{eq:mass-rescaling} gives \(\{P^\star(a)\,Q(\,\cdot\mid a)\}\), so
\begin{align*}
    T_{(\Omega,\sigma(\Omega),Q)}(I_{P^\star}) = \Bigl\{ \bigoplus_{a\in\mathrm{At}(\mathcal{A})} P^\star(a)\, Q(\,\cdot\mid a) \Bigr\} = \{ P^\star \wedge Q^{\mathcal{A}} \}\,.
\end{align*}
The same factorisation on \((\Omega,\mathcal{A},Q|_{\mathcal{A}})\), where \(I_{P^\star}\) reads \(\{P^\star\}\), gives \(T_{(\Omega,\mathcal{A},Q|_{\mathcal{A}})}(I_{P^\star}) = \{P^\star\}\); equation~\eqref{item:lower-scale-conservation} thus holds on slice constraints.

A general \(I\in\mathcal{I}_{\mathcal{A}}\) is the disjoint union of its slices \(I_{P^\star}\), \(P^\star\in I|_{\mathcal{A}}\), so the slice minima just computed reduce both sides of equation~\eqref{item:lower-scale-conservation} to minimisations over \(I|_{\mathcal{A}}\): writing \(D\) and \(D_{\mathcal{A}}\) for the additive divergences of Lemma~\ref{lem:isolate-system-axioms} on \((\Omega,\sigma(\Omega),Q)\) and \((\Omega,\mathcal{A},Q|_{\mathcal{A}})\),
\begin{align*}
    T_{(\Omega,\sigma(\Omega),Q)}(I)|_{\mathcal{A}} = \argmin_{P^\star\in I|_{\mathcal{A}}} D(P^\star\wedge Q^{\mathcal{A}}) \,,\qquad
    T_{(\Omega,\mathcal{A},Q|_{\mathcal{A}})}(I) = \argmin_{P^\star\in I|_{\mathcal{A}}} D_{\mathcal{A}}(P^\star) \,,
\end{align*}
and it remains to show that the two functionals rank \(I|_{\mathcal{A}}\) identically.

The Bayesian hypothesis supplies these cross-slice comparisons; as Lemma~\ref{lem:isolate-system-axioms} grants no differentiability of \(d\), they are extracted through mass transfers between two atoms rather than through a first-order condition.
For \(A = a_1\cup a_2\) a union of two atoms of \(\mathcal{A}\), Bayesian recovery on \((\Omega,\mathcal{A},Q|_{\mathcal{A}})\), rescaled by equation~\eqref{eq:mass-rescaling}, states that \(c\,Q|_{\mathcal{A}}(\,\cdot\mid A)\) is the unique minimiser of \(P^\star\mapsto \sum_{a\subseteq A} d(Q(a),P^\star(a))\) over \(\{P^\star : P^\star(A^\complement)=0,\, P^\star(\Omega)=c\}\), for every prior \(Q\), every such \(A\), and every \(c>0\), the atoms outside \(A\) carrying fixed mass and contributing constants.
Writing \(\varphi_q \defeq d(q,\,\cdot\,)\) and letting the atom masses \(q_1 = Q(a_1)\), \(q_2 = Q(a_2)\) and the ratio \(t = c/Q(A)\) run over all values, this reads
\begin{align}\label{eq:two-atom-transfer}
    \{(tq_1,tq_2)\} = \argmin_{p_1+p_2 \,=\, t(q_1+q_2)} \;\varphi_{q_1}(p_1)+\varphi_{q_2}(p_2) \,,\quad \forall\, q_1,q_2,t>0 \,.
\end{align}
Taking \(q_1=q_2=q\) in equation~\eqref{eq:two-atom-transfer} gives \(2\,\varphi_q\bigl((p_1+p_2)/2\bigr) < \varphi_q(p_1)+\varphi_q(p_2)\) for all \(p_1\neq p_2\); being moreover continuous (Lemma~\ref{lem:isolate-system-axioms}), \(\varphi_q\) is strictly convex on \(]0,\infty[\), so its one-sided derivatives \(\varphi_q^\prime(p-)\leq\varphi_q^\prime(p+)\) exist everywhere, are nondecreasing in \(p\), and coincide outside a countable set, with no smoothness assumption.
For arbitrary \(q_1,q_2\), the map \(\varepsilon\mapsto \varphi_{q_1}(tq_1+\varepsilon)+\varphi_{q_2}(tq_2-\varepsilon)\) is convex and minimised at \(\varepsilon=0\) by equation~\eqref{eq:two-atom-transfer}, so \(\varphi_{q_2}^\prime\bigl((tq_2)-\bigr) \leq \varphi_{q_1}^\prime\bigl((tq_1)+\bigr)\); the masses being arbitrary, \(\sup_{q>0}\varphi_{q}^\prime\bigl((tq)-\bigr) \leq \inf_{q>0}\varphi_{q}^\prime\bigl((tq)+\bigr)\), so the finite quantity \(g(t) \defeq \sup_{q>0}\varphi_{q}^\prime\bigl((tq)-\bigr)\) satisfies
\begin{align*}
    \varphi_{q}^\prime\bigl((tq)-\bigr) \,\leq\, g(t) \,\leq\, \varphi_{q}^\prime\bigl((tq)+\bigr) \,,\quad \forall\, q,t>0 \,.
\end{align*}
The function \(g\) is strictly increasing: for \(s<t\) and any \(q>0\), \(g(s)\leq \varphi_q^\prime\bigl((sq)+\bigr) < \varphi_q^\prime\bigl((tq)-\bigr) \leq g(t)\), the middle inequality because a convex function whose one-sided derivatives agree at the endpoints of \([sq,tq]\) is affine there, contradicting strict convexity.
A convex function is the integral of either of its one-sided derivatives, and for fixed \(q\) the bounds \(\varphi_q^\prime\bigl((sq)-\bigr)\) and \(\varphi_q^\prime\bigl((sq)+\bigr)\) squeezing \(g(s)\) coincide outside a countable set of \(s\), so
\begin{align*}
    d(q,tq) = d(q,q) + \int_q^{tq} \varphi_q^\prime(u+)\d{u} = d(q,q) + q\int_1^t g(s)\d{s} \,,\quad \forall\, q,t>0\,.
\end{align*}
Setting \(f(t)\defeq \int_1^t g(s)\d{s}\), strictly convex as the integral of a strictly increasing function, and \(\psi(q)\defeq d(q,q)\),
\begin{align*}
    d(q,p) = q\,f(p/q) + \psi(q) \,,
\end{align*}
which is the \(f\)-divergence form of clause (ii).
The term \(\psi\) is constant in the variable measure, hence inert in the preorders, and
\begin{align*}
    D(P^\star\wedge Q^{\mathcal{A}}) = \sum_{a\in\mathrm{At}(\mathcal{A})} Q(a)\, f\bigl(P^\star(a)/Q(a)\bigr) + \sum_{\omega\in\Omega} \psi(Q(\omega)) = D_{\mathcal{A}}(P^\star) + \text{const} \,,
\end{align*}
so both minimisations select the same elements of \(I|_{\mathcal{A}}\), establishing in passing Coarse-grain Consistency (I\ref{axiom:lower-scale-conservation}).

\emph{Uniqueness.}
By Theorem~\ref{theo:MaxEnt-by-f-divergence-axioms}, the operators satisfying Axioms (I\ref{axiom:isolated-system})--(I\ref{axiom:lower-scale-conservation}) are exactly the \(f\)-divergence minimisations; with clause (i), these are the additive-divergence inferences compatible with Bayesian conditioning.
\end{proof}

\section{Ordering Closure}\label{sec:convergence-proof}

\subsection{Construction}\label{subsec:closure-construction}
This section develops Definition~\ref{def:closure-preordering}.

\begin{definition}[Coherent Preorders]\label{def:coherent-preordering}
    Let \(\preceq\) be a total preorder on a set \(S\) and \(\preceq^\prime\) a total preorder on a subset \(S^\prime\).
    \(\preceq\) is coherent with \(\preceq^\prime\) if and only if for all \(x_1,x_2\in S^\prime\cap S\),
    \begin{align}
        x_1 \preceq^\prime x_2 \iff x_1 \preceq x_2 \,.
    \end{align}
\end{definition}

\begin{lemma}[Closure Coherence and Uniqueness]\label{lem:closure-preordering-well-defined}
    Let \((E,\mathcal{E},\nu)\) be a \(\sigma\)-finite measured space and \(q\in L^1_+(\nu)\).
    Any two preorders on subsets of \(L^1_+(\nu)\) that respect the property of equation~\eqref{eq:closure-definition-equivalence} agree on every pair of densities they both compare, so the closure \(\closure{\preceq_{\mathscr{A}_{\nu},q}}\) of Definition~\ref{def:closure-preordering} is uniquely defined on its set of comparable elements.
    Moreover, two such preorders admit a common extension respecting equation~\eqref{eq:closure-definition-equivalence} and ordering the union of their sets of comparable elements.
\end{lemma}

\begin{proof}
    Let \(\preceq_1\) and \(\preceq_2\) be two preorders respecting the property of equation~\eqref{eq:closure-definition-equivalence}, ordering respectively \(S_1,S_2\subseteq L^1_+(\nu)\).
    Let \(p,p^\prime\) be two \(\preceq_1\) and \(\preceq_2\) comparable elements.
    There are two associated sequences \((\mathcal{A}_{p,1,k})_{k\in \mathbb{N}}\) and \((\mathcal{A}_{p,2,k})_{k\in \mathbb{N}}\) in \(\mathscr{A}_{\nu}\) associated to \(\preceq_1\) and \(\preceq_2\) respectively as in Definition~\ref{def:closure-preordering}, and likewise for \(p^\prime\).
    The sequence \((\mathcal{B}_k)_{k\in \mathbb{N}}\) defined by \(\mathcal{B}_k = \sigma(\mathcal{A}_{p,1,k},\mathcal{A}_{p,2,k},\mathcal{A}_{p^\prime,1,k},\mathcal{A}_{p^\prime,2,k})\) is a sequence in \(\mathscr{A}_{\nu}\), each \(\mathcal{B}_k\) being generated by finitely many finite \(\sigma\)-algebras, and it respects the comparison condition of both \(\preceq_1\) and \(\preceq_2\).
    Thus, \(\preceq_1\) and \(\preceq_2\) agree on the comparison of \(p\) and \(p^\prime\), and are coherent as in Definition~\ref{def:coherent-preordering}.
    Moreover, by completing \(\preceq_1\) with the set of comparable elements of \(\preceq_2\) and vice versa, taking for each density the joined sequence above as its witness sequence, we define a preorder \(\preceq_{1,2}\) that respects the property of equation~\eqref{eq:closure-definition-equivalence} and orders \(S_1\cup S_2\).
\end{proof}

\begin{lemma}\label{lem:discrete-f-divergence-converge-to-general-f-divergence}
    Let \(f\,:\, ]0,\infty[ \,\longmapsto \mathbb{R}\) be convex.
    Let \((E,\mathcal{E},\nu)\) be a \(\sigma\)-finite measured space.
    Let \(p,q\in L^1_+(\nu)\) such that \(p \, (f\circ \frac{q}{p}) \in L^1(\nu)\).
    There exists a sequence \((\mathcal{P}_k)_{k\in\mathbb{N}}\) of finite partitions of \(E\), where \(\mathcal{P}_k = \{ A_{k,i} \}_{1\leq i \leq n_k} \) with \(A_{k,i}\in\mathcal{E}\), such that \(\sigma(\mathcal{P}_k) \subseteq \sigma(\mathcal{P}_{k+1})\), \(\sigma(\mathcal{P}_k) \in \mathscr{A}_{\nu}\) and

    \begin{align}\label{eq:app-100}
        \sum_{1 \leq i\leq n_k}  f\left(\frac{ \left( \int_{A_{k,i} }q(x) \d{}\nu(x)\right)}{ \left( \int_{A_{k,i} }p(x) \d{}\nu(x)\right)}\right)  \left( \int_{A_{k,i} }p(x) \d{}\nu(x)\right) \xrightarrow[k\to\infty]{} \int_E p(x) f\left(\frac{q(x)}{p(x)}\right)\d{\nu}(x) \,.
    \end{align}
\end{lemma}

\begin{proof}
    Write \(P \defeq p\d{\nu}\), \(Q \defeq q\d{\nu}\) and \(r \defeq \frac{q}{p}\); \(P\) and \(Q\) are finite measures and \(r>0\) holds \(\nu\)-almost everywhere.
    \citet[Th.~15]{Liese_on_divergences_2006} states that, for two probability measures and a convex function \(g\) on \(]0,\infty[\), the discretised \(g\)-divergences along an increasing sequence of sub-\(\sigma\)-algebras whose union generates the underlying \(\sigma\)-algebra converge to the \(g\)-divergence on that \(\sigma\)-algebra.
    To reduce to probability measures, set \(\alpha \defeq \lVert p \rVert_{L^1(\nu)}\), \(\beta \defeq \lVert q \rVert_{L^1(\nu)}\), \(\hat{P} \defeq P/\alpha\), \(\hat{Q} \defeq Q/\beta\) and \(g \defeq f\big(\frac{\beta}{\alpha}\,\cdot\,\big)\), which is convex on \(]0,\infty[\); for every finite partition \(\mathcal{P}\) of \(E\) into cells of positive \(\nu\)-measure,

    \begin{align}\label{eq:discretised-divergence-scaling-reduction}
        \sum_{A \in \mathcal{P}}  f\left(\frac{Q(A)}{P(A)}\right)  P(A) = \alpha \sum_{A \in \mathcal{P}}  g\left(\frac{\hat{Q}(A)}{\hat{P}(A)}\right)  \hat{P}(A) \,, \quad \text{and} \quad \int_E p(x) f\left(\frac{q(x)}{p(x)}\right)\d{\nu}(x) = \alpha \int_E g\left(\frac{\d{\hat{Q}}}{\d{\hat{P}}}\right) \d{\hat{P}} \,.
    \end{align}

    For \(k\in\mathbb{N}\), define a finite partition of \(E\) by the dyadic level sets of \(r\),

    \begin{align}\label{eq:increasing-sequence-partition-definition}
        \mathcal{P}_{k} \defeq \left\{ r^{-1}\Big(\big[ j 2^{-k} \,,\, (j+1) 2^{-k} \big)\Big) \;:\; 0 \leq j \leq 4^{k}-1 \right\} \cup \left\{ r^{-1}\Big(\big[ 2^{k} \,,\, \infty \big)\Big) \right\} \,.
    \end{align}

    Each cell of \(\mathcal{P}_{k}\) is a finite union of cells of \(\mathcal{P}_{k+1}\), so the sequence \((\sigma(\mathcal{P}_{k}))_{k\in\mathbb{N}}\) is increasing,\newline and \(\sigma\big(\bigcup_{k\in\mathbb{N}} \sigma(\mathcal{P}_{k})\big)= \sigma(r)\) since the dyadic intervals generate the Borel \(\sigma\)-algebra of \(]0,\infty[\).
    Let \(N\) be the union over \(k\in\mathbb{N}\) of the \(\nu\)-null cells of the partitions \(\mathcal{P}_{k}\); as a countable union of \(\nu\)-null sets, \(\nu(N)=0\).
    Modifying \(r\) on \(N\) to a fixed value \(c>0\) in the support of \(r\pf P\) and discarding the empty cells changes neither side of \eqref{eq:app-100} nor the two properties above, and ensures that every remaining cell has positive \(\nu\)-measure, so that \(\sigma(\mathcal{P}_{k}) \in \mathscr{A}_{\nu}\).

    The density \(\frac{\d{\hat{Q}}}{\d{\hat{P}}} = \frac{\alpha}{\beta}\, r\) is \(\sigma(r)\)-measurable, so the \(g\)-divergence of \(\hat{Q}\) and \(\hat{P}\) computed on the measurable space \((E,\sigma(r))\) equals the right-hand side of \eqref{eq:discretised-divergence-scaling-reduction}.
    Applying \citet[Th.~15]{Liese_on_divergences_2006} to \(\hat{Q}\), \(\hat{P}\) and \(g\) on \((E,\sigma(r))\) along the sequence \((\sigma(\mathcal{P}_{k}))_{k\in\mathbb{N}}\), and multiplying by \(\alpha\), yields \eqref{eq:app-100}.
\end{proof}

\begin{lemma}[Closure Construction]\label{lem:general-f-divergence-sup-of-discrete-f-divergence}
    Let \(f\,:\, ]0,\infty[ \,\longmapsto \mathbb{R}\) be convex.
    Let \((E,\mathcal{E},\nu)\) be a \(\sigma\)-finite measured space.
    Let \(p,q\in L^1_+(\nu)\) such that \(p \, (f\circ \frac{q}{p}) \in L^1(\nu)\).
    Then,

    \begin{align}
        \int_E p(x) f\left(\frac{q(x)}{p(x)}\right)\d{\nu}(x) = \sup_{\mathcal{A}\in\mathscr{A}_{\nu}} D_{(E,\mathcal{A},q\d{\nu})}(p\d{\nu}) \,,
    \end{align}
    with \( D_{(E,\mathcal{A},q\d{\nu})}\) defined by equation~\eqref{eq:f-divergence-2}.
    Moreover, there is an increasing sequence \((\mathcal{A}_{p,k})_{k\in\mathbb{N}}\subseteq \mathscr{A}_{\nu}\) such that \(\mathcal{A}_{p,k} \subseteq \mathcal{B}_{k}\) implies \(\lim_{k\to\infty} D_{(E,\mathcal{A}_{p,k},q\d{\nu})}(p\d{\nu}) = \int_E p(x) f\left(\frac{q(x)}{p(x)}\right)\d{\nu}(x)\) for all \((\mathcal{B}_k)_{k\in\mathbb{N}}\subseteq \mathscr{A}_{\nu}\).
\end{lemma}
\begin{proof}
        Let \(\{A_i\}_{0\leq i \leq n}\) be a finite partition of \(E\) such that \(\sigma(\{A_i\}_{0\leq i \leq n})=\mathcal{A}\).
    By Jensen inequality applied on the weight function \(\frac{p(x) \d{}\nu(x)}{\int_{A_i} p(x) \d{}\nu(x)}\), the convex function \(f\) on the set \(A_i\):

    \begin{align}\label{eq:convexity-argument-for-increase}
        \int_{A_i}  f\left(\frac{q(x)}{p(x)}\right) p(x) \d{\nu}(x) \geq    f\left(\frac{ \left( \int_{A_i }q(x) \d{}\nu(x)\right)}{ \left( \int_{A_i }p(x) \d{}\nu(x)\right)}\right) \left(\int_{A_i } p(x) \d{}\nu(x)\right) \,,\quad \forall \,0\leq i \leq n\,,
    \end{align}
   
    which directly implies that

    \begin{align}
        \int_E p(x) f\left(\frac{q(x)}{p(x)}\right)\d{\nu}(x) \geq \sum_{0 \leq i\leq n}   f\left(\frac{ \left( \int_{A_i }q(x) \d{}\nu(x)\right)}{ \left( \int_{A_i }p(x) \d{}\nu(x)\right)}\right)  \left( \int_{A_i }p(x) \d{}\nu(x)\right) = D_{(E,\mathcal{A},q\d{\nu})}(p\d{\nu}) \,,\quad \forall \mathcal{A}\in\mathscr{A}_{\nu} \,.
    \end{align}

    So, 
    \begin{align}\label{eq:200}
        \sup_{\mathcal{A}\in\mathscr{A}_{\nu}} D_{(E,\mathcal{A},q\d{\nu})}(p\d{\nu}) \leq \int_E p(x) f\left(\frac{q(x)}{p(x)}\right)\d{\nu}(x) \,.
    \end{align}

    Using Lemma~\ref{lem:discrete-f-divergence-converge-to-general-f-divergence}, we build an increasing sequence \(\sigma( \mathcal{P}_k )\subsetneq \sigma( \mathcal{P}_{k+1} )\) of finite partitions \(\mathcal{P}_k= \{A_{k,i}\}_{0\leq i \leq n_k}\), \(A_{k,i}\in\mathcal{E}\)      such that

    \begin{align}\label{eq:201}\notag
        D_{(E,\mathcal{A}_{p,k},q\d{\nu})}(p\d{\nu}) &= \sum_{0 \leq i\leq n_k}  f\left(\frac{ \left( \int_{A_{k,i} }q(x) \d{}\nu(x)\right)}{ \left( \int_{A_{k,i} }p(x) \d{}\nu(x)\right)}\right)  \left( \int_{A_{k,i} }p(x) \d{}\nu(x)\right)\\& \xrightarrow[k\to\infty]{} \int_E p(x) f\left(\frac{q(x)}{p(x)}\right)\d{\nu}(x) \,.
    \end{align}

    From \eqref{eq:201} and the upper bound \eqref{eq:200}, we conclude that

    \begin{align}
        \int_E p(x) f\left(\frac{q(x)}{p(x)}\right)\d{\nu}(x) \leq \sup_{ \mathcal{A} \in \mathscr{A}_{\nu}} D_{(E,\mathcal{A},q\d{\nu})}(p\d{\nu}) \leq \int_E p(x) f\left(\frac{q(x)}{p(x)}\right)\d{\nu}(x) \,,
    \end{align}

    and that for all \((\mathcal{B}_k)_{k\in\mathbb{N}}\subseteq \mathscr{A}_{\nu}\) such that \(\sigma( \mathcal{P}_k )\subseteq \mathcal{B}_k\), we have 

    \begin{align}\label{eq:covergence-of-discretisation-to-general-f-divergence}
        D_{(E,\mathcal{B}_k,q\d{\nu})}(p\d{\nu})  \xrightarrow[k\to\infty]{} \int_E p(x) f\left(\frac{q(x)}{p(x)}\right)\d{\nu}(x) \,.
    \end{align}
\end{proof}

\begin{lemma}\label{lem:not-integrable-order}
    Let \((E,\mathcal{E},\nu)\) be a \(\sigma\)-finite measured space. 
    Let \(q\in L^1_+(\nu)\).
    Let \(f\) be a strictly convex function such that the associated \(f\)-divergence \(D\) induces the order \(\preceq_{(E,\mathcal{A},q\d{\nu})}\) for all \(\mathcal{A}\in\mathscr{A}_{\nu}\).
    Let \(\closure{\preceq_{\mathscr{A}_{\nu},q}}\) be the closure (Definition~\ref{def:closure-preordering}) of the family of preorders \(\preceq_{\mathscr{A}_{\nu},q}\) induced by the divergences \((D_{(E,\mathcal{A},q\d{\nu})})_{\mathcal{A}\in\mathscr{A}_{\nu}}\).
    Let \(p\in L^1_+(\nu)\) not in \eqref{eq:set-definition-f-divergence}.
    Then, for all \(p^\prime\in L^1_+(\nu)\) in \eqref{eq:set-definition-f-divergence}, \(p^\prime\;\closure{\prec_{\mathscr{A}_{\nu},q}}\; p\).
\end{lemma}

\begin{proof}
    Let \(c\) be a subgradient of \(f\) at \(1\) and \(\varphi(x) \defeq f(x) - f(1) - c(x-1) \geq 0\).
    Replacing \(f\) by \(\varphi\) shifts \(\int_E p\,f(\tfrac{q}{p})\d{\nu}\) and every \(D_{(E,\mathcal{A},q\d{\nu})}(p\d{\nu})\), \(\mathcal{A}\in\mathscr{A}_{\nu}\), by the same finite constant \(-f(1)\int_E p\,\d{\nu} - c\int_E (q-p)\d{\nu}\), so we may assume \(f=\varphi\geq 0\); \(p\) not in \eqref{eq:set-definition-f-divergence} then reads \(\int_E p\,\varphi(\tfrac{q}{p})\d{\nu} = \infty\), the negative part being always integrable.

    We first build \((\mathcal{A}_{p,k})_{k\in\mathbb{N}}\) along which the divergence of \(p\) blows up.
    Let \(E_M \defeq \{\tfrac{1}{M} \leq \tfrac{q}{p} \leq M\}\), so that \(\int_{E_M} p\,\varphi(\tfrac{q}{p})\d{\nu} \uparrow \infty\) by monotone convergence (\(p\) and \(q\) being \(\nu\)-a.e. positive, \(E=\bigcup_M E_M\) up to a \(\nu\)-null set).
    Partition \(E_M\) by the level sets \(\{r_l \leq \tfrac{q}{p} < r_{l+1}\}\) of a grid \(\tfrac{1}{M}=r_0<\cdots<r_m=M\) chosen so that \(\varphi\) oscillates by at most \(1\) on each \([r_l,r_{l+1}]\) (uniform continuity on \([\tfrac1M,M]\)); each non-null cell \(A\) satisfies \(\int_A q\,\d\nu / \int_A p\,\d\nu\in[r_l,r_{l+1}]\), so its discretised term \(\varphi\bigl(\tfrac{\int_A q\d\nu}{\int_A p\d\nu}\bigr)\int_A p\,\d\nu\) differs from \(\int_A p\,\varphi(\tfrac{q}{p})\d\nu\) by at most \(\int_A p\,\d\nu\).
    Adjoining \(E\setminus E_M\), whose term is nonnegative, gives \(\mathcal{A}_M\in\mathscr{A}_{\nu}\) with \(D_{(E,\mathcal{A}_M,q\d{\nu})}(p\d{\nu}) \geq \int_{E_M} p\,\varphi(\tfrac{q}{p})\d{\nu} - \int_E p\,\d\nu\).
    The joins \(\mathcal{A}_{p,k}\defeq \sigma(\mathcal{A}_1,\cdots,\mathcal{A}_k)\) form an increasing sequence, and refining increases the divergence (equation~\eqref{eq:convexity-argument-for-increase} applied cell-wise), so \(D_{(E,\mathcal{B}_k,q\d{\nu})}(p\d{\nu}) \to\infty\) for every \((\mathcal{B}_k)_{k\in\mathbb{N}}\subseteq\mathscr{A}_{\nu}\) with \(\mathcal{A}_{p,k}\subseteq\mathcal{B}_k\).

    Finally, along every \((\mathcal{B}_k)_{k\in\mathbb{N}}\) refining \((\mathcal{A}_{p,k})_{k\in\mathbb{N}}\) and the sequence of Lemma~\ref{lem:general-f-divergence-sup-of-discrete-f-divergence} for \(p^\prime\), \(D_{(E,\mathcal{B}_k,q\d{\nu})}(p^\prime\d{\nu}) \to D_q(p^\prime) <\infty\) while \(D_{(E,\mathcal{B}_k,q\d{\nu})}(p\d{\nu})\to\infty\), so \(p^\prime\;\closure{\prec_{\mathscr{A}_{\nu},q}}\;p\) by Definition~\ref{def:closure-preordering}.
\end{proof}

\subsection{Divergence on general measurable space}\label{subsec:general-divergence-axioms}

We extend the axiomatic of Section~\ref{subsec:divergence-representation} to a \(\sigma\)-finite space \((E,\mathcal{E},\nu)\).

\medskip
\noindent\textbf{General measurable spaces.}
On a $\sigma$‑finite space $(E,\mathcal{E},\nu)$ we identify measures absolutely
continuous w.r.t.\ $\nu$ by their densities $p \in L^1_+(\nu)$, writing
$P(dx)=p(x)\,d\nu(x)$. 

\medskip
\noindent\textbf{Finite coarse-grainings.}
    $\mathcal{A}_\nu$ denotes the family of finite \(\sigma\)-algebras consisting of $\nu$-positive atoms:
    \[
      \mathcal{A}_\nu=\Big\{\sigma(\{A_i\}_{i=0}^n):\ E=\bigsqcup_{i=0}^n A_i\in\mathcal{E},\ \nu(A_i)>0 ,\ n\in\mathbb{N}\Big\} \,.
    \]
    We will then ask \(\preceq_{(E,\mathcal{E},q\d{\nu})}\) to be coherent with \(\closure{\preceq_{\mathscr{A}_{\nu},q}}\) through the following axiom:

\begin{axiomO}[From Discrete to General]\label{axiom:order-closure-consistency}
    \begin{align}
        p \; \closure{\preceq_{\mathscr{A}_{\nu},q}}\; p^\prime  \implies p \; \preceq_{(E,\mathcal{E},q\d\nu)} p^\prime \,,\quad \forall p,p^\prime \in L^1_+(\nu)
    \end{align}
\end{axiomO}

\begin{theorem}[General Divergence Ordering Axioms]\label{theo:general-f-divergence-axioms} 
    Let \((E,\mathcal{E},\nu)\) be a \(\sigma\)-finite space.
    Let 
    \begin{align}\label{eq:general-family-of-f-divergence}
        \mathcal{D} \defeq \left\{ \preceq_{(E,\mathcal{E},q\d{\nu})} \text{ on } \, L^1_+(\nu) \;:\;  q\in L^1_+(\nu) \right\}
    \end{align}

    be a measure-consistent preordering family such that for each \(q\in L^1_{+}(\nu)\), \(\preceq_{(E,\mathcal{E},q\d{\nu})}\) follows (O\ref{axiom:order-closure-consistency}) with \(\preceq_{\mathscr{A}_{\nu},q}\) a measure-consistent family of non-trivial continuous preorders
    that follows (O\ref{axiom:independence-subspace})--(O\ref{axiom:upper-scale-consistency}).

    Then \(\mathcal{D}\) is an \(f\)-divergence meaning that there is a strictly convex function \(f\) unique up to an affine transformation such that the functional 
    \begin{align}\label{eq:general-f-divergence}
        D_{q}(p) \defeq \int_{E} p(x) f\left(\frac{q(x)}{p(x)}\right) \d{\nu(x)} 
    \end{align}
    
    induces the preorder \(\preceq_{(E,\mathcal{E},q\d{\nu})}\) on 

    \begin{align}\label{eq:set-definition-f-divergence}
        \left\{p  \in L^1(\nu) \,:\, p\,(f\circ \frac{q}{p}) \in L^1(\nu)\right\}\,,
    \end{align} 

    for all \(q\in L^1_+(\nu)\). Moreover, for all \(p,p^\prime\in L^1_+(\nu)\) with \(p\) not in the set \eqref{eq:set-definition-f-divergence} and \(p^\prime\) in the set \eqref{eq:set-definition-f-divergence}: 
    \(p^\prime \prec_{(E,\mathcal{E},q\d{\nu})} p\,\).
    Moreover, if \(\preceq_{\mathscr{A}_{\nu},q}\) follows (O\ref{axiom:lower-scale-consistency}) or (O\ref{axiom:linear-constraint-consistency}), then \(f\) is respectively as in equations \eqref{eq:reyni-divergence}-\eqref{eq:KL-divergence}, meaning that \(D\) is an \(\alpha\)-divergence, or \(-\log\), meaning than \(D\) is the KL divergence.

\end{theorem}

\begin{proof}
    Let \(q\in L^1_+(\nu)\).  Let \(\mathcal{A}\in\mathscr{A}_{\nu}\). 
    By Proposition~\ref{prop:f-divergence-axioms} applied on the family \(\preceq_{\mathscr{A}_{\nu}}\),
    there is a function \(f\) strictly convex such that

    \begin{align}\label{eq:f-divergence-2}
       D_{(E,\mathcal{A},q\d{\nu})}(p\d{\nu}) \defeq  \sum_{a \in \mathrm{At}(\mathcal{A})} f\left(\frac{ \left( \int_{a}q(x) \d{}\nu(x)\right)}{ \left( \int_{a }p(x) \d{}\nu(x)\right)}\right)  \left( \int_{a }p(x) \d{}\nu(x)\right)  \,.
    \end{align}
    induces \(\preceq_{(E,\mathcal{A},q\d{\nu})}\).
    The convexity of \(f\) suffices for the convergence of the discretisation below.
    By Lemma~\ref{lem:discrete-f-divergence-converge-to-general-f-divergence} we know that any increasing sequence of \(\mathscr{A}_{\nu}\) bigger than the one established Lemma~\ref{lem:discrete-f-divergence-converge-to-general-f-divergence} leads to a correct definition of \(\closure{\preceq_{\mathscr{A}_{\nu},q}}\), on the subset \eqref{eq:set-definition-f-divergence} of \(L^1_+(\nu)\).
    So for \(p_1,p_2\) in the set \eqref{eq:set-definition-f-divergence}, \(D_q(p_1)<D_q(p_2)\) is equivalent to \(p_1 \closure{\prec_{\mathscr{A}_{\nu},q}} p_2\). Let \(p\in L^1_+(\nu)\) such that \(p \, (f\circ \frac{q}{p}) \in L^1(\nu)\), 
    If \(p\) not in the set \eqref{eq:set-definition-f-divergence}, the result follows from Lemma~\ref{lem:not-integrable-order}. 
    Proposition~\ref{prop:reyni-divergence-axioms} enables us to conclude when \(\preceq_{\mathscr{A}_{\nu},q}\) follows (O\ref{axiom:lower-scale-consistency}), and Theorem~\ref{theo:kl-divergence-axioms} when \(\preceq_{\mathscr{A}_{\nu},q}\) follows (O\ref{axiom:linear-constraint-consistency}).
\end{proof}

\begin{remark}
    No regularity assumption beyond the continuity of the preorder family is needed: the convergence of the discretisation to \eqref{eq:general-f-divergence} (Lemma~\ref{lem:discrete-f-divergence-converge-to-general-f-divergence}) requires only the convexity of \(f\), which the representation of Theorem~\ref{prop:f-divergence-axioms} already provides.
    \citet[Th.~13.1]{cencov1982} obtains the same limit behaviour by assuming \(\lim_{x\to\infty}\frac{f(x)}{|x|} = +\infty\), an assumption placed on the representative divergence rather than on the underlying preorder, which is exactly what our work is avoiding.
\end{remark}

In Proposition~\ref{theo:general-f-divergence-axioms}, we require the preordering family \eqref{eq:general-family-of-f-divergence} to be coherent with the closure \(\closure{\preceq_{\mathscr{A}_{\nu},q}}\) for reasons specified in Section~\ref{subsec:general-inference-axioms}.

\subsection{Inference on general measurable spaces}\label{subsec:general-inference-axioms}

To extend inference to a general measured space \((E,\mathcal{E},\nu)\), we define inference operators through their restriction to \(\mathcal{I}_{\mathscr{A}_{\nu}}\) and require two coherence properties with the discretisation of the space.
First, when the information \(I\) constrains the measure only on finitely many \(\nu\)-positive sets (\(I\in\mathcal{I}_{\mathcal{A}}\), \(\mathcal{A}\in\mathscr{A}_{\nu}\)), the operator should act as if \(E\) were discrete; this gives the coherence condition \eqref{item:general-discretisation-consistence}, the adaptation of the multi-scale equation \eqref{item:lower-scale-conservation} to general spaces.
Second, when \(I\) constrains infinitely small \(\nu\)-positive sets (an integral linear constraint) or infinitely many of them (an infinite summation constraint), the result should equal the limit of the operator along increasingly fine discretisations.
This limit \(\lim_{\mathcal{A}_k \nearrow{}} T_{(E,\mathcal{A}_k,q\d\nu)}(\{p_1,p_2\})\) depends on the chosen refining sequence \((\mathcal{A}_k)_{k\in\mathbb{N}} \subseteq \mathscr{A}_{\nu}\), so the closure of Definition~\ref{def:closure-preordering} is needed to make it independent of that choice, as in \eqref{item:limiting-refinement-consistency}.

\begin{axiomI}[Discretisation Coherency]\label{axiom:discretisation-coherency}
Let \(T_{(E,\mathcal{E},q\d{\nu})}\) be an inference on \((E,\mathcal{E},q\d{\nu})\). \((T_{(E,\mathcal{A},q\d{\nu})})_{\mathcal{A}\in\mathscr{A}_{\nu}}\) a measure-consistent family of inferences is said to be coherent with  \(T_{(E,\mathcal{E},q\d{\nu})}\) if and only if

\begin{subequations}
     \begin{align}\label{item:upper-sacle-constraint-does-not-affect-lower-scale}
        T_{(E,\mathcal{E},q\d{\nu})}(I)^{\mathcal{A}} &= T_{(E,\mathcal{E},q\d{\nu})}(I^\prime)^{\mathcal{A}} \,, \\\label{item:general-discretisation-consistence}
        T_{(E,\mathcal{E},q\d{\nu})}(I)|_{\mathcal{A}} &= T_{(E,\mathcal{A},q\d{\nu})}(I)  \,, \qquad\; \forall \,  I,I^\prime\in \mathcal{I}_\mathcal{A},\, \mathcal{A}\in \mathscr{A}_{\nu}  \,,\\\label{item:limiting-refinement-consistency}
        T_{(E,\mathcal{E},q\d{\nu})}(\{p_1,p_2\}) &= \lim_{\mathcal{A}_k \nearrow{}} T_{(E,\mathcal{A}_k,q\d\nu)}(\{p_1,p_2\}) \,, \quad\; \forall \, \mathcal{A}_k \nearrow{}  \supseteq \mathcal{A}_{p_i,k} \subseteq \mathscr{A}_{\nu}\,,p_i\in L_+^1(\nu) , i = 1,2 \,,
    \end{align}
\end{subequations}
    where \(\mathcal{A}_{p_1,k}\) and \(\mathcal{A}_{p_2,k}\) are defined as in Definition~\ref{def:closure-preordering}.
\end{axiomI}

\begin{theorem}\label{theo:inference-on-general-space}
    Let \((E,\mathcal{E},\nu)\) be a \(\sigma\)-finite measured space. 
    Let \(\mathcal{I}\) be a \(\sigma\)-algebra of \(L^1_+(\nu)\) such that \(\mathcal{I} \supseteq \mathcal{I}_{\mathscr{A}_{\nu}}\).
    Let

    \begin{align}\label{eq:general-family-of-f-inference-3}
        \mathcal{T} \defeq \left\{ T_{(E,\mathcal{E},q\d{\nu})} \text{ on } \mathcal{I} \,:\, q\in L^1_+(\nu)\right\}
    \end{align}
    be a continuous measure-consistent family of inferences that respects the \textit{isolated system} assumption~\eqref{item:isolated-system-axiom} for all \(A\in\mathcal{E}\) such that \(\nu(A)>0\), and follows (I\ref{axiom:discretisation-coherency}) for all \(q\in L^1_{+}(\nu)\).

    Then, \(\mathcal{T}\) is an inference by \(f\)-divergence minimisation, meaning that there exists a continuous strictly convex function \(f\) on \(]0,\infty[\) unique up to strictly positive affine transformation
    such that the functional defined at equation \eqref{eq:general-f-divergence} induces the ordering \(\preceq_{T_{(E,\mathcal{E},q\d{\nu})}}\)  restrained to the set \eqref{eq:set-definition-f-divergence} and any \(p\in L^1_+(\nu)\) not in the set \eqref{eq:set-definition-f-divergence} is \(\preceq_{T_{(E,\mathcal{E},q\d{\nu})}}\) bigger than the set \eqref{eq:set-definition-f-divergence}.
    Moreover if \(\mathcal{T}\) also follows (I\ref{axiom:upper-scale-conservation}) or (I\ref{axiom:linear-constraint-independance}) on \(\mathcal{I}_{\mathscr{A}_{\nu}}\), then \(D_q\) is respectively an \(\alpha\)-Divergence, or the Kullback Leibler Divergence.
\end{theorem}

\begin{proof}
    By equations~\eqref{item:upper-sacle-constraint-does-not-affect-lower-scale} and \eqref{item:general-discretisation-consistence} of hypothesis (I\ref{axiom:discretisation-coherency}) together with the \textit{isolated system} assumption, \(\mathcal{T}\) restrained to the discrete information \(\mathcal{I}_{\mathscr{A}_{\nu}}\) is a continuous measure-consistent family that follows the hypotheses of Theorem~\ref{theo:MaxEnt-by-f-divergence-axioms}, hence admits an \(f\)-divergence representation.
    By equation~\eqref{item:limiting-refinement-consistency} of hypothesis (I\ref{axiom:discretisation-coherency}), \(\preceq_{(E,\mathcal{E},q\d\nu)}\) equals \(\closure{\preceq_{\mathscr{A}_{\nu},q}}\) wherever both are defined, so Theorem~\ref{theo:general-f-divergence-axioms} concludes the first part of the proof.
    For the second part, \(\mathcal{T}\) restrained to \(\mathcal{I}_{\mathscr{A}_{\nu}}\) follows Theorem~\ref{theo:MaxEnt-by-reyni-divergence-axioms} when it satisfies (I\ref{axiom:upper-scale-conservation}), giving an \(\alpha\)-divergence, and Theorem~\ref{theo:MaxEnt-by-kl-divergence-axioms} when it satisfies (I\ref{axiom:linear-constraint-independance}), giving the Kullback--Leibler divergence.
\end{proof}

\section{Functional Equations}\label{app:functional-equtions}

\begin{lemma}\label{lem:help-f-divergence-proof}
    Let \(d\) be the generating function of an additive divergence (unique up to a positive affine transformation and continuous in its second variable) which induces a preorder that follows the assumptions of Proposition~\ref{prop:f-divergence-axioms}.
    Then, the function
    \begin{align*}
        \widetilde{d}(q,p) \defeq d(q,p) - d(q,q) + qd(1,1)\,,\quad \forall q,p > 0 \,,
    \end{align*}
    respects
    \begin{align}\label{eq:speudo-linear-relation}
        \widetilde{d}(u(p_1+p_2),p_1+p_2) = \widetilde{d}(u p_1,p_1) + \widetilde{d}(u p_2,p_2) \,,\quad \forall u, p_1, p_2 > 0 \,,
    \end{align}
    and \(\widetilde{d}(q,\cdot)\) induces the same ordering as \(d(q,\cdot)\) for all \(q > 0\).
\end{lemma}

\begin{proof}
    By application of hypothesis of equation~\eqref{item:scale-consistency} on two different spaces, and their union, we manage to prove the relation of equation~\eqref{eq:speudo-linear-relation}.
    Let \(q>0\), \(\Omega=\{e_1,e_2\}\), \(\mathcal{A} = \{\Omega,\emptyset\}\), \( P, Q\in\mathcal{M}(\Omega,\sigma(\Omega))\) such that 
    \begin{align}
        u &= \frac{ Q(e_1)}{ P(e_1)} = \frac{ Q(e_2)}{ P(e_2)} = \frac{ Q(\Omega)}{ P(\Omega)}  ,\, q = Q(\Omega) \,,\; p_i =  P(e_i)  \,, \; \forall i=1,2 \,.
    \end{align}

    The hypothesis \eqref{item:scale-consistency} implies

    \begin{align}
         P(\Omega) \mapsto D_{(\Omega,\mathcal{A}, Q)}( P) = d( Q(\Omega), P(\Omega)) \,,
    \end{align}
    induces the same order as

    \begin{align}
     p_1+p_2\mapsto D_{(\Omega,\sigma(\Omega), Q)}( P) =   d\left(\frac{ Q(\Omega)}{ P(\Omega)} p_1 , p_1 \right) + d\left(\frac{ Q(\Omega)}{ P(\Omega)} p_2 , p_2 \right)  \,,
    \end{align}
    
    and with the uniqueness up to a positive affine transformation of the divergence, there are \(a(t) >0\), \(b(t)\in  \mathbb{R}\) such that

    \begin{align}
        a(t) d(q, p)) + b(t) =  d(t q , t p) + d((1-t) q , (1-t) p)   \,,\quad \forall \,t\in]0,1[,\, p,q\in\mathbb{R}_+^\star \,.
    \end{align}

    We show that \(a(t)\) is unique by reproducing the same reasoning with \(q^\prime>0\), \(\Omega^\prime = \{e^\prime_1, \cdots ,e^\prime_n\}\), \(\mathcal{A}^\prime = \{\Omega^\prime,\emptyset\}\), \( P^\prime, Q^\prime\in\mathcal{M}(\Omega^\prime,\sigma(\Omega^\prime))\) such that

    \begin{align}\label{eq:propotionality-over-space-mass}
        \frac{ Q^\prime(e^\prime_1)}{ P^\prime(e^\prime_1)} = \cdots =  \frac{ Q^\prime(e^\prime_n)}{ P^\prime(e^\prime_n)} = \frac{ Q^\prime(\Omega^\prime)}{ P^\prime(\Omega^\prime)} >0 ,\, q = Q(\Omega) \,, p^\prime = P^\prime(\Omega),\, p_i =  P(e_i)  \,, \; \forall i=1,n\,,
    \end{align}

    we have
    \begin{align}
        a^\prime((p_i)) d(q^\prime, p^\prime) + b^\prime((p_i)) =  \sum_{i=1}^n d(p_i q^\prime , p_i p^\prime)   \,,\quad \forall \,p_i> 0, \sum_i p_i= 1 ,\,n\in\mathbb{N}^\star\,, p^\prime>0 \,.
    \end{align}

    By application of the hypothesis \eqref{item:scale-consistency} on \(\Omega\sqcup \Omega^\prime\), \(\mathcal{A}\oplus \mathcal{A}^\prime\), \( P\oplus  P^\prime\), \( Q\oplus  Q^\prime\),  we have

    \begin{align}
        a(t) d(q, p) + b(t) + a^\prime((p_i)) d(q^\prime, p^\prime) + b^\prime((p_i)) &=  d(t q , t p) + d((1-t) q , (1-t) p) + \sum_{i=1}^n d(p_i q^\prime , p_i p^\prime)   \,,
    \end{align}

    for all \(p,p^\prime>0 \) which by uniqueness of the divergence expression \eqref{eq:d-divergence-representation} up to an affine transformation implies that \(a^{\prime}((p_i))=a(t) \).
    Finally, let \(\Omega= \{e_1,e_2,e_3\}\), and \( P, Q\in\mathcal{M}(\Omega,\sigma(\Omega))\) following equation~\eqref{eq:propotionality-over-space-mass}, by reproducing the same reasoning but this time sequentially, first on the restriction \(\{e_1,e_2\}\), see \eqref{eq:app-1}-\eqref{eq:app-2} and then on \(\{e_1\sqcup e_2,e_3\}\), see \eqref{eq:app-3},
    we have \(b,b^\prime,b^{\prime\prime}\in  \mathbb{R}\) such that

    \begin{subequations}
    \begin{align}\label{eq:app-1}
        a d(q, p)) + b &=  d(p_1 q , p_1 p) + d(p_2 q , p_2 p) + d(p_3 q , p_3 p)  \\\label{eq:app-2}
        &= a d((p_1+p_2) q , (p_1+p_2) p) + b^\prime + d(p_3 q , p_3 p) \\\label{eq:app-3}
        &= (a-1) d((p_1+p_2) q , (p_1+p_2) p) + b^\prime + a d(q, p) + b^{\prime\prime}\,,\\
        \notag &\quad \forall \,p_i> 0, \sum_i p_i= 1 ,\, p>0 \,,
    \end{align}
    \end{subequations}

    which implies that \(a-1=0\).
    Now, we set

    \begin{align}
        \widetilde{d}(q,\,.\,) = d(q,\,.\,) - d(q,q) +  q d(1,1) \,, \quad \forall \, q>0 \,,
    \end{align}

    for \(q>0\) fixed and any probability vector \((p_i)_{1\leq i\leq n}\) with \(p_i> 0\), there is a \(b\in\mathbb{R}\) such that

    \begin{align}
         \widetilde{d}(q, p) + b =  \sum_{i=1}^n \widetilde{d}(p_i q , p_i p)    \,,\quad \forall  p>0 \,,
    \end{align}
    and for \(p=q\),

    \begin{align}
         \widetilde{d}(q, q) + b &=  \sum_{i=1}^n \widetilde{d}(p_i q , p_i q)   \\
        d(q,q) - d(q,q) + d(1,1)q + b &= \sum_{i=1}^n d(p_i q , p_i q) -  d(p_i q , p_i q) + p_i qd(1,1) \\
        d(1,1)q + b &=  \left(\sum_{i=1}^n p_i\right) d(1,1)q \\
        d(1,1)q + b &=   d(1,1)q \,,
    \end{align}

    so, \(b=0\) for all \(q>0\).
    We define \(h(\frac{q}{p},p) \defeq \widetilde{d}(q,p)\), for all \(q,p>0\). 
    \(h\) is then continuous in its second variable. 
    For \(P_1, P_2,u>0\), we have 

    \begin{align}
        \widetilde{d}(u(P_1 + P_2), P_1 + P_2) &=   \widetilde{d}(u P_1  , P_1 ) + \widetilde{d}(u P_2, P_2) \\
         h(u, P_1 + P_2) &=   h(u , P_1 ) + h(u , P_2)  \,.
    \end{align}

\end{proof}

\begin{lemma}\label{lem:functional-equation-convexity}
    Let \(f: \mathbb{R}^+ \to \mathbb{R}\) be a continuous and strictly convex function.
    Let \(a, b: \mathbb{R}^+ \to \mathbb{R}\) be continuous functions such that \(a > 0\).
    The functional equation
    \begin{equation}
        a(u) f(x) + b(u) = f(xu) \quad \text{for all } x, u > 0 \,,\label{eq:1}
    \end{equation}
    has for solution the class of functions
    \(f(x) = K x^k + C\) with \(C\in \mathbb{R}\), and \((k,K)\) as \(k \in (0, 1)\) and \(K < 0\), or  \(k > 1\) and \(K > 0\), or \(k < 0\) and \(K > 0\), and
    \(f(x) = K \log(x) + C\) with \(K < 0\) and \(C\in \mathbb{R}\).
\end{lemma}

\begin{proof}

    The proof is constructed using an analysis and synthesis approach.

    \textbf{Analysis}

    Taking \(x\) and \(u\) positive real numbers, since \(xu = ux\), we have \(f(xu) = f(ux)\), leading in equation \eqref{eq:1} to
    \begin{equation}
    a(u) f(x) + b(u) = a(x) f(u) + b(x) \,.\label{eq:2}
    \end{equation}

    Setting \(x=1\) in \eqref{eq:2} gives
    \begin{align}
        a(u) f(1) + b(u) = a(1) f(u) + b(1)
    \end{align}
    Let \(C_1 = f(1)\), \(C_2 = b(1)\), and \(C_3 = a(1)\). 
    We can then express \(b(u)\) by 

    \begin{equation}
    b(u) = C_3 f(u) - C_1 a(u) + C_2 \,.\label{eq:3}
    \end{equation}

    By substitution of \(b\) from  \eqref{eq:3} into the original equation \eqref{eq:1}, we obtain
    \begin{align}
    a(u) f(x) + (C_3 f(u) - C_1 a(u) + C_2) = f(xu) \,,
    \end{align}

    setting \(u=1\) gives

    \begin{align}
        a(1) f(x) + (C_3 f(1) - C_1 a(1) + C_2) = f(x) \,,
    \end{align}

    then substituting the constants \(C_1, C_2, C_3\) yields
    \begin{align}
        C_3 f(x) + (C_3 C_1 - C_1 C_3 + C_2) = f(x)\,,
    \end{align}
    therefore,
    \begin{align}
        C_3 f(x) + C_2 = f(x) \implies C_2 = (1 - C_3) f(x) \,.
    \end{align}

    Since \(C_2\) is constant, \((1 - C_3) f(x)\) must be constant. 
    Since \(f(x)\) is strictly convex, it cannot be a constant function. 
    Thus, we must have
    \begin{align}
        1 - C_3 = 0 \quad \text{and} \quad C_2 = 0
    \end{align}

    This yields the constraints

    \begin{align}
        a(1) = 1 \quad \text{ and } \quad b(1) = 0 \,.
    \end{align}

    Substituting in \eqref{eq:3}, we simplify \(b(u)\) by 

    \begin{align}
        b(u) = f(u) - C_1 a(u)\,.
    \end{align}

    Substituting this back into \eqref{eq:1} gives the equation
    \begin{align}a(u) f(x) + f(u) - C_1 a(u) = f(xu)\,.\end{align}
    Rearranging, we have
    \begin{align}\label{eq:5}
        f(xu) - f(u) = a(u) (f(x) - C_1)\,.
    \end{align}
    Define a new function \(g(x) = f(x) - C_1\).
    Since \(C_1=f(1)\), \(g(1)=0\), and \(g(x)\) is also strictly convex, then \eqref{eq:5} gives
    \begin{equation}
    g(xu) - g(u) = a(u) g(x) \,. \label{eq:4}
    \end{equation}

    From the symmetry in \eqref{eq:2}, we found

    \begin{align}
    a(x) g(u) + g(x) = g(xu)= a(u) g(x) + g(u) \,,
    \end{align}
    which can be rearranged to
    \begin{align}
        g(u) (1 - a(x)) = g(x) (1 - a(u))\,,
    \end{align}

    \medskip
    \underline{Case 1: \(a(u) \neq 1\) for all \(u\neq 1\)}

    Since \(a \neq 1\) on \(\mathbb{R}^+ \setminus \{1\}\), we can write

    \begin{align} \frac{g(x)}{1 - a(x)} = \frac{g(u)}{1 - a(u)}\,, \end{align}
    for all \(x, u \neq 1\).
    Both sides must equal a constant \(K\), which is non-zero: \(K = 0\) would give \(g \equiv 0\) (using \(g(1)=0\)), contradicting strict convexity.
    This leads to
    \begin{align} g(x) = K (1 - a(x)) \,. \end{align}
    Substituting this expression for \(g(x)\) into \eqref{eq:4} gives
    \begin{align}
        K(1 - a(xu)) - K(1 - a(u)) &= a(u) \cdot K(1 - a(x)) \\
        1 - a(xu) - 1 + a(u) &= a(u) - a(u)a(x)\\
        -a(xu) &= -a(u)a(x)\,,
    \end{align}

    which implies that \(a(xu) = a(u)a(x)\) for all \(x,u > 0\) such that \(a(x) \neq 1 \) and \(a(u)\neq 1\); since \(a(1) = 1\), the identity extends to all \(x, u > 0\).
    This is Cauchy's functional equation for multiplication.
    Since \(a\) is continuous and \(a(x)>0\), the solution is

    \begin{align}
        a(x) = x^k \quad \text{for some } k \in \mathbb{R}\,,
    \end{align}
    with \(k \notin \{0, 1\}\): \(k = 0\) would give \(a \equiv 1\), excluded in this case, and \(k = 1\) would make \(g\) affine, contradicting strict convexity.

    Substituting \(a(x) = x^k\) back into the expression for \(g(x)\) gives
    \begin{align}g(x) = K (1 - x^k)\,,\end{align}
    and thus, since \(f = g + C_1\), renaming the constants,
    \begin{align}f(x) = K x^k + C \,,\end{align}
    for some \(K \neq 0\) and \(C \in \mathbb{R}\).

    \medskip
    \underline{Case 2: \(a(u^\star) = 1\) for some \(u^\star \neq 1\)}

    Consider \(u^\star \neq 1\) such that \(a(u^\star) = 1\), we have from \eqref{eq:1},
    \begin{align}
        f(x) + b(u^\star) = f(xu^\star) \,,\quad  \forall x >0 \,,
    \end{align}

    and by symmetry of \eqref{eq:1} 

    \begin{align}
        f(x) + b(u^\star) = a(x)f(u^\star) + b(x) \,,\quad\forall x > 0 \,.
    \end{align}

    Inserting the deduced expression for \(b\) in \eqref{eq:1} gives
    \begin{align}
        a(u) f(x) + f(u) - f(u^\star) a(u) + b(u^\star) = f(xu) \,,\quad\forall u,x > 0 \,.
    \end{align}

    Setting \(x=1\) in the above equation gives 
    \begin{align}
        a(u)f(1) + f(u) - a(u)f(u^\star) + b(u^\star) &= f(u)\\
        a(u) (f(1) - f(u^\star)) + b(u^\star) &= 0 \,,
    \end{align}

    so either \(f(u^\star) = f(1)\) and \(b(u^\star) = 0\), or \(a\) is constant.
    In the first case, the same dichotomy applied to any other \(u' \neq 1\) with \(a(u') = 1\) gives \(f(u') = f(1)\) (unless \(a\) is constant, which is the second case); since a strictly convex function takes any given value at most twice, \(u^\star\) is the unique such point.
    The reasoning of \underline{Case 1} then applies for \(x, u \notin \{1, u^\star\}\) and gives \(g = K(1 - a)\) there; at the two remaining points both sides vanish, so the identity holds on all of \(\mathbb{R}^+\) and yields, as before, \(a(x) = x^k\).
    But \(a(u^\star) = (u^\star)^k = 1\) with \(u^\star \neq 1\) forces \(k = 0\), that is \(a \equiv 1\): the first case collapses into the second.

    In the second case, \(a\) is constant, and as \(a(1) = 1\), we have

    \begin{align}
        f(x) + b(u) &= f(u) + b(x) \,\forall x,u >0\,,\\
        f(x) - b(x) &= f(u) - b(u) \,\forall x,u >0\,.\\
    \end{align}

    So \(f-b\) is constant and we can write
    \begin{align}
        b(x) = f(x) + C \,.
    \end{align}

    Substituting back into \eqref{eq:1} gives

    \begin{align}
        f(xu) = f(x) + f(u) + C \,,
    \end{align}
    which is Cauchy's functional equation for addition.
    Since \(f\) is continuous, the solution is
    \begin{align}
        f(x) = L \log(x) + D \,,
    \end{align}
    for some constants \(L, D \in \mathbb{R}\).
    By strict convexity of \(f\) on \(\mathbb{R}^+\), \(L<0\).

    \medskip
    \textbf{Synthesis}
  
    Conversely, we verify that each candidate solves \eqref{eq:1} for some continuous \(a > 0\) and continuous \(b\), and determine the constants for which it is strictly convex.

    \medskip
    \underline{Case 1: \(f(x) = K x^k + C\)}

    Taking \(a(u) = u^k\) and \(b(u) = C(1 - u^k)\), both continuous with \(a > 0\), gives
    \begin{align}
        a(u) f(x) + b(u) = u^k \left(K x^k + C\right) + C \left(1 - u^k\right) = K (xu)^k + C = f(xu) \,,
    \end{align}
    so \eqref{eq:1} holds.
    For strict convexity, \(f''(x) > 0\):
    \begin{align}
        f'(x) &= K k x^{k-1} \,, \\
        f''(x) &= K k (k-1) x^{k-2} \,.
    \end{align}
    Since \(x^{k-2} > 0\), we require the constant factor to be positive, \(K k (k-1) > 0\), which holds exactly in the following cases: \(k \in (0, 1)\) and \(K < 0\), or  \(k > 1\) and \(K > 0\), or \(k < 0\) and \(K > 0\).

    \medskip
    \underline{Case 2: \(f(x) = K \log(x) + C\)}

    Taking \(a \equiv 1\) and \(b(u) = K \log(u)\) gives
    \begin{align}
        a(u) f(x) + b(u) = K \log(x) + C + K \log(u) = f(xu) \,,
    \end{align}
    so \eqref{eq:1} holds.
    For strict convexity, \(f''(x) = -K x^{-2} > 0\) requires \(K < 0\).
\end{proof}

\end{document}